\documentclass[11pt]{article}

\usepackage{amsfonts}
\usepackage{amssymb}
\usepackage{amsmath}
\usepackage{graphicx}
\usepackage{setspace}
\usepackage{amsthm}
\usepackage[labelfont=bf]{caption}
\usepackage[left=2.5cm,top=2.7cm,right=2.5cm,bottom=2.7cm]{geometry}
\theoremstyle{plain}
\newtheorem{thm}{Theorem}[section] 
\newtheorem{defn}[thm]{Definition} 
\newtheorem{exmp}[thm]{Example} 
\newtheorem{lem}[thm]{Lemma} 
\newtheorem{pro}[thm]{Proposition} 
\newtheorem{rem}[thm]{Remark} 
\newtheorem{co}[thm]{Corollary}
\newtheorem{ass}[thm]{Assumption}
\makeatletter
\newcommand{\vast}{\bBigg@{3.2}}
\newcommand{\Vast}{\bBigg@{5}}
\makeatother
\begin{document}
\begin{center}
\section*{Limit theorems for multivariate Brownian semistationary processes and feasible results}
\subsection*{Riccardo Passeggeri\footnote[1]{Department of Mathematics, Imperial College London, UK. Email: riccardo.passeggeri14@imperial.ac.uk} and Almut E.~D.~Veraart\footnote[2]{Department of Mathematics, Imperial College London, UK. Email: a.veraart@imperial.ac.uk}}
\today
\end{center}
\begin{abstract}
In this paper we introduce the \textit{multivariate} Brownian semistationary (BSS) processes and study the joint asymptotic behaviour of its realised covariation using in-fill asymptotics. First, we present a central limit theorem for general stationary multivariate Gaussian processes, which are not necessarily semimartingales. Then, we show weak laws of large numbers, central limit theorems and feasible results for BSS processes. An explicit example based on the so-called gamma kernels is also provided.
\\
\\
\textbf{Key words: }multivariate Brownian semistationary process; central limit theorem; law of large numbers; feasible; non-semimartingale; Wiener chaos; high frequency data; intermittency; gamma kernel.
\end{abstract}
\section{Introduction}
\indent The univariate Brownian semistationary (BSS) process is a stochastic process of the following form:
\begin{equation*}
Y_{t}=\mu+\int_{-\infty}^{t}g(t-s)\sigma_{s}dW_{s}+\int_{-\infty}^{t}q(t-s)a_{s}ds,
\end{equation*}
where $\mu$ is a constant, $W$ is a Brownian measure on $\mathbb{R}$, $g$ and $q$ are non-negative deterministic functions on $\mathbb{R}$, with $g(t)=q(t)=0$ for $t\leq 0$, and $\sigma$ and $a$ are c\`{a}dl\`{a}g processes. The name Brownian semistationary process comes from the fact that when $\sigma$ is stationary and independent of $W$, and $a$ is stationary then $Y$ is stationary. These processes were firstly introduced in \cite{BSS} and, since then, they have been extensively used in applications due to their flexibility and, thus, their capacity of modelling a variety of empirical phenomena. Two of the most notable fields of applications are turbulence and finance.\\
In the context of turbulence, where the process $\sigma$ represents the intermittency of the dynamics, these processes are able to reproduce the key stylized features of turbulence data, such as homogeneity, stationarity, skewness, isotropy and certain scaling laws (see \cite{Pakkanen,CHPP} and discussion therein). In finance, the BSS process has been applied to the modelling of energy spot prices (\cite{BBV,last}) and of logarithmic volatility of futures (\cite{logarithmicvolatilityPakkanen}), among others. Furthermore, fast and efficient simulation schemes for the univariate BSS are available (\cite{scheme}).

One of the key aspects of the BSS process that has been analysed in great detail in the last decade is the asymptotic behaviour of its realised power variation. The realised power variation of a process $Y_{t}$ is the sum of absolute powers of increments of a process, \textit{i.e.}
\begin{equation*}
\sum_{i=1}^{\lfloor nt \rfloor}\left|\Delta_{i}^{n}Y\right|^{r},\quad\text{where}\quad \Delta_{i}^{n}Y:=Y_{\frac{i}{n}}-Y_{\frac{i-1}{n}}.
\end{equation*}
For general semimartingales, the study of realised variation has a pivotal role in estimating the key aspects of the process under consideration, \textit{e.g}.~the integrated squared volatility given by $\int_{0}^{t}\sigma^{2}_{s}ds$ (see \cite{BCD} for further discussions).
This has led to the development of numerous works on this topic (see \cite{Jacod} and reference therein). On the other hand, the BSS process is not in general a semimartingale and the theory of realised power variation for semimartingales does not apply in this case. New results based on different mathematical tools, mainly the ones presented in the works of Peccati, Nourdin and coauthors (see \cite{NP} and reference therein), have been obtained. In \cite{BCD} the multipower variation for BSS processes is presented, while in \cite{Andrea1} Granelli and Veraart obtain the realised covariation for the bivariate BSS without drift. It is important to mention that in the general multivariate setting we have the work \cite{Econometrica} for the semimartingale case, but no work for the case of BSS processes outside the semimartingale framework. 

In this article we introduce the multivariate BSS process, study the joint asymptotic behaviour of its realised covariation, present feasible results and relevant examples. In particular, we will study the asymptotic behaviour of
\begin{equation*}
\sum_{i=1}^{\lfloor nt\rfloor}\left(\frac{\Delta_{i}^{n}Y^{(k)}}{\tau_{n}^{(k)}}\frac{\Delta_{i}^{n}Y^{(l)}}{\tau_{n}^{(l)}}\right)_{k,l=1,...,p},
\end{equation*}
where $p\in\mathbb{N}$, $\tau_{n}^{(j)}>0$ and $Y_{t}^{(j)}$ is the $j$-th component of the multivariate BSS process, for $j=1,...,p$. 
This work is motivated by the manifold applications of the BSS process and it is not just an extension to the multivariate case of the results presented in \cite{BCD} and \cite{Andrea1}. Indeed, in these previous works, the realised power (co)variation was always scaled by a scaling factor ($\tau_{n}$) restricted to a specific structure. We eliminate this restriction and this enables us to obtain all the feasible results presented in this work, which were not obtainable otherwise.\\
We remark that despite the more general theory developed here, no additional assumptions will be added other than the ones already introduced in \cite{BCD} and \cite{Andrea1} (but used in a multivariate setting).

Due to the various potential applications of the multivariate BSS process, it appears natural to derive feasible results, namely results that can be computed directly from real data. We focus on two objects:
\begin{equation*}
\left(\frac{\sum_{i=1}^{\lfloor nt\rfloor}\Delta^{n}_{i}Y^{(k)}\Delta^{n}_{i}Y^{(l)}}{\sqrt{\sum_{i=1}^{\lfloor nt\rfloor}\left(\Delta^{n}_{i}Y^{(k)}\right)^{2}}\sqrt{\sum_{i=1}^{\lfloor nt\rfloor}\left(\Delta^{n}_{i}Y^{(l)}\right)^{2}}}\right)_{k,l=1,...,p}\quad\text{and}\quad\left(\frac{\sum_{i=1}^{\lfloor nt\rfloor}\Delta^{n}_{i}Y^{(k)}\Delta^{n}_{i}Y^{(l)}}{\sum_{i=1}^{\lfloor nT\rfloor}\Delta^{n}_{i}Y^{(k)}\Delta^{n}_{i}Y^{(l)}}\right)_{k,l=1,...,p}.
\end{equation*} 
Both objects belong to the class of \textit{realised (co)variation ratio}. Similar ratios tailored to the univariate case have been used in the literature to construct a consistent estimator of key parameters, \textit{e.g}.~the smoothness parameter $\alpha$ of the BSS process in \cite{CHPP}. The second object can be defined as the \textit{relative covolatility} of the BSS process, since it represents the multivariate representation of the relative volatility concept introduced in \cite{Pakkanen}.

This paper is structured as follows. In Chapter \ref{preliminary}, we introduce the multivariate BSS process, the general setting and the basic mathematical concepts of this work. It is usually the case that when a univariate process is extended to the multivariate case there is more than one way to do it. Hence, we present two possible multivariate extensions of the one-dimensional BSS process. Further, in the same chapter we introduce the Gaussian core, which is a key object for the mathematical understanding and estimation of the BSS process. In Chapters \ref{JOINT Gaussian} and \ref{CLT-BSS-chapter}, we present the joint central limit theorem (CLT) for general multivariate stationary Gaussian processes and for the BSS processes, respectively. In particular, in these two chapters we are going to present different cases depending on which multivariate extension of the BSS process and values of the scaling factor $\tau_{n}$ are considered. In Chapters \ref{WEAK}, \ref{FEASIBLE} and \ref{EXAMPLE}, we prove the weak law of large numbers (WLLN), derive the feasible results and present an example of our process. In Chapter \ref{CONCLUSION}, we provide some final remarks and open questions.
\section{Preliminaries}\label{preliminary}
In this section we are going to explore the setting and some of the basic mathematical tools used throughout this article.
\\
Let $T>0$ denote a finite time horizon and let $(\Omega,\mathcal{F},(\mathcal{F}_{t}),\mathbb{P})$ be a filtered complete probability space. In the following we always assume that $p,n\in\mathbb{N}$ and that $\mathcal{B}(\mathbb{R})$ denotes the class of Borel sets of $\mathbb{R}$. We recall the definition of a Brownian measure.
\begin{defn}
An $\mathcal{F}_{t}$-adapted Brownian measure $W:\Omega\times\mathcal{B}(\mathbb{R})\rightarrow\mathbb{R}$ is a Gaussian stochastic measure such that, if $A\in\mathcal{B}(\mathbb{R})$ with $\mathbb{E}[(W(A))^{2}]<\infty$, then $W(A)\sim N(0,Leb(A))$, where $Leb$ is the Lebesgue measure. Moreover, if $A\subseteq[t,\infty)$, then $W(A)$ is independent of $\mathcal{F}_{t}$. 
\end{defn}
\noindent We will assume that $(\Omega,\mathcal{F},(\mathcal{F}_{t}),\mathbb{P})$ supports $p$ independent $\mathcal{F}_{t}$-Brownian measures on $\mathbb{R}$. Consider the stochastic process $\{\textbf{G}_{t}\}_{t\in[0,T]}$ defined as
\begin{equation*}
\textbf{G}_{t}:=\begin{pmatrix}
     G^{(1)}_{t}   \\
     \vdots\\
     G^{(p)}_{t} 
    \end{pmatrix}=\int_{-\infty}^{t}\begin{pmatrix}
   g^{(1,1)}(t-s) & \dots & g^{(1,p)}(t-s)  \\
   \vdots & \ddots & \vdots\\
   g^{(p,1)}(t-s) & \dots & g^{(p,p)}(t-s)
  \end{pmatrix} \begin{pmatrix}
     dW^{(1)}_{s}   \\
     \vdots\\
     dW^{(p)}_{s} 
    \end{pmatrix},
\end{equation*}
where the integral has to be considered componentwise, for $i,j=1,...,p$, $g^{(i,j)}\in L^{2}((0,\infty))$ are deterministic functions and continuous on $\mathbb{R}\setminus\{0\}$, and $(W^{(1)},...,W^{(p)})$ are jointly Gaussian $\mathcal{F}_{t}$-Brownian measures on $\mathbb{R}$. Thus, we have $G^{(i)}_{t}=\sum_{j=1}^{p}\int_{-\infty}^{t}g^{(i,j)}(t-s)dW^{(j)}_{s}$. We call the process $\{\textbf{G}_{t}\}_{t\in[0,T]}$ the multivariate Gaussian core, and it is possible to see that it is a stationary Gaussian process. The Gaussian core will play a crucial role in the limit theorems for the BSS process.
\begin{rem}
Notice that we do not assume independence of the Brownian measures. The only requirement is that they are jointly Gaussian so that the process $\{\textbf{G}_{t}\}_{t\in[0,T]}$ is Gaussian. This level of generality is needed to prove the central limit theorem (CLT) for the BSS process. In fact, as we will later see, proving a CLT for the Gaussian core driven by independent Brownian measures is not sufficient for proving the CLT for the BSS process.
\end{rem}
For $j\in\{1,...,p\}$ and $l\in\{1,...,n\}$, let $\tau_{n}^{(j)}$ be a (scaling) constant depending on $G^{(j)}$ and $n$ whose explicit form will be introduced later on, and let $\Delta^{n}_{l}G^{(j)}:=G^{(j)}_{\frac{l}{n}}-G^{(j)}_{\frac{l-1}{n}}$. Since $\{\textbf{G}_{t}\}_{t\in[0,T]}$ is a Gaussian process, we can use the machinery of Malliavin calculus. In particular, let $\mathcal{H}$ be the Hilbert space generated by the random variables given by:
\begin{equation}\label{generate}
\left(\dfrac{\Delta^{n}_{l}G^{(j)}}{\tau_{n}^{(j)}}\right)_{n\geq1,1\leq l\leq \lfloor nt\rfloor,j\in\{1,...,p\}}
\end{equation}
equipped with the scalar product $\langle\cdot,\cdot \rangle_{\mathcal{H}}$ induced by $L^{2}(\Omega,\mathcal{F},\mathbb{P})$, \textit{i.e}.~for $X,Y\in\mathcal{H}$ we have $\langle X,Y\rangle_{\mathcal{H}}=\mathbb{E}[XY]$. Notice that $\mathcal{H}$ is a closed subset of $L^{2}(\Omega,\mathcal{F},\mathbb{P})$ composed by $L^{2}$-Gaussian random variables generated by $(\ref{generate})$. In particular, we have an isonormal Gaussian process because $(\ref{generate})$ are jointly Gaussian random variables since they are (rescaled) increments of the Gaussian process $\{\textbf{G}_{t}\}_{t\in[0,T]}$. Following the setting of \cite{NP}, we assume that $\mathcal{F}$ is generated by $\mathcal{H}$. Finally, recall that any element of $L^{2}(\Omega,\mathcal{F},\mathbb{P})$ has a unique decomposition in terms of the Wiener chaos expansion of $\mathcal{H}$ (see \cite{NP}).

Next, we present and define the multivariate BSS process. Since there are several ways to generalise a univariate BSS process to a multivariate one, we will present two particularly relevant multivariate extensions.
\begin{defn}\label{BSS}
Consider $p$ Brownian measures $W^{(1)},...,W^{(p)}$. Further consider $p^{2}$ non-negative deterministic functions $g^{(1,1)},...,g^{(p,p)}\in L^{2}((0,\infty))$ which are continuous on $\mathbb{R}\setminus\{0\}$ and such that $g^{(i,j)}(t)=0$ for $t\leq0$ and for $i,j=1,...,p$. Let $\sigma^{(1,1)},...,\sigma^{(p,p)}$ be c\`{a}dl\`{a}g, $\mathcal{F}_{t}$-adapted stochastic processes and assume for all $t\in[0,T]$ and $i,j,k=1,...,p$ that $\int_{-\infty}^{t}(g^{(i,j)}(t-s)\sigma_{s}^{(j,k)})^{2}ds<\infty$. Let $\{\textbf{U}_{t}\}_{t\in[0,T]}=\{(U_{t}^{(1)},...,U_{t}^{(p)})\}_{t\in[0,T]}$ be a stochastic process in the nature of a drift term. Define,
\begin{equation*}
\textbf{Y}_{t}:=\begin{pmatrix}
     Y^{(1)}_{t}   \\\vdots\\
     Y^{(p)}_{t} 
    \end{pmatrix}=\int_{-\infty}^{t} \begin{pmatrix}
             g^{(1,1)}(t-s)  &\cdots  & g^{(1,p)}(t-s) \\ \vdots &
             \ddots  & \vdots \\ g^{(p,1)}(t-s) & \cdots & g^{(p,p)}(t-s)       
            \end{pmatrix} 
            \begin{pmatrix}
                         \sigma^{(1,1)}_{s}  &\cdots  & \sigma^{(1,p)}_{s} \\ \vdots &
                         \ddots  & \vdots \\ \sigma^{(p,1)}_{s} & \cdots & \sigma^{(p,p)}_{s}       
                        \end{pmatrix} 
  \begin{pmatrix}
     dW^{(1)}_{s}   \\\vdots\\
     dW^{(p)}_{s} 
    \end{pmatrix}+\begin{pmatrix}
         U^{(1)}_{t}   \\\vdots\\
         U^{(p)}_{t} 
        \end{pmatrix}
\end{equation*}
and
\begin{equation*}
\textbf{X}_{t}:=\begin{pmatrix}
     X^{(1)}_{t}   \\\vdots\\
     X^{(p)}_{t} 
    \end{pmatrix}=\int_{-\infty}^{t} \begin{pmatrix}
             g^{(1,1)}(t-s)\sigma^{(1,1)}_{s}  &\cdots  & g^{(1,p)}(t-s)\sigma^{(1,p)}_{s} \\ \vdots &
             \ddots  & \vdots \\ g^{(p,1)}(t-s)\sigma^{(p,1)}_{s} & \cdots & g^{(p,p)}(t-s)\sigma^{(p,p)}_{s}       
            \end{pmatrix} 
  \begin{pmatrix}
     dW^{(1)}_{s}   \\\vdots\\
     dW^{(p)}_{s} 
    \end{pmatrix}+\begin{pmatrix}
             U^{(1)}_{t}   \\\vdots\\
             U^{(p)}_{t} 
            \end{pmatrix}.
\end{equation*}
Then the vector valued processes $\{\textbf{Y}_{t}\}_{t\in[0,T]}$ and $\{\textbf{X}_{t}\}_{t\in[0,T]}$ are both called multivariate Brownian semistationary processes.
\end{defn}
\begin{rem}
	Note that we consider the case where the components are of dimension $p\times p$ or $p\times 1$. Extensions to non-square matrices are straightforward to obtain.
\end{rem}
Now, we will discuss properties of the tensor product in Hilbert spaces. Consider two real Hilbert spaces $\mathcal{H}_{1}$ and $\mathcal{H}_{2}$ endowed with the inner products $\langle\cdot,\cdot \rangle_{\mathcal{H}_{1}}$ and $\langle\cdot,\cdot \rangle_{\mathcal{H}_{2}}$ respectively. Given $f,x\in\mathcal{H}_{1}$ and $g,y\in\mathcal{H}_{2}$, we denote by $[f\otimes g](x,y):=\langle x,f \rangle_{\mathcal{H}_{1}}\langle y,g \rangle_{\mathcal{H}_{2}}$ the bilinear form $f\otimes g:\mathcal{H}_{1}\times \mathcal{H}_{2}\rightarrow\mathbb{R}$. Let $\mathcal{K}$ be the set of all finite linear combinations of such bilinear forms, namely $\mathcal{K}:=Span(f\otimes g:f\in\mathcal{H}_{1},g\in\mathcal{H}_{2})$. We are going to present a result on the inner product for this space.
\begin{lem}
The bilinear form $\langle\langle\cdot,\cdot \rangle \rangle$ on $\mathcal{K}$ defined by $\langle\langle f_{1}\otimes g_{1},f_{2}\otimes g_{2} \rangle \rangle:=\langle f_{1},f_{2} \rangle_{\mathcal{H}_{1}}\langle g_{1},g_{2} \rangle_{\mathcal{H}_{2}}$ is symmetric, well defined and positive definite, and thus defines a scalar product on $\mathcal{K}$.
\end{lem}
\begin{proof}
This is a well known result, for details see Reed and Simon's book \cite{ReedBook}.
\end{proof}
\noindent Observe that $\mathcal{K}$ endowed with $\langle\langle\cdot,\cdot \rangle \rangle$ is not complete. In the next three definitions we introduce the notion of a tensor product between Hilbert spaces, and the symmetrisation and contraction of a tensor product.
\begin{defn}
The tensor product of the Hilbert spaces $\mathcal{H}_{1}$ and $\mathcal{H}_{2}$ is the Hilbert space $\mathcal{H}_{1}\otimes \mathcal{H}_{2}$ defined to be the completion of $\mathcal{K}$ under the scalar product $\langle\langle\cdot,\cdot \rangle \rangle$. Further, we denote by $\mathcal{H}_{1}^{\otimes n}$ the $n$-fold tensor product between $\mathcal{H}_{1}$ and itself.
\end{defn} 
\begin{defn}
If $f\in\mathcal{H}^{\otimes n}$ is of the form $f=h_{1}\otimes \cdots h_{n}$ for $h_{1},...,h_{n}\in\mathcal{H}$, then the symmetrisation of $f$, denoted by $\tilde{f}$, is defined by $\tilde{f}:=\frac{1}{n!}\sum_{\sigma}h_{\sigma(1)}\otimes \cdots h_{\sigma(n)}$, where the sum is taken over all permutations of $\{1,...,n\}$. The closed subspace of $\mathcal{H}^{\otimes n}$ generated by the elements of the form $\tilde{f}$ is called the $n$-fold symmetric tensor product of $\mathcal{H}$, and is denoted by $\mathcal{H}^{\odot n}$.
\end{defn}
\begin{defn}
Let $g=g_{1}\otimes\cdots\otimes g_{n}\in\mathcal{H}^{\otimes n}$ and $h=h_{1}\otimes\cdots\otimes h_{n}\in\mathcal{H}^{\otimes m}$. For any $0\leq p\leq n\wedge m$, we define the $p$-th contraction of $g\otimes h$ as the following element of $\mathcal{H}^{\otimes m+n-p}$: $g\otimes_{p} h:=\langle g_{1},h_{1} \rangle_{\mathcal{H}}\cdots \langle g_{p},h_{p} \rangle_{\mathcal{H}} g_{p+1}\otimes\cdots\otimes g_{n}\otimes h_{p+1}\otimes\cdots\otimes h_{m}$. Note that, even if $g$ and $h$ are symmetric, their $p$-th contraction is not, in general, a symmetric tensor. We therefore denote by $g\tilde{\otimes}_{p}h$ its symmetrisation.
\end{defn}
Let us now move to the discussion of multiple integrals in the Malliavin calculus setting (see section 2.7 of \cite{NP}). We denote by $I_{p}:\mathcal{H}^{\odot p}\rightarrow\mathcal{W}_{p}$ the isometry from the symmetric tensor product $\mathcal{H}^{\odot p}$, equipped with the norm $\sqrt{p!}\|\cdot\|_{\mathcal{H}^{\otimes p}}$, onto the $p$-th Wiener chaos $\mathcal{W}_{p}$. In other words, the image of a $p$-th multiple integral lies in the $p$-th Wiener chaos. The first property that we are going to present is the \textit{isometry property of integrals}.
\begin{pro}\label{isometry property of integrals}
Fix integers $1\leq q\leq p$, as well as $f\in\mathcal{H}^{\odot p}$ and $g\in\mathcal{H}^{\odot q}$. We have
\begin{equation*}
\mathbb{E}[I_{p}(f)I_{q}(g)]=\begin{cases}
p!\langle f,g\rangle_{\mathcal{H}^{\otimes p}},\quad\text{if }p=q, \\
0,\quad\text{otherwise}.
\end{cases}
\end{equation*}
\end{pro}
\begin{proof}
See Proposition 2.7.5 in \cite{NP}.
\end{proof}
\noindent Moreover, we have the following product formula for multiple integrals.
\begin{thm}\label{Theorem 2.7.10}
Let $p,q\geq 1$. If $f\in\mathcal{H}^{\odot p}$ and $g\in\mathcal{H}^{\odot q}$, then
\begin{equation*}
I_{p}(f)I_{q}(g)=\sum_{r=0}^{p\wedge q}r!\binom{p}{r}\binom{q}{r}I_{p+q-2r}(f\tilde{\otimes}_{r}g).
\end{equation*}
\end{thm}
\begin{proof}
See Theorem 2.7.10 in \cite{NP}.
\end{proof}
\noindent Similarly to \cite{Andrea1} we apply the product formula for multiple integrals to conclude that for $i,j=1,...,p$
\begin{equation*}
\dfrac{\Delta^{n}_{l}G^{(i)}}{\tau_{n}^{(i)}}\dfrac{\Delta^{n}_{l}G^{(j)}}{\tau_{n}^{(j)}}=I_{1}\left(\dfrac{\Delta^{n}_{l}G^{(i)}}{\tau_{n}^{(i)}} \right)I_{1}\left(\dfrac{\Delta^{n}_{l}G^{(j)}}{\tau_{n}^{(j)}} \right)=\sum_{r=0}^{1}r!\binom{1}{r}\binom{1}{r}I_{2-2r}\left(\dfrac{\Delta^{n}_{l}G^{(i)}}{\tau_{n}^{(i)}}\tilde{\otimes}_{r}\dfrac{\Delta^{n}_{l}G^{(j)}}{\tau_{n}^{(j)}}\right)
\end{equation*}
\begin{equation*}
=I_{2}\left(\dfrac{\Delta^{n}_{l}G^{(i)}}{\tau_{n}^{(i)}}\tilde{\otimes}\dfrac{\Delta^{n}_{l}G^{(j)}}{\tau_{n}^{(j)}}\right)+\mathbb{E}\left[\dfrac{\Delta^{n}_{l}G^{(i)}}{\tau_{n}^{(i)}}\dfrac{\Delta^{n}_{l}G^{(j)}}{\tau_{n}^{(j)}} \right]
\end{equation*}
\begin{equation*}
\Rightarrow I_{2}\left(\dfrac{\Delta^{n}_{l}G^{(i)}}{\tau_{n}^{(i)}}\tilde{\otimes}\dfrac{\Delta^{n}_{l}G^{(j)}}{\tau_{n}^{(j)}}\right)=\dfrac{\Delta^{n}_{l}G^{(i)}}{\tau_{n}^{(i)}}\dfrac{\Delta^{n}_{l}G^{(j)}}{\tau_{n}^{(j)}}-\mathbb{E}\left[\dfrac{\Delta^{n}_{l}G^{(i)}}{\tau_{n}^{(i)}}\dfrac{\Delta^{n}_{l}G^{(j)}}{\tau_{n}^{(j)}} \right].
\end{equation*}

Now, we introduce the space $\mathcal{D}([0,T],\mathbb{R}^{n})$. This space is the set of all c\`{a}dl\`{a}g functions from $[0,T]$ to $\mathbb{R}^{n}$ and it is called the Skorokhod space. The norm in this space is defined as $\|f\|_{\mathcal{D}([0,T],\mathbb{R}^{n})}:=\sup\limits_{t\in[0,T]}\|f\|_{\mathbb{R}^{n}}$ where $f\in\mathcal{D}([0,T],\mathbb{R}^{n})$ and $\|\cdot\|_{\mathbb{R}^{n}}$ is any norm on $\mathbb{R}^{n}$ (it is a finite dimensional vector space, thus all the norms are equivalent). This metric works fine for $C([0,T],\mathbb{R}^{n})$ (the space of continuous functions from $[0,T]$ to $\mathbb{R}^{n}$), but it is stronger than the usual Skorokhod metric $J_{1}$ (or $M_{1}$). However, in this paper the functions to which our random elements (\textit{i.e}.~random variables and stochastic processes) convergence are continuous, and in this case these metrics are all equivalent.

We end this chapter with some results on stable convergence. We use the notations $\stackrel{u.c.p.}{\rightarrow}$, $\stackrel{P}{\rightarrow}$, $\stackrel{st}{\rightarrow}$, $\stackrel{d}{\rightarrow}$ for convergence in probability locally uniformly in time, convergence in probability, stable convergence, and convergence in distribution, respectively. In the case of the Skorokhod space with uniform metric, suppose $X_{n},X$ are $\mathcal{D}([0,T],\mathbb{R}^{d})$-valued stochastic processes defined on the same filtered probability space, then we have that $X_{n}\stackrel{P}{\rightarrow}X$ if and only if $X_{n}\stackrel{u.c.p.}{\rightarrow}X$, since they are both equal to $\lim\limits_{n\rightarrow\infty}\mathbb{P}(\sup\limits_{t\in[0,T]}\| X_{n,t}-X_{t}\|_{\mathbb{R}^{d}})$.
\begin{thm}[Continuous mapping theorem]\label{Cotinuous mapping}
Let $(S,m)$ a metric space and let $(S_{n},m)\subset (S,m)$ be arbitrary subsets and $g_{n}:(S_{n},m)\mapsto (E,\mu)$ be arbitrary maps $(n\geq 0)$ such that, for every sequence $x_{n}\in (S_{n},m)$, if $x_{n'}\rightarrow x$ along a subsequence and $x\in(S_{0},m)$ then $g_{n'}(x_{n'})\rightarrow g_{0}(x)$. Then, for arbitrary maps $X_{n}:\Omega_{n}\mapsto (S_{n},m)$ and every random element $X$ with values in $(S_{0},m)$ such that $g_{0}(X)$ is a random element in $(E,\mu)$: if $X_{n}\stackrel{d}{\rightarrow}X$, then $g_{n}(X_{n})\stackrel{d}{\rightarrow}g_{0}(X)$; if $X_{n}\stackrel{P}{\rightarrow}X$, then $g_{n}(X_{n})\stackrel{P}{\rightarrow}g_{0}(X)$; and if $X_{n}\stackrel{a.s.}{\rightarrow}X$, then $g_{n}(X_{n})\stackrel{a.s.}{\rightarrow}g_{0}(X)$.
\end{thm}
\begin{proof}
See Theorem 18.11 of \cite{V}.
\end{proof}
Notice that $(S,m)$ might be a function space like the Skorokhod space endowed with the uniform metric. For stable convergence we have the following theorem, see Theorem 1 in \cite{Aldo}.
\begin{thm} Let $X_{n}$ be random elements defined on the same probability space $(\Omega,\mathcal{F},\mathbb{P})$. Suppose that $X_{n}\stackrel{st}{\rightarrow}X$, that $\sigma$ is any fixed
$\mathcal{F}$-measurable random variable and that $g(x, y)$ is a continuous function of two variables. Then, $g(X_{n},\sigma)\stackrel{st}{\rightarrow}g(X,\sigma)$.
\end{thm}
\begin{proof}
	It follows from Theorem \ref{Cotinuous mapping} and the definition of stable convergence.
\end{proof}
\begin{pro}
Let $X_{n}, Y_{n}$ and $Y$ be random elements defined on the same probability space and assume that $Y_{n}\stackrel{P}{\rightarrow}Y$ and $X_{n}\stackrel{st}{\rightarrow}X$. Then, $(X_{n},Y_{n})\stackrel{st}{\rightarrow}(X,Y)$.
\end{pro}
\begin{proof}
	See Section 2 of \cite{Jac}.
\end{proof}
\section{Joint CLT for stationary multivariate Gaussian processes}\label{JOINT Gaussian}
As mentioned in the introduction one of the differences from previous works on limit theorems for BSS processes is that we use a different scaling factor $\tau$. In this chapter we are going to present two cases. For the first case we use the same $\tau$ used in the literature, while in the second we use a new formulation. The differences between the two approaches will be pointed out subsequently.
\begin{rem}
	Throughout this chapter $\{\textbf{G}_{t} \}_{t\in[0,T]}$ is a general stationary multivariate Gaussian process. Thus, it is not necessarily the Gaussian core.
\end{rem}
\subsection{Case I}
For $i,j=1,...,p$ and $k\in\mathbb{N}$, let us define the scaling factor by
\begin{equation}\label{tau}
\tau_{n}^{(j)}:=\sqrt{\mathbb{E}\left[\left(\Delta^{n}_{1}G^{(j)}\right)^{2}\right]},
\end{equation}
and the multivariate process $\{\textbf{Z}_{t}^{n}\}_{t\in[0,T]}=\{(Z_{(1,1),t}^{n},...,Z_{(p,p),t}^{n})^{\top}\}_{t\in[0,T]}$ as
\begin{equation*}
Z_{(i,j),t}^{n}:=\frac{1}{\sqrt{n}}\sum_{i=1}^{\lfloor nt\rfloor}I_{2}\left(\dfrac{\Delta^{n}_{l}G^{(i)}}{\tau_{n}^{(i)}}\tilde{\otimes}\dfrac{\Delta^{n}_{l}G^{(j)}}{\tau_{n}^{(j)}}\right),\quad
\text{and }\quad r_{i,j}^{(n)}(k):=\mathbb{E}\left[\dfrac{\Delta^{n}_{1}G^{(i)}}{\tau_{n}^{(i)}}\dfrac{\Delta^{n}_{1+k}G^{(j)}}{\tau_{n}^{(j)}} \right].
\end{equation*}
Thanks to Theorem 2.7.7 of \cite{NP} (reported in this work as Theorem \ref{Theorem 2.7.10}), $Z_{(i,j),t}^{n}$ belongs to $\mathcal{W}_{2}$, namely the second Wiener chaos. In addition, we will consider the following assumption on the correlation.
\begin{ass}\label{A2} Let $\lim\limits_{n\rightarrow\infty}\sum_{k=1}^{\infty}\left(r_{i,j}^{(n)}(k)\right)^{2}<\infty$, for $i,j=1,...,p$.
\end{ass}
\begin{rem}Let $\lim\limits_{n\rightarrow\infty}\sum_{k=1}^{\infty}\left(r_{i,j}^{(n)}(k)\right)^{2}<\infty$, for $i,j=1,...,p$. For $x,y,z,w=1,...,p$, the following limit holds
\begin{equation*}
\lim\limits_{n\rightarrow\infty}\frac{1}{n}\sum_{k=1}^{n}k\left(r_{x,z}^{(n)}(k)r_{y,w}^{(n)}(k)+r_{y,z}^{(n)}(k)r_{x,w}^{(n)}(k)\right)+\frac{1}{n}\sum_{k=n+1}^{2n-1}(2n-k)\left(r_{x,z}^{(n)}(k)r_{y,w}^{(n)}(k)+r_{y,z}^{(n)}(k)r_{x,w}^{(n)}(k)\right)=0.
\end{equation*}
This is because given two sequences $\{a_{k}\}_{k\in\mathbb{N}}$ and $\{b_{k}\}_{k\in\mathbb{N}}$ then if $\sum_{k=1}^{\infty}\left(a_{k}\right)^{2}<\infty$ and $\sum_{k=1}^{\infty}\left(b_{k}\right)^{2}<\infty$ then $\lim\limits_{l\rightarrow\infty}\frac{1}{l}\sum_{k=1}^{l}k\left(a_{k}\right)^{2}=0$ and $\lim\limits_{l\rightarrow\infty}\frac{1}{l}\sum_{k=1}^{l}k\left(b_{k}\right)^{2}=0$ which imply that $\lim\limits_{l\rightarrow\infty}\frac{1}{l}\sum_{k=1}^{l}k a_{k}b_{k}=0$ (see the use of condition (3.9) in \cite{BCD}).
\end{rem}
We now present two propositions regarding the process $\{\textbf{Z}_{t}^{n}\}_{t\in[0,T]}$ and they will lead us to the main theorem of this section. The first one is on the convergence of the finite dimensional distributions, while the second one is on the tightness of the law of the process.
\begin{pro}\label{C-M}
Let $d\in\mathbb{N}$ and let $(a_{l},b_{l}]$ pairwise disjoint intervals in $[0,T]$ where $l=1,..,d$. Consider
\begin{equation*}
\textbf{Z}_{b_{l}}^{n}-\textbf{Z}_{a_{l}}^{n}=(Z_{(1,1),b_{l}}^{n}-Z_{(1,1),a_{l}}^{n},...,Z_{(1,p),b_{l}}^{n}-Z_{(1,p),a_{l}}^{n},Z_{(2,1),b_{l}}^{n}-Z_{(2,1),a_{l}}^{n},...,Z_{(p,p),b_{l}}^{n}-Z_{(p,p),a_{l}}^{n})^{\top}.
\end{equation*}
Then, under the Assumption \ref{A2}, $(\textbf{Z}_{b_{l}}^{n}-\textbf{Z}_{a_{l}}^{n})_{1\leq l\leq d}\stackrel{d}{\rightarrow}(\textbf{D}^{\frac{1}{2}}(B_{b_{l}}-B_{a_{l}}))_{1\leq l\leq d}$ as $n\rightarrow\infty$, where $\textbf{D}\in\mathcal{M}^{p^{2}\times p^{2}}(\mathbb{R})$ and $B_{t}$ is a $p^{2}$-dimensional Brownian motion. In particular, associating for each combination $(i,j)$ a combination $((x,y),(z,w))$, where $x,y,z,w=1,...,p$, using the formula $(i,j)\leftrightarrow((\lfloor \frac{i-1}{p} \rfloor+1,i-p\lfloor \frac{i-1}{p} \rfloor),(\lfloor \frac{j-1}{p} \rfloor+1,j-p\lfloor \frac{j-1}{p} \rfloor))$, we have $(\textbf{D})_{ij}=(\textbf{D})_{(x,y),(z,w)}$
\begin{equation*}
=\lim\limits_{n\rightarrow\infty}\frac{2}{n}\sum_{k=1}^{n-1}(n-k)\left(r_{x,z}^{(n)}(k)r_{y,w}^{(n)}(k)+r_{y,z}^{(n)}(k)r_{x,w}^{(n)}(k)\right)+\left(r_{x,z}^{(n)}(0)r_{y,w}^{(n)}(0)+r_{y,z}^{(n)}(0)r_{x,w}^{(n)}(0)\right).
\end{equation*}
\end{pro}
\begin{proof}
First, we compute the covariances. Notice that it is sufficient to focus on the case $l=2$ and $a_{1}=0,b_{1}=a_{2}=1,b_{2}=2$. Recall the isometry property of integrals (\textit{i.e}.~Proposition \ref{isometry property of integrals}) and that for $f_{1},g_{1},f_{2},g_{2}\in\mathcal{H}$ we have $\langle f_{1}\otimes g_{1},f_{2}\otimes g_{2}\rangle_{\mathcal{H}^{\otimes 2}}:=\langle f_{1},f_{2}\rangle_{\mathcal{H}}\langle g_{1},g_{2}\rangle_{\mathcal{H}}$. Then, for $x,y,z,w=1,...,p$, we have
\begin{equation*}
\mathbb{E}[(Z_{(x,y),1}^{n}-Z_{(x,y),0}^{n})(Z_{(z,w),2}^{n}-Z_{(z,w),1}^{n})]
=\frac{2}{n}\left\langle\sum_{i=1}^{n}\dfrac{\Delta^{n}_{i}G^{(x)}}{\tau_{n}^{(x)}}\tilde{\otimes}\dfrac{\Delta^{n}_{i}G^{(y)}}{\tau_{n}^{(y)}},\sum_{j=n+1}^{2n}\dfrac{\Delta^{n}_{j}G^{(z)}}{\tau_{n}^{(z)}}\tilde{\otimes}\dfrac{\Delta^{n}_{j}G^{(w)}}{\tau_{n}^{(w)}} \right\rangle_{\mathcal{H}^{\otimes 2}}
\end{equation*}
\begin{equation*}
=\frac{1}{2n}\sum_{i=1}^{n}\sum_{j=n+1}^{2n}\left\langle\dfrac{\Delta^{n}_{i}G^{(x)}}{\tau_{n}^{(x)}}\otimes\dfrac{\Delta^{n}_{i}G^{(y)}}{\tau_{n}^{(y)}},\dfrac{\Delta^{n}_{j}G^{(z)}}{\tau_{n}^{(z)}}\otimes\dfrac{\Delta^{n}_{j}G^{(w)}}{\tau_{n}^{(w)}} \right\rangle_{\mathcal{H}^{\otimes 2}}
\end{equation*}
\begin{equation*}
+\left\langle\dfrac{\Delta^{n}_{i}G^{(y)}}{\tau_{n}^{(y)}}\otimes\dfrac{\Delta^{n}_{i}G^{(x)}}{\tau_{n}^{(x)}},\dfrac{\Delta^{n}_{j}G^{(z)}}{\tau_{n}^{(z)}}\otimes\dfrac{\Delta^{n}_{j}G^{(w)}}{\tau_{n}^{(w)}} \right\rangle_{\mathcal{H}^{\otimes 2}}+\left\langle\dfrac{\Delta^{n}_{i}G^{(x)}}{\tau_{n}^{(x)}}\otimes\dfrac{\Delta^{n}_{i}G^{(y)}}{\tau_{n}^{(y)}},\dfrac{\Delta^{n}_{j}G^{(w)}}{\tau_{n}^{(w)}}\otimes\dfrac{\Delta^{n}_{j}G^{(z)}}{\tau_{n}^{(z)}} \right\rangle_{\mathcal{H}^{\otimes 2}}
\end{equation*}
\begin{equation*}
+\left\langle\dfrac{\Delta^{n}_{i}G^{(y)}}{\tau_{n}^{(y)}}\otimes\dfrac{\Delta^{n}_{i}G^{(x)}}{\tau_{n}^{(x)}},\dfrac{\Delta^{n}_{j}G^{(w)}}{\tau_{n}^{(w)}}\otimes\dfrac{\Delta^{n}_{j}G^{(z)}}{\tau_{n}^{(z)}} \right\rangle_{\mathcal{H}^{\otimes 2}}
\end{equation*}
\begin{equation*}
=\frac{1}{2n}\sum_{i=1}^{n}\sum_{j=n+1}^{2n}\mathbb{E}\left[\dfrac{\Delta^{n}_{i}G^{(x)}}{\tau_{n}^{(x)}}\dfrac{\Delta^{n}_{j}G^{(z)}}{\tau_{n}^{(z)}}\right]\mathbb{E}\left[\dfrac{\Delta^{n}_{i}G^{(y)}}{\tau_{n}^{(y)}}\dfrac{\Delta^{n}_{j}G^{(w)}}{\tau_{n}^{(w)}}\right]+\mathbb{E}\left[\dfrac{\Delta^{n}_{i}G^{(y)}}{\tau_{n}^{(y)}}\dfrac{\Delta^{n}_{j}G^{(z)}}{\tau_{n}^{(z)}}\right]\mathbb{E}\left[\dfrac{\Delta^{n}_{i}G^{(x)}}{\tau_{n}^{(x)}}\dfrac{\Delta^{n}_{j}G^{(w)}}{\tau_{n}^{(w)}}\right]
\end{equation*}
\begin{equation*}
+\mathbb{E}\left[\dfrac{\Delta^{n}_{i}G^{(x)}}{\tau_{n}^{(x)}}\dfrac{\Delta^{n}_{j}G^{(w)}}{\tau_{n}^{(w)}}\right]\mathbb{E}\left[\dfrac{\Delta^{n}_{i}G^{(y)}}{\tau_{n}^{(y)}}\dfrac{\Delta^{n}_{j}G^{(z)}}{\tau_{n}^{(z)}}\right]+\mathbb{E}\left[\dfrac{\Delta^{n}_{i}G^{(y)}}{\tau_{n}^{(y)}}\dfrac{\Delta^{n}_{j}G^{(w)}}{\tau_{n}^{(w)}}\right]\mathbb{E}\left[\dfrac{\Delta^{n}_{i}G^{(x)}}{\tau_{n}^{(x)}}\dfrac{\Delta^{n}_{j}G^{(z)}}{\tau_{n}^{(z)}}\right]
\end{equation*}
\begin{equation*}
=\frac{1}{n}\sum_{i=1}^{n}\sum_{j=n+1}^{2n}\mathbb{E}\left[\dfrac{\Delta^{n}_{i}G^{(x)}}{\tau_{n}^{(x)}}\dfrac{\Delta^{n}_{j}G^{(z)}}{\tau_{n}^{(z)}}\right]\mathbb{E}\left[\dfrac{\Delta^{n}_{i}G^{(y)}}{\tau_{n}^{(y)}}\dfrac{\Delta^{n}_{j}G^{(w)}}{\tau_{n}^{(w)}}\right]+\mathbb{E}\left[\dfrac{\Delta^{n}_{i}G^{(y)}}{\tau_{n}^{(y)}}\dfrac{\Delta^{n}_{j}G^{(z)}}{\tau_{n}^{(z)}}\right]\mathbb{E}\left[\dfrac{\Delta^{n}_{i}G^{(x)}}{\tau_{n}^{(x)}}\dfrac{\Delta^{n}_{j}G^{(w)}}{\tau_{n}^{(w)}}\right]
\end{equation*}
\begin{equation*}
=\frac{1}{n}\sum_{k=1}^{n}k\left(r_{x,z}^{(n)}(k)r_{y,w}^{(n)}(k)+r_{y,z}^{(n)}(k)r_{x,w}^{(n)}(k)\right)+\frac{1}{n}\sum_{k=n+1}^{2n-1}(2n-k)\left(r_{x,z}^{(n)}(k)r_{y,w}^{(n)}(k)+r_{y,z}^{(n)}(k)r_{x,w}^{(n)}(k)\right).
\end{equation*}
Hence, we have that, for $x,y,z,w=1,...,p$,
\begin{equation*}
\mathbb{E}[(Z_{(x,y),1}^{n}-Z_{(x,y),0}^{n})(Z_{(z,w),2}^{n}-Z_{(z,w),1}^{n})]
\end{equation*}
\begin{equation*}
=\frac{1}{n}\sum_{k=1}^{n}k\left(r_{x,z}^{(n)}(k)r_{y,w}^{(n)}(k)+r_{y,z}^{(n)}(k)r_{x,w}^{(n)}(k)\right)+\frac{1}{n}\sum_{k=n+1}^{2n-1}(2n-k)\left(r_{x,z}^{(n)}(k)r_{y,w}^{(n)}(k)+r_{y,z}^{(n)}(k)r_{x,w}^{(n)}(k)\right).
\end{equation*}
Under Assumption \ref{A2} we have that $\lim\limits_{n\rightarrow\infty}\mathbb{E}[(Z_{(x,y),1}^{n}-Z_{(x,y),0}^{n})(Z_{(z,w),2}^{n}-Z_{(z,w),1}^{n})]=0$.
\\ For the variances, following similar computations as above, we have
\begin{equation*}
\mathbb{E}[(Z_{(x,y),1}^{n}-Z_{(x,y),0}^{n})(Z_{(z,w),1}^{n}-Z_{(z,w),0}^{n})]
\end{equation*}
\begin{equation*}
=\frac{2}{n}\sum_{k=1}^{n-1}(n-k)\left(r_{x,z}^{(n)}(k)r_{y,w}^{(n)}(k)+r_{y,z}^{(n)}(k)r_{x,w}^{(n)}(k)\right)+\left(r_{x,z}^{(n)}(0)r_{y,w}^{(n)}(0)+r_{y,z}^{(n)}(0)r_{x,w}^{(n)}(0)\right).
\end{equation*}
Under Assumption \ref{A2} we have that $\lim\limits_{n\rightarrow\infty}\mathbb{E}[(Z_{(x,y),1}^{n}-Z_{(x,y),0}^{n})(Z_{(z,w),1}^{n}-Z_{(z,w),0}^{n})]<\infty$.

Further, since we need to use a matrix formulation for our result, we need to associate an element $((x,y),(z,w)))$ with $x,y,z,w=1,...,p$ to $(i,j)$ with $i,j=1,...,p^{2}$. Thus, we make the following association $i,j=1\leftrightarrow (1,1)$; $i,j=2\leftrightarrow (1,2)$;...; $i,j=p\leftrightarrow (1,p)$; $i,j=p+1\leftrightarrow (2,1)$;...; $i,j=2p\leftrightarrow (2,p)$; $i,j=2p+1\leftrightarrow (3,1)$;...; $i,j=p^{2}\leftrightarrow (p,p)$. This can be written compactly as $(i,j)\leftrightarrow((\lfloor \frac{i-1}{p} \rfloor+1,i-p\lfloor \frac{i-1}{p} \rfloor),(\lfloor \frac{j-1}{p} \rfloor+1,j-p\lfloor \frac{j-1}{p} \rfloor))$.\\
The last step is to show that for $l=1,..,d$ and $x,y=1,...,p$
\begin{equation*}
\Bigg\|\left( \frac{1}{\sqrt{n}}\sum_{i=\lfloor na_{l}\rfloor+1}^{\lfloor nb_{l}\rfloor}\dfrac{\Delta^{n}_{i}G^{(x)}}{\tau_{n}^{(x)}}\tilde{\otimes}\dfrac{\Delta^{n}_{i}G^{(y)}}{\tau_{n}^{(y)}} \right)\otimes_{1}\left( \frac{1}{\sqrt{n}}\sum_{i=\lfloor na_{l}\rfloor+1}^{\lfloor nb_{l}\rfloor}\dfrac{\Delta^{n}_{i}G^{(x)}}{\tau_{n}^{(x)}}\tilde{\otimes}\dfrac{\Delta^{n}_{i}G^{(y)}}{\tau_{n}^{(y)}} \right)\Bigg\|_{\mathcal{H}^{\otimes 2}}\rightarrow0.
\end{equation*}
However, this is true under the Assumption \ref{A2} by using the same arguments used in the proof of Theorem 4.2 of \cite{Andrea1}.
\end{proof}
\begin{pro}\label{tight-M}
Under the Assumption \ref{A2}, let $\mathbf{P}^{n}$ be the law of the process $\{\textbf{Z}^{n}_{t}\}_{t\in[0,T]}$ on the Skorokhod space $\mathcal{D}([0,T],\mathbb{R}^{p^{2}})$ Then, the sequence $\{\mathbf{P}^{n}\}_{n\in\mathbb{N}}$ is tight.
\end{pro}
\begin{proof}
It follows from the tightness of the components of the vector $\textbf{Z}^{n}_{t}$ which is proved following the same arguments as in the proof of Theorem 2 in \cite{BCD} or of Theorem 4.3 of \cite{Andrea1} or of Theorem 7 of \cite{Cor}. The differences are that instead of $i,j=1,2$, we have $i,j=1,...,p$ and that we have a more general but still sufficient assumption (\textit{i.e}.~Assumption \ref{A2}).
\end{proof}
\begin{thm}\label{CLT-M}
Let the Assumption \ref{A2} hold. Then we have
\begin{equation*}
\left\{\frac{1}{\sqrt{n}}\sum_{l=1}^{\lfloor nt\rfloor}\left(\dfrac{\Delta^{n}_{l}G^{(i)}}{\tau_{n}^{(i)}}\dfrac{\Delta^{n}_{l}G^{(j)}}{\tau_{n}^{(j)}}-\mathbb{E}\left[\dfrac{\Delta^{n}_{l}G^{(i)}}{\tau_{n}^{(i)}}\dfrac{\Delta^{n}_{l}G^{(j)}}{\tau_{n}^{(j)}} \right] \right)_{i,j=1,...,p}\right\}_{t\in[0,T]}\stackrel{st}{\rightarrow}\left\{\textbf{D}^{\frac{1}{2}}B_{t}\right\}_{t\in[0,T]},
\end{equation*}
where $\textbf{D}$ and $B_{t}$ are given in Proposition \ref{C-M}. In particular, $B_{t}$ is independent of $G^{(1)},...,G^{(p)}$ and the convergence is in $\mathcal{D}([0,T],\mathbb{R}^{p^{2}})$, namely the Skorokhod space equipped with the uniform metric.
\end{thm}
\begin{proof}
First, notice that
\begin{equation*}
\left\{\frac{1}{\sqrt{n}}\sum_{l=1}^{\lfloor nt\rfloor}\left(\dfrac{\Delta^{n}_{l}G^{(i)}}{\tau_{n}^{(i)}}\dfrac{\Delta^{n}_{l}G^{(j)}}{\tau_{n}^{(j)}}-\mathbb{E}\left[\dfrac{\Delta^{n}_{l}G^{(i)}}{\tau_{n}^{(i)}}\dfrac{\Delta^{n}_{l}G^{(j)}}{\tau_{n}^{(j)}} \right] \right)_{i,j=1,...,p}\right\}_{t\in[0,T]}\stackrel{d}{\rightarrow}\left\{\textbf{D}^{\frac{1}{2}}B_{t}\right\}_{t\in[0,T]},
\end{equation*}
follows from Proposition $\ref{C-M}$ and Proposition \ref{tight-M}, and following the arguments used in the proof of Theorem 4.4 in \cite{Andrea1}. The independence of $B_{t}$ from $G^{(1)}_{t},...,G^{(p)}_{t}$ is given by the fact that $G^{(i)}_{t},...,G^{(p)}_{t}$ belong to the first Wiener chaos, while $B_{t}$  is the limiting process of objects belonging to the second Wiener chaos. Moreover, we have
\begin{equation*}\label{G and the Wiener chaos}
\left\{\left(G_{t}^{(i)}\right)_{i=1,...,p},\frac{1}{\sqrt{n}}\sum_{l=1}^{\lfloor nt \rfloor}\left(\dfrac{\Delta^{n}_{l}G^{(i)}}{\tau_{n}^{(i)}}\dfrac{\Delta^{n}_{l}G^{(j)}}{\tau_{n}^{(j)}}-\mathbb{E}\left[\dfrac{\Delta^{n}_{l}G^{(i)}}{\tau_{n}^{(i)}}\dfrac{\Delta^{n}_{l}G^{(j)}}{\tau_{n}^{(j)}} \right] \right)_{i,j=1,...,p}\right\}_{t\in[0,T]}
\end{equation*}
\begin{equation}\label{AldoProp2}
\stackrel{d}{\rightarrow}\left\{\left(G_{t}^{(i)}\right)_{i=1,...,p},\textbf{D}^{\frac{1}{2}}B_{t}\right\}_{t\in[0,T]},
\end{equation}
in the space $\mathcal{D}([0,T],\mathbb{R}^{p}\times\mathbb{R}^{p^{2}})$, namely the Skorokhod space equipped with the uniform metric. Notice that this result comes from the convergence of the finite dimensional distributions of $(\ref{G and the Wiener chaos})$, which follows from the arguments at the beginning of this proof together with the orthogonality of different Wiener chaos, and from the tightness of the law of $(\ref{G and the Wiener chaos})$, which follows from the tightness of each component of the vector proved in Proposition \ref{tight-M}. 
\\Concerning the convergence of the finite dimensional distributions, notice that for any $t\in[0,T]$ each element of $\left(G_{t}^{(i)}\right)_{i=1,...,p}$ and of $\frac{1}{\sqrt{n}}\sum_{l=1}^{\lfloor nt \rfloor}\left(\frac{\Delta^{n}_{l}G^{(i)}}{\tau_{n}^{(i)}}\frac{\Delta^{n}_{l}G^{(j)}}{\tau_{n}^{(j)}}-\mathbb{E}\left[\frac{\Delta^{n}_{l}G^{(i)}}{\tau_{n}^{(i)}}\frac{\Delta^{n}_{l}G^{(j)}}{\tau_{n}^{(j)}} \right] \right)_{i,j=1,...,p}$ belong to the first and second Wiener chaos, respectively; hence, for any $i,j,k=1,...,p$ and $s,t\in[0,T]$, we have
\begin{equation*}
\mathbb{E}\left[G_{s}^{(k)}\frac{1}{\sqrt{n}}\sum_{l=1}^{\lfloor nt \rfloor}\left(\frac{\Delta^{n}_{l}G^{(i)}}{\tau_{n}^{(i)}}\frac{\Delta^{n}_{l}G^{(j)}}{\tau_{n}^{(j)}}-\mathbb{E}\left[\frac{\Delta^{n}_{l}G^{(i)}}{\tau_{n}^{(i)}}\frac{\Delta^{n}_{l}G^{(j)}}{\tau_{n}^{(j)}} \right] \right)\right]=0.
\end{equation*}
Then, the convergence of the finite dimensional distributions, that is for any $M\in\mathbb{N}$ and disjoint intervals $[a_{m},b_{m}]$ with $m=1,...,M$
\begin{equation*}
\left(\left(G_{b_{m}}^{(i)}-G_{a_{m}}^{(i)}\right)_{i=1,...,p},\frac{1}{\sqrt{n}}\sum_{l=a_{m}}^{\lfloor nb_{m} \rfloor}\left(\frac{\Delta^{n}_{l}G^{(i)}}{\tau_{n}^{(i)}}\frac{\Delta^{n}_{l}G^{(j)}}{\tau_{n}^{(j)}}-\mathbb{E}\left[\frac{\Delta^{n}_{l}G^{(i)}}{\tau_{n}^{(i)}}\frac{\Delta^{n}_{l}G^{(j)}}{\tau_{n}^{(j)}} \right] \right)_{i,j=1,...,p}\right)_{m=1,...,M}
\end{equation*}
\begin{equation*}
\stackrel{d}{\rightarrow}\left(\left(G_{b_{m}}^{(i)}-G_{a_{m}}^{(i)}\right)_{i=1,...,p},\textit{\textbf{D}}^{1/2}(B_{b_{m}}-B_{a_{m}})\right)_{m=1,...,M},
\end{equation*}
follow from Theorem 6.2.3 in \cite{NP} and the previous arguments.

In order to obtain the stable convergence is sufficient to use Proposition 2 of \cite{Aldo}. Observe that condition $D''$ in that proposition is implied by $(\ref{AldoProp2})$ using Bayes' theorem and independence of the limiting process $\{B_{t}\}_{t\in[0,T]}$ from $\left(\{G_{t}^{(i)}\}_{t\in[0,T]}\right)_{i=1,...,p}$. In particular, notice that by Bayes' theorem: $(\ref{AldoProp2})\Rightarrow$ for all fixed $n_{1},...,n_{k}\in\mathbb{N}$ and
\begin{equation*}
A\in\sigma\Bigg(\frac{1}{\sqrt{n_{1}}}\sum_{l=1}^{\lfloor n_{1}t \rfloor}\left(\frac{\Delta^{n_{1}}_{l}G^{(i)}}{\tau_{n_{1}}^{(i)}}\frac{\Delta^{n_{1}}_{l}G^{(j)}}{\tau_{n_{1}}^{(j)}}-\mathbb{E}\left[\frac{\Delta^{n_{1}}_{l}G^{(i)}}{\tau_{n_{1}}^{(i)}}\frac{\Delta^{n_{1}}_{l}G^{(j)}}{\tau_{n_{1}}^{(j)}} \right] \right)_{i,j=1,...,p},...,
\end{equation*}
\begin{equation*}
\frac{1}{\sqrt{n_{k}}}\sum_{l=1}^{\lfloor n_{k}t \rfloor}\left(\frac{\Delta^{n_{k}}_{l}G^{(i)}}{\tau_{n_{k}}^{(i)}}\frac{\Delta^{n_{k}}_{l}G^{(j)}}{\tau_{n_{k}}^{(j)}}-\mathbb{E}\left[\frac{\Delta^{n_{k}}_{l}G^{(i)}}{\tau_{n_{k}}^{(i)}}\frac{\Delta^{n_{k}}_{l}G^{(j)}}{\tau_{n_{k}}^{(j)}} \right] \right)_{i,j=1,...,p}\Bigg),\text{ with $\mathbb{P}(A)>0$, we have }
\end{equation*}
\begin{equation*}
\lim\limits_{n\rightarrow\infty}\mathbb{P}\left(\left(\frac{1}{\sqrt{n}}\sum_{l=1}^{\lfloor nt \rfloor}\frac{\Delta^{n}_{l}G^{(i)}}{\tau_{n}^{(i)}}\frac{\Delta^{n}_{l}G^{(j)}}{\tau_{n}^{(j)}}-\mathbb{E}\left[\frac{\Delta^{n}_{l}G^{(i)}}{\tau_{n}^{(i)}}\frac{\Delta^{n}_{l}G^{(j)}}{\tau_{n}^{(j)}} \right] <x^{(i,j)}\right)_{i,j=1,...,p}\Bigg|A\right)
\end{equation*}
\begin{equation*}
=\mathbb{P}\left(\left(\left(\textit{\textbf{D}}^{1/2}B_{t}\right)_{i,j}<x^{(i,j)}\right)_{i,j=1,...,p}\Bigg|A\right).
\end{equation*}
In addition, by the independence of the limiting process $\{B_{t}\}_{t\in[0,T]}$ from $\left(\{G_{t}^{(i)}\}_{t\in[0,T]}\right)_{i=1,...,p}$ we get
\begin{equation*}
\mathbb{P}\left(\left(\left(\textit{\textbf{D}}^{1/2}B_{t}\right)_{i,j}<x^{(i,j)}\right)_{i,j=1,...,p}\Bigg|A\right)=\mathbb{P}\left(\left(\left(\textit{\textbf{D}}^{1/2}B_{t}\right)_{i,j}<x^{(i,j)}\right)_{i,j=1,...,p}\right).
\end{equation*}
Following the same computations, it is possible to show the result for any set of points $\{t_{1},....,t_{a}\}\in[0,T]^{a}$ for $a\in\mathbb{N}$ and not just one $t\in[0,T]$. Therefore, we obtain mixing convergence (see \cite{Aldo} for details), hence stable convergence, in $\mathcal{D}([0,T],\mathbb{R}^{p^{2}})$.
\end{proof}
\subsection{Case II}\label{A second case}
Let $i,j=1,...,p$. In this section we will show that the results presented in the previous section hold for different choice of $\tau_{n}$. The new $\tau_{n}$ have an order equal to or greater than the order of the previous $\tau_{n}$ as $n$ goes to infinity. Indeed, let $\tau_{n}^{(\beta(j))}:=O\left(\tau_{n}^{(j)}\right)$ for $j=1,...,p$ (\textit{e.g}.~consider $\tau_{n}^{(\beta(j))}=\max\limits_{j=1,...,p}\left(\tau_{n}^{(j)}\right)$ for some $j$ or see Example \ref{Example1}). In this section we will work with the Hilbert space generated by the Gaussian random variables 
\begin{equation*}
\left(\dfrac{\Delta^{n}_{l}G^{(j)}}{\tau_{n}^{(\beta(j))}}\right)_{n\geq1,1\leq l\leq \lfloor nt\rfloor,j\in\{1,...,p\}}\quad\text{and let}\quad r_{i,j}^{(\beta),(n)}(k):=\mathbb{E}\left[\frac{\Delta^{n}_{1}G^{(i)}}{\tau_{n}^{(\beta(i))}}\frac{\Delta^{n}_{1+k}G^{(j)}}{\tau_{n}^{(\beta(j))}} \right].
\end{equation*}
\begin{ass}\label{A3}Let $\lim\limits_{n\rightarrow\infty}\sum_{k=1}^{\infty}\left(r_{i,j}^{(\beta),(n)}(k)\right)^{2}<\infty$, for $i,j=1,...,p$.
\end{ass}
We can now present and prove a modification of the main result of the previous section.
\begin{thm}\label{2-CLT-M}
Let the Assumption \ref{A3} hold. Let $\alpha,\beta=1,2$. Then we have
\begin{equation*}
\left\{\frac{1}{\sqrt{n}}\sum_{l=1}^{\lfloor nt\rfloor}\left(\dfrac{\Delta^{n}_{l}G^{(i)}}{\tau_{n}^{(\beta(i))}}\dfrac{\Delta^{n}_{l}G^{(j)}}{\tau_{n}^{(\beta(j))}}-\mathbb{E}\left[\dfrac{\Delta^{n}_{l}G^{(i)}}{\tau_{n}^{(\beta(i))}}\dfrac{\Delta^{n}_{l}G^{(j)}}{\tau_{n}^{(\beta(j))}} \right] \right)_{i,j=1,...,p}\right\}_{t\in[0,T]}\stackrel{st}{\rightarrow}\left\{\textbf{D}^{(\beta)\,\frac{1}{2}}B_{t}\right\}_{t\in[0,T]},
\end{equation*}
In particular, associating for each combination $(i,j)$ a combination $((x,y),(z,w))$ where $x,y,z,w=1,...,p$ using the formula $(i,j)\leftrightarrow((\lfloor \frac{i-1}{p} \rfloor+1,i-p\lfloor \frac{i-1}{p} \rfloor),(\lfloor \frac{j-1}{p} \rfloor+1,j-p\lfloor \frac{j-1}{p} \rfloor))$, we have
\begin{equation*}
(\textbf{D})_{ij}^{(\beta)}=(\textbf{D})_{(x,y),(z,w)}^{(\beta)}=\lim\limits_{n\rightarrow\infty}\frac{2}{n}\sum_{k=1}^{n-1}(n-k)\left(r_{x,z}^{(\beta),(n)}(k)r_{y,w}^{(\beta),(n)}(k)+r_{y,z}^{(\beta),(n)}(k)r_{x,w}^{(\beta),(n)}(k)\right)
\end{equation*}
\begin{equation*}
+\left(r_{x,z}^{(\beta),(n)}(0)r_{y,w}^{(\beta),(n)}(0)+r_{y,z}^{(\beta),(n)}(0)r_{x,w}^{(\beta),(n)}(0)\right).
\end{equation*}
Further, $B_{t}$ is a $p^{2}$-dimensional Brownian motion independent of $G^{(1)},...,G^{(p)}$ and the convergence is in $\mathcal{D}([0,T],\mathbb{R}^{p^{2}})$, namely the Skorokhod space equipped with the uniform metric.
\end{thm}
\begin{proof}
It follows from the same arguments as in the proofs of Proposition $\ref{C-M}$, Proposition \ref{tight-M}, and Theorem \ref{CLT-M}. This is because the only difference is that we have a greater denominator than before (\textit{i.e}.~$\tau_{n}^{(\beta(j))}\geq \tau_{n}^{(j)}$) which changes neither the logic of the arguments nor the computations. This is because in our framework we do not need that $\mathbb{E}\left[\left( \frac{\Delta_{i}^{n}G^{(k)}}{\tau_{n}^{(k)}}\right)^{2} \right]=1$. This was different for the previous literature where the equality was needed in order to use Theorem 2.7.7 together with Theorem 2.7.8 of \cite{NP}.
\end{proof}
\begin{rem}
Notice that while the larger value of the $\tau_{n}^{(\beta(j))}$ does not trigger any modification in the proof of the results, it may reduce some of the components of $\textbf{D}^{(\beta)}$ to zero. Indeed, the ideal situation would be the one where, for some $j$, $\tau_{n}^{(\beta(j))}= \tau_{n}^{(j)}$ and, for others, $\tau_{n}^{(\beta(j))}>\tau_{n}^{(j)}$.
\end{rem}
\begin{exmp}\label{Example1}
Let in this example $\{\textbf{G}_{t}\}_{t\in[0,T]}$ be a Gaussian core. Consider a partition of the set $\{1,...,p\}$ and call its elements $I_{\alpha_{1}},...,I_{\alpha_{v}}$ for some $v\in\mathbb{N}$. Hence, $I_{\alpha_{h}}\subset\{1,...,p\}$ for $h=1,...,v$, and $I_{\alpha_{h}}\cap I_{\alpha_{l}}=\emptyset$ for $h,l=1,...,v$ with $h\neq l$. For $h=1,..,v$, define 
\begin{equation}\label{tau-secondcase}
\tau_{n}^{(\alpha_{h})}:=\sqrt{\mathbb{E}\left[\left(\sum_{i\in I_{\alpha_{h}}}\Delta^{n}_{1}G^{(i)} \right)^{2}\right]},
\end{equation}
 and consider $\frac{\Delta^{n}_{l}G^{(j)}}{\tau_{n}^{(\beta(j))}}$ where $\beta(j):=\alpha_{h}$ when $j\in I_{\alpha_{h}}$. In addition, assume that, for $h=1,...,v$, the $\mathcal{F}_{t}$-Brownian measures $W^{(i)}$ are \textit{independent} for $i\in I_{\alpha_{h}}$. This means that $\mathbb{E}[W^{(i)}W^{(j)}]=0$ if $i,j\in I_{\alpha_{h}}$ for some $h=1,...,v$. Then $\tau_{n}^{(\beta(j))}\geq\tau_{n}^{(j)}$ for $j=1,..,p$, where $\tau_{n}^{(j)}$ has been defined in $(\ref{tau})$ and, thus, we can apply Theorem \ref{2-CLT-M}.
\end{exmp}
\section{Joint CLT for the multivariate BSS process}\label{CLT-BSS-chapter}
In this chapter we will present and prove our main results consisting in the joint central limit theorem for the two types of multivariate BSS processes.
We will present first the CLT for the scaling factor $\tau$ used in the literature (\textit{i.e}~Case I) and then the CLT for the new formulation (\textit{i.e}.~Case II). For Case II we have two scenarios depending on which multivariate extension of the univariate BSS process we consider (see Definition \ref{BSS}). In this and in the next two chapters we will use a multivariate version of the continuous mapping theorem applied to stable convergence (for reference see \cite{Aldo}). Moreover, we will adopt the following three assumptions.
\begin{ass}\label{1}
	For $m,l=1,...,p$, let $g^{(m,l)}$ be continuously differentiable with derivative $(g^{(m,l)})'\in L^{2}((b^{(m,l)},\infty))$ and $(g^{(m,l)})'$ is non-increasing in $[b^{(m,l)},\infty)$ for some $b^{(m,l)}>0$. Let $\sigma^{(m,l)}$ has $\alpha^{(m,l)}$-H\"{o}lder continuous sample paths, for $\alpha^{(m,l)}\in\left(\frac{1}{2},1\right)$. Define
\begin{equation*}
\pi_{n}^{(m,l)}(A):=\frac{\int_{A}\left(g^{(m,l)}(s+\Delta_{n})-g^{(m,l)}(s)\right)^{2}ds}{\int_{0}^{\infty}\left(g^{(m,l)}(s+\Delta_{n})-g^{(m,l)}(s)\right)^{2}ds},\quad\text{ where }\quad A\in\mathcal{B}(\mathbb{R}).
\end{equation*}
We impose that there exists a constant $\lambda<-1$ such that for any $\epsilon_{n}=O\left(n^{-\kappa}\right)$, $\kappa\in(0,1)$, we have
\begin{equation*}
\pi_{n}^{(m,l)}((\epsilon_{n},\infty))=O(n^{\lambda(1-\kappa)}).
\end{equation*}
\end{ass}
In the next sections we are going to present different cases. Each case will have particular $\tau_{n}$ and $r^{(n)}$, but the underlying assumptions will have the same structure. Hence, for the last two assumptions we are going to use the variables $\tau_{n}$ and $r^{(n)}$, whose specific form will not be introduced here, but instead it will be specified in the context where the assumptions are used. In other words, we preferred to have a general form for these two assumptions (with an unspecified $\tau_{n}$ and $r^{(n)}$) in order to avoid repeating the same assumptions with different $\tau_{n}$ and $r^{(n)}$ for each case.
\begin{ass}\label{2}Let $\lim\limits_{n\rightarrow\infty}\sum_{h=1}^{\infty}\left(r^{(n)}(h)\right)^{2}<\infty$.
\end{ass}
\begin{ass}\label{3}
Let $k,l,m=1,...,p$. The quantity
\begin{equation*}
\frac{\sqrt{\mathbb{E}\left[\left(\int_{-\infty}^{(i-1)\Delta_{n}} \Delta^{n}_{i}g^{(k,l)}(s)\sigma_{s}^{(l,m)}dW^{(m)}_{s}\right)^{2}\right]}}{\tau_{n}}=\frac{\sqrt{\int_{0}^{\infty} \left(g^{(k,l)}(s+\Delta_{n})-g^{(k,l)}(s)\right)^{2}\mathbb{E}\left[\left(\sigma_{(i-1)\Delta_{n}-s}^{(l,m)}\right)^{2}\right]ds}}{\tau_{n}}
\end{equation*}
is uniformly bounded in $n\in\mathbb{N}$ and $i\in\{1,...,n\}$.
\end{ass}
\noindent  In many situations this last assumption is satisfied, for example when the stochastic volatilities are second order stationary. Furthermore, using the computations of Lemma 1 of \cite{BCD}, it is possible to observe that the above assumption is equivalent to assuming that
\begin{equation*}
\int_{1}^{\infty} \left(g^{(k,l)}(s+\Delta_{n})-g^{(k,l)}(s)\right)^{2}\mathbb{E}\left[\left(\sigma_{(i-1)\Delta_{n}-s}^{(l,m)}\right)^{2}\right]ds<\infty.
\end{equation*}
This assumption is less restrictive than Assumption (4.4) of \cite{BCD} but it is sufficient for their results (see the proofs of Lemma 1 and Lemma 2 in \cite{BCD}). Moreover, since the assumptions of this work and of \cite{BCD} are similar, it is suggested to see Section 4.3 \cite{BCD} for a more detailed discussion of the assumptions.
\subsection{Case I}
\subsubsection{The bivariate case}\label{The joint bivariate case}
We will start with the bivariate case in order to simplify the exposition. The arguments for higher dimensions are the same. Consider the stochastic process $\{\textbf{Y}_{t}\}_{t\in[0,T]}$ defined as
\begin{equation*}
\textbf{Y}_{t}:=\begin{pmatrix}
     Y^{(1)}_{t}   \\
     Y^{(2)}_{t} 
    \end{pmatrix}=\int_{-\infty}^{t}\begin{pmatrix}
   g^{(1,1)}(t-s) & g^{(1,2)}(t-s)  \\
   g^{(2,1)}(t-s) & g^{(2,2)}(t-s)
  \end{pmatrix} \begin{pmatrix}
     \sigma^{(1,1)}_{s} & \sigma^{(1,2)}_{s}  \\
     \sigma^{(2,1)}_{s} & \sigma^{(2,2)}_{s}
    \end{pmatrix} 
  \begin{pmatrix}
     dW^{(1)}_{s}   \\
     dW^{(2)}_{s} 
    \end{pmatrix}+ \begin{pmatrix}
             U^{(1)}_{t}   \\
             U^{(2)}_{t} 
            \end{pmatrix},
\end{equation*}
where $g^{(i,j)}(\cdot), i,j=1,2$ are deterministic functions and $W^{(1)},W^{(2)}$ are two (possibly dependent) $\mathcal{F}_{t}$-Brownian measures on $\mathbb{R}$. From a modelling point of view the dependency of the Brownian measures is not very important since it is always possible to shift it from the Brownian measure to the stochastic volatilities by just rewriting the latter. We have
\begin{equation*}
Y^{(1)}_{t}=\int_{-\infty}^{t}g^{(1,1)}(t-s)\sigma^{(1,1)}_{s}+g^{(1,2)}(t-s)\sigma^{(2,1)}_{s}dW^{(1)}_{s}
\end{equation*}
\begin{equation*}
+\int_{-\infty}^{t}g^{(1,1)}(t-s)\sigma^{(1,2)}_{s}+g^{(1,2)}(t-s)\sigma^{(2,2)}_{s}dW^{(2)}_{s}+U_{t}^{(1)},\quad\text{and}
\end{equation*}
\begin{equation*}
Y^{(2)}_{t}=\int_{-\infty}^{t}g^{(2,1)}(t-s)\sigma^{(1,1)}_{s}+g^{(2,2)}(t-s)\sigma^{(2,1)}_{s}dW^{(1)}_{s}
\end{equation*}
\begin{equation*}
+\int_{-\infty}^{t}g^{(2,1)}(t-s)\sigma^{(1,2)}_{s}+g^{(2,2)}(t-s)\sigma^{(2,2)}_{s}dW^{(2)}_{s}+U_{t}^{(2)}.
\end{equation*}
Let us define, for $k,r,m=1,2$,
\begin{equation*}
\Delta_{i}^{n}Z^{(k,r,m)}:=\int_{(i-1)\Delta_{n}}^{i\Delta_{n}}g^{(k,r)}(i\Delta_{n}-s)\sigma^{(r,m)}_{s}dW^{(m)}_{s}+\int_{-\infty}^{(i-1)\Delta_{n}}\Delta^{n}_{i} g^{(k,r)}(s)\sigma^{(r,m)}_{s}dW^{(m)}_{s},
\end{equation*}
where we recall that $\Delta^{n}_{i} g^{(k,l)}(s)=g^{(k,l)}(i\Delta_{n}-s)-g^{(k,l)}((i-1)\Delta_{n}-s)$.
Then, we have
\begin{equation*}
\Delta_{i}^{n}Y^{(k)}=\Delta_{i}^{n}Z^{(k,1,1)}+\Delta_{i}^{n}Z^{(k,2,1)}+\Delta_{i}^{n}Z^{(k,1,2)}+\Delta_{i}^{n}Z^{(k,2,2)}+\Delta_{i}^{n}U^{(k)}.
\end{equation*}
For $m,r,k=1,2$, let $G^{(k,r;m)}_{t}:=\int_{0}^{t}g^{(k,r)}(t-s)dW_{s}^{(m)}$.
Hence,
\begin{equation*}
\Delta^{n}_{i}G^{(k,r;m)}=\int_{(i-1)\Delta_{n}}^{i\Delta_{n}}g^{(k,r)}(i\Delta_{n}-s)dW^{(m)}+\int_{-\infty}^{(i-1)\Delta_{n}}\Delta^{n}_{i} g^{(k,r)}(s)dW^{(m)}.
\end{equation*}
Further, let $\tau_{n}^{(k,r)}:=\sqrt{\mathbb{E}\left[\left(\Delta^{n}_{1}G^{(k,r;m)}\right)^{2}\right]}\quad\text{and}\quad r_{k,r,m;l,q,w}^{(n)}(h):=\mathbb{E}\left[\dfrac{\Delta^{n}_{1}G^{(k,r;m)}}{\tau_{n}^{(k,r)}}\dfrac{\Delta^{n}_{1+h}G^{(l,q;w)}}{\tau_{n}^{(l,q)}} \right]$.

The Gaussian core $\{\textbf{G}_{t}\}_{t\in[0,T]}$ is here given by $\textbf{G}_{t}=\left(G^{(k,r;m)}_{t} \right)_{m,k,r=1,2}$, and it is a Gaussian process since $W^{(1)}$ and $W^{(2)}$ are jointly Gaussian. Notice that we are working with the separable Hilbert space $\mathcal{H}$ generated by the jointly Gaussian random variables $\left(\frac{\Delta^{n}_{1}G^{(k,r;m)}}{\tau_{n}^{(k,r)}}\right)_{n\geq1,1\leq l\leq \lfloor nt\rfloor,k,r,m\in\{1,2\}}$. Further, observe that $\mathbb{E}\left[\dfrac{\Delta^{n}_{1}G^{(k,r;m)}}{\tau_{n}^{(k,r)}}\dfrac{\Delta^{n}_{1+h}G^{(l,q;w)}}{\tau_{n}^{(l,q)}} \right]=\mathbb{E}\left[\dfrac{\Delta^{n}_{i}G^{(k,r;m)}}{\tau_{n}^{(k,r)}}\dfrac{\Delta^{n}_{i}G^{(l,q;w)}}{\tau_{n}^{(l,q)}} \right]$ for any $i=1,...,\lfloor nt \rfloor$. Before presenting the CLT, we introduce an assumption, similar to assumption (4.10) in \cite{BCD}, that controls the asymptotic behaviour of the drift process.
\begin{ass}\label{AU1} Let 
$\frac{1}{\sqrt{n}}\sum_{i=1}^{\lfloor nt\rfloor}\frac{\Delta_{i}^{n}Z^{(k,r,m)}}{\tau^{(k,r)}_{n}}\Delta_{i}^{n}U^{(l)}\stackrel{u.c.p.}{\rightarrow}0$ and $\frac{1}{\sqrt{n}}\sum_{i=1}^{\lfloor nt\rfloor}\Delta_{i}^{n}U^{(k)}\Delta_{i}^{n}U^{(l)}\stackrel{u.c.p.}{\rightarrow}0$ for any $k,l,r,m=1,...,p$.
\end{ass}
\begin{thm}\label{I}
Under Assumptions \ref{1}, \ref{2} and \ref{3} applied to $\tau_{n}^{(l,q)},r_{k,r;l,q}^{(n)}$ for $l,q,k,r=1,2$, and Assumption \ref{AU1}, we have the following stable convergence.
\begin{equation*}
\Bigg\{\sqrt{n}\Bigg[\frac{1}{n}\sum_{i=1}^{\lfloor nt\rfloor}\Bigg(\sum_{r,m=1}^{2}\frac{\Delta_{i}^{n}Z^{(k,r,m)}}{\tau^{(k,r)}_{n}}+\Delta_{i}^{n}U^{(k)}\Bigg)\Bigg(\sum_{q,w=1}^{2}\frac{\Delta_{i}^{n}Z^{(l,q,w)}}{\tau^{(l,q)}_{n}}+\Delta_{i}^{n}U^{(l)}\Bigg)
\end{equation*}
\begin{equation*}
-\sum_{r,m,q=1}^{2}\mathbb{E}\left[\frac{\Delta_{1}^{n}G^{(k,r,m)}}{\tau^{(k,r)}_{n}}\frac{\Delta_{1}^{n}G^{(l,q,m)}}{\tau^{(l,q)}_{n}}\right]\int_{0}^{t}\sigma^{(r,m)}_{s}\sigma^{(q,m)}_{s}ds\Bigg]_{k,l=1,...,p}\Bigg\}_{t\in[0,T]}\stackrel{st}{\rightarrow}\Bigg\{\int_{0}^{t}V_{s}\textit{\textbf{D}}^{1/2}dB_{s}\Bigg\}_{t\in[0,T]}
\end{equation*}
in $\mathcal{D}([0,T],\mathbb{R}^{4})$, where $\textit{\textbf{D}}$ and $V_{s}$ are introduced in the Appendix, and $B_{s}$ is a 64-dimensional Brownian motion.
\end{thm}
\begin{proof}
Let us assume for now that the drift process $\textbf{U}_{t}=0$ for any $t\in[0,T]$. 
We can split our formulation in two components $A_{n}+C_{n}$, where $A_{n}$ contains the elements that go to zero, while $C_{n}$ contains the ones that do not converge to zero. In this proof we will use the so-called \textit{blocking technique}, see \cite{BCD}, \cite{CHPP} and \cite{Andrea1} for details. Let us first focus on $C_{n}$, which is defined as
\begin{equation*}
C_{n}:=\sqrt{n}\Bigg(\frac{1}{n}\sum_{j=1}^{\lfloor pt\rfloor}\sum_{i\in I_{p}(j)}\sum_{r,m,q,w=1}^{2}\Bigg(\frac{\Delta^{n}_{i}G^{(k,r;m)}}{\tau_{n}^{(k,r)}}\frac{\Delta^{n}_{i}G^{(l,q;w)}}{\tau_{n}^{(l,q)}}
\end{equation*}
\begin{equation*}
-\mathbb{E}\left[\frac{\Delta^{n}_{i}G^{(k,r;m)}}{\tau_{n}^{(k,r)}}\frac{\Delta^{n}_{i}G^{(l,q;w)}}{\tau_{n}^{(l,q)}}\right]\Bigg)\sigma^{(r,m)}_{(j-1)\Delta_{p}}\sigma^{(q,w)}_{(j-1)\Delta_{p}} \Bigg)_{k,l=1,2},
\end{equation*}
where $I_{p}(j)=\left\{ i\big|\frac{i}{n}\in\left(\frac{j-1}{p},\frac{j}{p}\right]\right\}$, which can be rewritten as
\begin{equation*}
C_{n}=\frac{1}{\sqrt{n}}\sum_{j=1}^{\lfloor pt\rfloor}\sum_{i\in I_{p}(j)}V_{(j-1)\Delta_{p}}\left(\frac{\Delta^{n}_{i}G^{(k,r;m)}}{\tau_{n}^{(k,r)}}\frac{\Delta^{n}_{i}G^{(l,q;w)}}{\tau_{n}^{(l,q)}}-\mathbb{E}\left[\frac{\Delta^{n}_{i}G^{(k,r;m)}}{\tau_{n}^{(k,r)}}\frac{\Delta^{n}_{i}G^{(l,q;w)}}{\tau_{n}^{(l,q)}}\right]\right)_{r,m,q,w,k,l=1,2},
\end{equation*}
where 
\begin{equation*}
V_{(j-1)\Delta_{p}}=\begin{pmatrix}
   \sigma_{(j-1)\Delta_{p}} & \textbf{0} & \textbf{0} & \textbf{0} \\ \textbf{0} &
   \sigma_{(j-1)\Delta_{p}} & \textbf{0} & \textbf{0} \\  \textbf{0} & \textbf{0} & \sigma_{(j-1)\Delta_{p}} & \textbf{0} \\ \textbf{0} & \textbf{0} & \textbf{0} & \sigma_{(j-1)\Delta_{p}}

  \end{pmatrix},
\end{equation*}
where $\sigma_{(j-1)\Delta_{p}}=\left(\sigma^{(r,m)}_{(j-1)\Delta_{p}}\sigma^{(q,w)}_{(j-1)\Delta_{p}}\right)_{r,m,q,w=1,2}^{\top}$ (hence it is a row vector of 16 elements), and $\textbf{0}$ is a row vector of 16 elements containing only zeros. Hence, $V_{(j-1)\Delta_{p}}$ is a $4\times 64$ matrix.
Now, by Theorem \ref{CLT-M}, we have that 
\begin{equation*}
\left\{\frac{1}{\sqrt{n}}\sum_{i=1}^{\lfloor nt\rfloor}\left(\frac{\Delta^{n}_{i}G^{(k,r;m)}}{\tau_{n}^{(k,r)}}\frac{\Delta^{n}_{i}G^{(l,q;w)}}{\tau_{n}^{(l,q)}}-\mathbb{E}\left[\frac{\Delta^{n}_{i}G^{(k,r;m)}}{\tau_{n}^{(k,r)}}\frac{\Delta^{n}_{i}G^{(l,q;w)}}{\tau_{n}^{(l,q)}}\right]\right)_{r,m,q,w,k,l=1,2}\right\}_{t\in[0,T]}
\end{equation*}
\begin{equation*}
\stackrel{st}{\rightarrow}\left\{\textit{\textbf{D}}^{1/2}B_{t}\right\}_{t\in[0,T]},\quad\text{as}\quad n\rightarrow\infty
\end{equation*}
in $\mathcal{D}([0,T],\mathbb{R}^{64})$, where the symmetric matrix $\textit{\textbf{D}}^{1/2}$ is a $64\times64$ matrix and $B_{t}$ is a $64$-dimensional Brownian motion. In particular, in order to define the elements of the matrix $\textit{\textbf{D}}$ we proceed with the following association of $(z,y)$ to $((r_{z},m_{z},q_{z},w_{z},k_{z},l_{z}),(r_{y},m_{y},q_{y},w_{y},k_{y},l_{y}))$. Let $\nu(r,m,q,w)$ be the set of all the possible combinations of $r,m,q,w\in\{1,2\}$ and let $\nu_{s}(r,m,q,w)$ determines the $s$ element of $\nu(r,m,q,w)$. It is possible to see that $\nu(r,m,q,w)$ contains $2^{4}$ elements, hence $s\in\{1,...,2^{4}\}$. Now, we impose the following: for $z=y=s\leftrightarrow (\nu_{s}(r,m,q,w),1,1)$; $z=y=2^{4}+1\leftrightarrow (\nu_{1}(r,m,q,w),1,2)$;...; $z=y=2^{4}+s\leftrightarrow (\nu_{s}(r,m,q,w),1,2)$;...; $z=y=2^{5}+1\leftrightarrow (\nu_{1}(r,m,q,w),2,1)$;...; $z=y=2^{5}+s\leftrightarrow (\nu_{s}(r,m,q,w),2,1)$;...; $z=y=2^{5}+2^{4}+1\leftrightarrow (\nu_{1}(r,m,q,w),2,2)$;...; $z=y=2^{5}+2^{4}+s\leftrightarrow (\nu_{s}(r,m,q,w),2,2)$;...; $z=y=2^{6}\leftrightarrow (\nu_{2^{4}}(r,m,q,w),2,2)$. This can be written compactly as 
\begin{equation*}
(z,y)\leftrightarrow\Bigg(\Big(\nu_{z-\lfloor\frac{z-1}{2^{4}}\rfloor 2^{4}}(r,m,q,w),\lfloor\frac{\lfloor\frac{z-1}{2^{4}}\rfloor}{2}\rfloor+1,\lfloor\frac{z-1}{2^{4}}\rfloor+1-2\lfloor\frac{\lfloor\frac{z-1}{2^{4}}\rfloor}{2}\rfloor\Big)
\end{equation*}
\begin{equation*}
,\Big(\nu_{y-\lfloor\frac{y-1}{2^{4}}\rfloor 2^{4}}(r,m,q,w),\lfloor\frac{\lfloor\frac{y-1}{2^{4}}\rfloor}{2}\rfloor+1,\lfloor\frac{y-1}{2^{4}}\rfloor+1-2\lfloor\frac{\lfloor\frac{y-1}{2^{4}}\rfloor}{2}\rfloor\Big)\Bigg).
\end{equation*}
Moreover, here it becomes evident why we need Theorem 2.6 to hold for jointly Gaussian Brownian measures and not just for independent ones: this is because the vector of Brownian measures is not composed of independent Brownian measures but rather of the same Brownian measures recurring repeatedly .

Continuing with the proof, we observe that by using the properties of the stable convergence and the assumption on the $\mathcal{F}$-measurability of the $\sigma$s, we obtain for fixed $p$
\begin{equation*}
C_{n}\stackrel{st}{\rightarrow}\left\{\sum_{j=1}^{\lfloor pt\rfloor}V_{(j-1)\Delta_{p}}\textit{\textbf{D}}^{1/2}\left(B_{j\Delta_{p}}-B_{(j-1)\Delta_{p}}\right)\right\}_{t\in[0,T]},\quad\text{as}\quad n\rightarrow\infty
\end{equation*}
in $\mathcal{D}([0,T],\mathbb{R}^{4})$ where the dimensions are $4\times64$, $64\times64$ and $64\times 1$, respectively. The convergence in this space is implied by the convergence in $\mathcal{D}([0,T],\mathbb{R}^{64})$. In addition, since the stochastic volatilities are c\`{a}dl\`{a}g we have
\begin{equation*}
\sum_{j=1}^{\lfloor pt\rfloor}V_{(j-1)\Delta_{p}}\textit{\textbf{D}}^{1/2}\left(B_{j\Delta_{p}}-B_{(j-1)\Delta_{p}}\right)\stackrel{P}{\rightarrow}\int_{0}^{t}V_{s}\textit{\textbf{D}}^{1/2}dB_{s},
\end{equation*}
as $p\rightarrow\infty$. From this we obtain the stable convergence of $C_{n}$.

Concerning $A_{(l,k),n}$ we have the same arguments of Theorem \cite{Andrea1} and Theorem 4 of \cite{BCD}. This is because we can focus on the single elements
\begin{equation*}
\sqrt{n}\Bigg[\frac{1}{n}\sum_{i=1}^{\lfloor nt\rfloor}\frac{\Delta_{i}^{n}Z^{(k,r,m)}}{\tau^{(k,r)}_{n}}\frac{\Delta_{i}^{n}Z^{(l,q,w)}}{\tau^{(l,1)}_{n}}-\mathbb{E}\left[\frac{\Delta^{n}_{i}G^{(k,r,m)}}{\tau_{n}^{(k,r)}}\frac{\Delta^{n}_{i}G^{(l,q,w)}}{\tau_{n}^{(l,q)}} \right]\int_{0}^{t}\sigma^{(r,m)}_{s}\sigma^{(q,w)}_{s}ds\Bigg]
\end{equation*}
and directly apply their arguments using the assumptions of this theorem. Notice that when $m\neq w$, then we have $\mathbb{E}\left[\frac{\Delta^{n}_{i}G^{(k,r,m)}}{\tau_{n}^{(k,r)}}\frac{\Delta^{n}_{i}G^{(l,q,w)}}{\tau_{n}^{(l,q)}} \right]=0$. 
In particular, for each $(l,k)$ we have that $A_{(l,k),n}$ converges to zero in distribution in $\mathcal{D}([0,T],\mathbb{R}^{4})$, which implies that they converge jointly to zero stably in distribution. Now, since $C_{n}$ converges stably and $A_{n}$ converges stably to zero we have that they jointly converge stably. This concludes the proof for the case $\textbf{U}_{t}=\textbf{0}$, where $\textbf{0}$ is a vector of zeros.

Now consider $\textbf{U}_{t}\neq\textbf{0}$. In order to get the stated result we need to prove that the following elements converge in u.c.p. to zero, so that the stable convergence obtained so far in this proof remains the same. These elements are
\begin{equation*}
\sqrt{n}\Bigg[\frac{1}{n}\sum_{i=1}^{\lfloor nt\rfloor}\left(\frac{\Delta_{i}^{n}Z^{(k,1,1)}}{\tau^{(k,1)}_{n}}+\frac{\Delta_{i}^{n}Z^{(k,2,1)}}{\tau^{(k,2)}_{n}}+\frac{\Delta_{i}^{n}Z^{(k,1,2)}}{\tau^{(k,1)}_{n}}+\frac{\Delta_{i}^{n}Z^{(k,2,2)}}{\tau^{(k,2)}_{n}}\right)\Delta_{i}^{n}U^{(l)}
\end{equation*}
\begin{equation*}
+\left(\frac{\Delta_{i}^{n}Z^{(l,1,1)}}{\tau^{(l,1)}_{n}}+\frac{\Delta_{i}^{n}Z^{(l,2,1)}}{\tau^{(l,2)}_{n}}+\frac{\Delta_{i}^{n}Z^{(l,1,2)}}{\tau^{(l,1)}_{n}}+\frac{\Delta_{i}^{n}Z^{(l,2,2)}}{\tau^{(l,2)}_{n}}\right)\Delta_{i}^{n}U^{(k)}+\Delta_{i}^{n}U^{(k)}\Delta_{i}^{n}U^{(l)}\Bigg]_{k,l=1,2}.
\end{equation*}
Thanks to Assumption \ref{AU1} (with $p=2$) they go to zero in u.c.p. component wise (\textit{i.e}.~for fixed $k,l$) and, hence, jointly. Thus, using the properties of the stable convergence the proof is complete.
\end{proof}
\subsubsection{The multidimensional case and the \textit{vech} formulation}
It is possible to obtain a multidimensional version of Theorem \ref{I}. The reason why we presented the bivariate case first is because in this way we simplified the notations and, hence, facilitated the understanding of the arguments, which are the same for the multivariate case.
\begin{thm}\label{I-Multi}
Under Assumptions \ref{1}, \ref{2} and \ref{3} applied to $\tau_{n}^{(l,q)},r_{k,r;l,q}^{(n)}$ for $l,q,k,r=1,...,p$, and Assumption \ref{AU1}, we have the following stable convergence.
\begin{equation*}
\Bigg\{\sqrt{n}\Bigg[\frac{1}{n}\sum_{i=1}^{\lfloor nt\rfloor}\Bigg(\sum_{r,m=1}^{p}\frac{\Delta_{i}^{n}Z^{(k,r,m)}}{\tau^{(k,r)}_{n}}+\Delta_{i}^{n}U^{(k)}\Bigg)\Bigg(\sum_{q,w=1}^{p}\frac{\Delta_{i}^{n}Z^{(l,q,w)}}{\tau^{(l,q)}_{n}}+\Delta_{i}^{n}U^{(l)}\Bigg)
\end{equation*}
\begin{equation*}
-\sum_{r,m,q=1}^{p}\mathbb{E}\left[\frac{\Delta_{1}^{n}G^{(k,r,m)}}{\tau^{(k,r)}_{n}}\frac{\Delta_{1}^{n}G^{(l,q,m)}}{\tau^{(l,q)}_{n}}\right]\int_{0}^{t}\sigma^{(r,m)}_{s}\sigma^{(q,m)}_{s}ds\Bigg]_{k,l=1,...,p}\Bigg\}_{t\in[0,T]}\stackrel{st}{\rightarrow}\Bigg\{\int_{0}^{t}V_{s}\textit{\textbf{D}}^{1/2}dB_{s}\Bigg\}_{t\in[0,T]}
\end{equation*}
in $\mathcal{D}([0,T],\mathbb{R}^{p^{2}})$, where $\textit{\textbf{D}}$ and $V_{s}$ are introduced in the Appendix, and $B_{s}$ is a $p^{6}$-dimensional Brownian motion.
\end{thm}
\begin{proof}
It follows from the same arguments as the ones used in the proof of Theorem \ref{I}.
\end{proof}
It is possible to obtain a \textit{vech} formulation of our results, thus reducing their dimensions without losing any information. This is possible because of the symmetry of our object of study, that is there is no difference between $\frac{\Delta_{i}^{n}G^{(k,r,m)}}{\tau^{(k,r)}_{n}}\frac{\Delta_{i}^{n}G^{(l,q,w)}}{\tau^{(l,q)}_{n}}$ and $\frac{\Delta_{i}^{n}G^{(l,q,w)}}{\tau^{(l,q)}_{n}}\frac{\Delta_{i}^{n}G^{(k,r,m)}}{\tau^{(k,r)}_{n}}$ and between $\frac{\Delta_{i}^{n}Z^{(k,r,m)}}{\tau^{(k,r)}_{n}}\frac{\Delta_{i}^{n}Z^{(l,q,w)}}{\tau^{(l,q)}_{n}}$ and $\frac{\Delta_{i}^{n}Z^{(l,q,w)}}{\tau^{(l,q)}_{n}}\frac{\Delta_{i}^{n}Z^{(k,r,m)}}{\tau^{(k,r)}_{n}}$. Hence, we have the following formulation of our results.
\begin{co}\label{I-Vech}
Under Assumptions \ref{1}, \ref{2} and \ref{3} applied to $\tau_{n}^{(l,q)},r_{k,r;l,q}^{(n)}$ for $l,q,k,r=1,...,p$, and Assumption \ref{AU1}, we have the following stable convergence.
\begin{equation*}
\Bigg\{\sqrt{n}\Bigg[\frac{1}{n}\sum_{i=1}^{\lfloor nt\rfloor}\Bigg(\sum_{r,m=1}^{p}\frac{\Delta_{i}^{n}Z^{(k,r,m)}}{\tau^{(k,r)}_{n}}+\Delta_{i}^{n}U^{(k)}\Bigg)\Bigg(\sum_{q,w=1}^{p}\frac{\Delta_{i}^{n}Z^{(l,q,w)}}{\tau^{(l,q)}_{n}}+\Delta_{i}^{n}U^{(l)}\Bigg)
\end{equation*}
\begin{equation*}
-\mathbb{E}\left[\frac{\Delta_{1}^{n}G^{(k,r,m)}}{\tau^{(k,r)}_{n}}\frac{\Delta_{1}^{n}G^{(l,q,w)}}{\tau^{(l,q)}_{n}}\right]\int_{0}^{t}\sigma^{(r,m)}_{s}\sigma^{(q,w)}_{s}ds\Bigg]_{k=1,...,p;l\leq k}\Bigg\}_{t\in[0,T]}\stackrel{st}{\rightarrow}\Bigg\{\int_{0}^{t}V_{s}\textit{\textbf{D}}^{1/2}dB_{s}\Bigg\}_{t\in[0,T]}
\end{equation*}
in $\mathcal{D}([0,T],\mathbb{R}^{\frac{p}{2}(p+1)})$, where $\textit{\textbf{D}}$ and $V_{s}$ are introduced in the Appendix, and $B_{s}$ is a $\frac{p^{5}}{2}(p+1)$-dimensional Brownian motion.
\end{co}
\begin{proof}
It follows from the same arguments as the ones used in the proof of Theorem \ref{I}.
\end{proof}
\begin{rem}
Notice that $\frac{p^{5}}{2}(p+1)$ comes from $p^{4}(p^{2}-\sum_{j=1}^{p-1}j)$ where $\sum_{j=1}^{p-1}j$ indicates that we are not considering the strictly upper (or lower) triangular elements of the $p\times p$ matrix. Moreover, for the remaining sections and chapter we will always adopt the \textnormal{vech} version of our results.
\end{rem}
\subsection{Case II: first scenario}
Despite the process $\textbf{Y}_{t}$ being the same as in the previous section, we introduce a new formulation for the $\tau$. This formulation is in line with the one presented in Section \ref{A second case}. Furthermore, in this section we present and prove the results for one of the two versions of the multivariate BSS process introduced in Definition \ref{BSS}. In the next section we will do the same for the other version.

Consider the stochastic process $\{\textbf{Y}_{t}\}_{t\in[0,T]}=\{(Y_{t}^{(1)},...,Y_{t}^{(p)})\}_{t\in[0,T]}$ given by
\begin{equation}\label{Y-secondcase}
\textbf{Y}_{t}=\int_{-\infty}^{t} \begin{pmatrix}
             g^{(1,1)}(t-s)  &\cdots  & g^{(1,p)}(t-s) \\ \vdots &
             \ddots  & \vdots \\ g^{(p,1)}(t-s) & \cdots & g^{(p,p)}(t-s)       
            \end{pmatrix} 
            \begin{pmatrix}
                         \sigma^{(1,1)}_{s}  &\cdots  & \sigma^{(1,p)}_{s} \\ \vdots &
                         \ddots  & \vdots \\ \sigma^{(p,1)}_{s} & \cdots & \sigma^{(p,p)}_{s}       
                        \end{pmatrix} 
  \begin{pmatrix}
     dW^{(1)}_{s}   \\\vdots\\
     dW^{(p)}_{s}
    \end{pmatrix}+ \begin{pmatrix}
             U^{(1)}_{t}   \\\vdots\\
             U^{(p)}_{t} 
            \end{pmatrix}.
\end{equation}
Assume that the $\mathcal{F}_{t}$-Brownian measures are all \textit{independent} of each other. Let us define, for $k=1,...,p$,  $\bar{\tau}_{n}^{(k)}:=\max\limits_{r=1,...,p}O\left( \tau_{n}^{(k,r)}\right)$. For example, we may take $\bar{\tau}_{n}^{(k)}:=\sqrt{\mathbb{E}\left[\left(\sum_{r=1}^{p}\Delta^{n}_{1}G^{(k,r,r)}_{t}\right)^{2}\right]}$ or $\bar{\tau}_{n}^{(k)}:=\max\limits_{r=1,...,p}\sqrt{\mathbb{E}\left[\left(\Delta^{n}_{1}G^{(k,r,r)}_{t}\right)^{2}\right]}$. The Gaussian core is given by $\{\textbf{G}_{t}\}_{t\in[0,T]}$ with $\textbf{G}_{t}=\left(G^{(k,r,m)}_{t} \right)_{k,r,m=1,...,p}=\left(\int_{0}^{t}g^{(k,r)}(t-s)dW_{s}^{(m)} \right)_{k,r,m=1,...,p}$.  Moreover, let for $h\in\mathbb{N}$, $\bar{r}_{k,r,m;l,q,w}^{(n)}(h):=\mathbb{E}\left[\dfrac{\Delta^{n}_{1}G^{(k,r,m)}}{\bar{\tau}_{n}^{(k)}}\dfrac{\Delta^{n}_{1+h}G^{(l,q,w)}}{\bar{\tau}_{n}^{(l)}} \right]$, and observe that if $m\neq w$ then $\bar{r}_{k,r,m;l,q,w}^{(n)}(h)=0$. 
\begin{ass}\label{AU2} Let 
$\frac{1}{\sqrt{n}}\sum_{i=1}^{\lfloor nt\rfloor}\frac{\Delta_{i}^{n}(Y^{(k)}-U^{(k)})}{\bar{\tau}^{(k)}_{n}}\frac{\Delta_{i}^{n}U^{(l)}}{\bar{\tau}^{(l)}_{n}}\stackrel{u.c.p.}{\rightarrow}0$ and $\frac{1}{\sqrt{n}}\sum_{i=1}^{\lfloor nt\rfloor}\frac{\Delta_{i}^{n}U^{(k)}}{\bar{\tau}^{(k)}_{n}}\frac{\Delta_{i}^{n}U^{(l)}}{\bar{\tau}^{(l)}_{n}}\stackrel{u.c.p.}{\rightarrow}0$ for any $k,l=1,...,p$.
\end{ass}
\begin{thm}\label{NEW-I-Vech}
Under Assumptions \ref{1}, \ref{2} and \ref{3} applied to $\bar{\tau}_{n}^{(l)},\bar{r}_{k,r,m;l,q,m}^{(n)}$ for $l,q,k,r,m=1,...,p$, and Assumption \ref{AU2}, we have the following stable convergence.
\begin{equation*}
\Bigg\{\sqrt{n}\Bigg[\frac{1}{n}\sum_{i=1}^{\lfloor nt\rfloor}\frac{\Delta_{i}^{n}Y^{(k)}}{\bar{\tau}^{(k)}_{n}}\frac{\Delta_{i}^{n}Y^{(l)}}{\bar{\tau}^{(l)}_{n}}-\sum_{r,m,q=1}^{p}\mathbb{E}\left[\frac{\Delta_{1}^{n}G^{(k,r,m)}}{\bar{\tau}^{(k)}_{n}}\frac{\Delta_{1}^{n}G^{(l,q,m)}}{\bar{\tau}^{(l)}_{n}}\right]\int_{0}^{t}\sigma^{(r,m)}_{s}\sigma^{(q,m)}_{s}ds\Bigg]_{k=1,...,p;l\leq k}\Bigg\}_{t\in[0,T]}
\end{equation*}
\begin{equation*}
\stackrel{st}{\rightarrow}\Bigg\{\int_{0}^{t}V_{s}\textit{\textbf{D}}^{1/2}dB_{s}\Bigg\}_{t\in[0,T]}
\end{equation*}
in $\mathcal{D}([0,T],\mathbb{R}^{\frac{p}{2}(p+1)})$, where $\textit{\textbf{D}}$ and $V_{s}$ are introduced in the Appendix, and $B_{s}$ is a $\frac{p^{5}}{2}(p+1)$-dimensional Brownian motion.
\end{thm}
\begin{proof}
It follows from the same arguments as the ones used in the proof of Theorem \ref{I}, the only difference is the use of Theorem \ref{2-CLT-M} instead of Theorem \ref{CLT-M}. In particular, the elements that converge to zero (which we call $A_{n}$ in the proof of Theorem \ref{I}) still converge to zero since we are using a larger denominator. For the other elements (which we call $C_{n}$) we do have the stated convergence due to Theorem \ref{2-CLT-M}.
\end{proof}
\subsection{Case II: second scenario}\label{The second case}
In this section we present and prove the results for the other version of the multivariate BSS process introduced in Definition \ref{BSS}. In addition, at the price of a simple assumption on the stochastic volatilities $\sigma$s, this new form allows for a definition of $\tau_{n}$ in terms of the BSS process and not in terms of Gaussian core.

Consider the stochastic process $\{\textbf{X}_{t}\}_{t\in[0,T]}=\{(X_{t}^{(1)},...,X_{t}^{(p)})\}_{t\in[0,T]}$ given by
\begin{equation}\label{Y-thirdcase}
\textbf{X}_{t}=\int_{-\infty}^{t} \begin{pmatrix}
             g^{(1,1)}(t-s)\sigma^{(1,1)}_{s}  &\cdots  & g^{(1,p)}(t-s)\sigma^{(1,p)}_{s} \\ \vdots &
             \ddots  & \vdots \\ g^{(p,1)}(t-s)\sigma^{(p,1)}_{s} & \cdots & g^{(p,p)}(t-s)\sigma^{(p,p)}_{s}       
            \end{pmatrix} 
  \begin{pmatrix}
     dW^{(1)}_{t}   \\\vdots\\
     dW^{(p)}_{t} 
    \end{pmatrix}+ \begin{pmatrix}
         U^{(1)}_{s}   \\\vdots\\
         U^{(p)}_{s} 
        \end{pmatrix}.
\end{equation}
Assume that the $\mathcal{F}_{t}$-Brownian measures are all \textit{independent} of each other and of the drift process, and that $\mathbb{E}\left[(\sigma_{s}^{(k,m)})^{2}\right]>C$ for any $k,m=1,...,p$ and $s\in[0,T]$, \textit{e.g}.~the volatilities are second order stationary. Let us define, for $k=1,...,p$, $\tilde{\tau}_{n}^{(k)}:=\sqrt{\mathbb{E}\left[\left(\Delta^{n}_{1}X^{(k)}\right)^{2}\right]}$. Following Example \ref{Example1}, it is possible to observe that $\tilde{\tau}_{n}^{(k)}$ is similar to $\tau_{n}^{(\alpha_{h})}$ introduced in $(\ref{tau-secondcase})$ (except for the drift component). The Gaussian core is given by $\{\textbf{G}_{t}\}_{t\in[0,T]}$ with $\textbf{G}_{t}=\left(G^{(k,m)}_{t} \right)_{k,m=1,...,p}=\left(\int_{0}^{t}g^{(k,m)}(t-s)dW_{s}^{(m)} \right)_{k,m=1,...,p}$ and the partition is in the $k$ variables, namely we split $\{1,...,p\}^{2}$ into $\{1\}\times\{1,...,p\},...,\{p\}\times\{1,...,p\}$. Notice that for each element of the partition the associated Brownian measures are independent of each other. Further, we define $\tilde{r}_{k,m;l,w}^{(n)}(h):=\mathbb{E}\left[\dfrac{\Delta^{n}_{1}G^{(k,m)}}{\tilde{\tau}_{n}^{(k)}}\dfrac{\Delta^{n}_{1+h}G^{(l,w)}}{\tilde{\tau}_{n}^{(l)}} \right]$; observe that if $m\neq w$ then $\tilde{r}_{k,m;l,w}^{(n)}(h)=0$ and that the rate of $\tilde{\tau}_{n}^{(k)}$ is greater than or equal to the order of $\tau_{n}^{(k,r)}$ $n\rightarrow\infty$  for any $r=1,..,p$.
\begin{thm}\label{I-Vech-2}
Under Assumptions \ref{1}, \ref{2} and \ref{3} applied to $\tilde{\tau}_{n}^{(l)},\tilde{r}_{k,m;l,m}^{(n)}$ and Assumption \ref{AU2} applied to $X^{(k)}$ for $l,k,m=1,...,p$, we have the following stable convergence.
\begin{equation*}
\Bigg\{\sqrt{n}\Bigg[\frac{1}{n}\sum_{i=1}^{\lfloor nt\rfloor }\frac{\Delta_{i}^{n}X^{(k)}}{\tilde{\tau}_{n}^{(k)}}\frac{\Delta_{i}^{n}X^{(l)}}{\tilde{\tau}_{n}^{(l)}}-\sum_{m=1}^{p}\mathbb{E}\left[\frac{\Delta_{1}^{n}G^{(k,m)}}{\tilde{\tau}_{n}^{(k)}}\frac{\Delta_{1}^{n}G^{(l,m)}}{\tilde{\tau}_{n}^{(l)}}\right]\int_{0}^{t}\sigma^{(k,m)}_{s}\sigma^{(l,m)}_{s}ds\Bigg]_{k=1,...,p;l\leq k}\Bigg\}_{t\in[0,T]}
\end{equation*}
\begin{equation*}
\stackrel{st}{\rightarrow}\Bigg\{\int_{0}^{t}V_{s}\textit{\textbf{D}}^{1/2}dB_{s}\Bigg\}_{t\in[0,T]}
\end{equation*}
in $\mathcal{D}([0,T],\mathbb{R}^{\frac{p}{2}(p+1)})$, where $\textit{\textbf{D}}$ and $V_{s}$ are introduced in the Appendix and $B_{s}$ is a $\frac{p^{3}}{2}(p+1)$-dimensional Brownian motion.
\end{thm}
\begin{proof}
It follows from the arguments of Theorem \ref{I} and the results of Theorem \ref{2-CLT-M}.
\end{proof}
\begin{rem}
It is possible to obtain a similar result to Theorem \ref{I-Vech-2} using a formulation for the $\tau$s which is not written in terms of the $\sigma$s, (hence without the assumption of the independence of the Brownian measures and on the expected squared value of the $\sigma$s required for Theorem \ref{I-Vech-2}). Indeed, we could have proceeded as in the previous section (\textit{i.e}.~Case II first scenario). The decision to present this setting comes from the novel possibility, given by the multidimensional structure of the BSS process presented in this section (see (\ref{Y-thirdcase})), to write $\tau$ in terms of $X$. Thus, in case we have an estimate of $\mathbb{E}\left[\left(\Delta^{n}_{1}X^{(k)}\right)^{2}\right]$, Theorem \ref{I-Vech-2} becomes a feasible central limit theorem.
\end{rem}
\section{Weak laws of large numbers}\label{WEAK}
From the central limit theorems proved in the previous sections it possible to derive the weak law of large numbers (WLLN). First, we present the following lemma which follows from the definition of uniform convergence in probability.
\begin{lem}\label{lemma}
Consider $p$ real valued stochastic processes $\{H^{(1)}_{t}\}_{t\in[0,T]},...,\{H^{(p)}_{t}\}_{t\in[0,T]}$. Further, consider sequences of random variables $\{H^{(i),n}_{t}\}_{t\in[0,T]}$ such that $H^{(i),n}_{t}\stackrel{u.c.p.}{\rightarrow}H^{(i)}_{t}$, for $i=1,...,p$ and for any $t\in[0,T]$, then $(H^{(1),n}_{t},...,H^{(p),n}_{t})\stackrel{u.c.p.}{\rightarrow}(H^{(1)}_{t},...,H^{(p)}_{t})$.
\begin{proof}
Consider the case $p=2$, since for $p>2$ the proof uses exactly the same arguments. Let $t\in[0,T]$, we have
\begin{equation*}
\mathbb{P}\left(\sup\limits_{s\in[0,t]}|(H^{(1),n}_{s},H^{(2),n}_{s})-(H^{(1)}_{s},H^{(2)}_{s})|>\epsilon\right)\leq \mathbb{P}\left(\sup\limits_{s\in[0,t]}|H^{(1),n}_{s}-H^{(1)}_{s}|+|H^{(2),n}_{s}-H^{(2)}_{s}|>\epsilon\right)
\end{equation*}
\begin{equation*}
\leq \mathbb{P}\left(\sup\limits_{s\in[0,t]}|H^{(1),n}_{s}-H^{(1)}_{s}|>\epsilon\right)+\mathbb{P}\left(\sup\limits_{s\in[0,t]}|H^{(2),n}_{s}-H^{(2)}_{s}|>\epsilon\right)\rightarrow 0,\quad\text{as $n\rightarrow\infty$.}
\end{equation*}
\end{proof}
\end{lem}
We will derive the WLLN for the first scenario of Case II and point out that using the same arguments it is possible to obtain similar results for all the CLT presented in this work. \\Let $\bar{r}_{k,r,m;l,q,w}(h):=\lim\limits_{n\rightarrow\infty}\bar{r}^{(n)}_{k,m;l,w}(h)$ and $\tilde{r}_{k,m;l,w}(h):=\lim\limits_{n\rightarrow\infty}\bar{r}^{(n)}_{k,m;l,w}(h)$, for $k,r,m,l,q,w=1,...,p$.
\\ 
\begin{thm}\label{LLN}
Let the assumptions of Theorem \ref{NEW-I-Vech} hold. Then we have
\begin{equation*}
\left(\frac{1}{n}\sum_{i=1}^{\lfloor nt\rfloor }\frac{\Delta_{i}^{n}Y^{(k)}}{\bar{\tau}^{(k)}_{n}}\frac{\Delta_{i}^{n}Y^{(l)}}{\bar{\tau}^{(l)}_{n}}\right)_{k=1,...,p;l\leq k}\stackrel{u.c.p.}{\rightarrow}\left(\sum_{r,m,q=1}^{p}\bar{r}_{k,r,m;l,q,m}(0)\int_{0}^{t}\sigma^{(r,m)}_{s}\sigma^{(q,m)}_{s}ds\right)_{k=1,...,p;l\leq k}
\end{equation*}
\end{thm}
\begin{proof}
Fix $l,k\in\{1,...,p\}$. From Theorem \ref{NEW-I-Vech} we have that 
\begin{equation*}
\sqrt{n}\Bigg[\frac{1}{n}\sum_{i=1}^{\lfloor nt\rfloor}\frac{\Delta_{i}^{n}Y^{(k)}}{\tau^{(k)}_{n}}\frac{\Delta_{i}^{n}Y^{(l)}}{\tau^{(l)}_{n}}-\sum_{r,m,q=1}^{p}\mathbb{E}\left[\frac{\Delta_{1}^{n}G^{(k,r,m)}}{\bar{\tau}^{(k)}_{n}}\frac{\Delta_{1}^{n}G^{(l,q,m)}}{\bar{\tau}^{(l)}_{n}}\right]\int_{0}^{t}\sigma^{(r,m)}_{s}\sigma^{(q,m)}_{s}ds\Bigg]
\end{equation*}
converges in distribution. Now, by Slutsky's theorem we have that  
\begin{equation*}
\Bigg[\frac{1}{n}\sum_{i=1}^{\lfloor nt\rfloor}\frac{\Delta_{i}^{n}Y^{(k)}}{\tau^{(k)}_{n}}\frac{\Delta_{i}^{n}Y^{(l)}}{\tau^{(l)}_{n}}-\sum_{r,m,q=1}^{p}\mathbb{E}\left[\frac{\Delta_{1}^{n}G^{(k,r,m)}}{\bar{\tau}^{(k)}_{n}}\frac{\Delta_{1}^{n}G^{(l,q,m)}}{\bar{\tau}^{(l)}_{n}}\right]\int_{0}^{t}\sigma^{(r,m)}_{s}\sigma^{(q,m)}_{s}ds\Bigg]\stackrel{d}{\rightarrow}0,
\end{equation*}
\begin{equation*}
\Rightarrow\Bigg[\frac{1}{n}\sum_{i=1}^{\lfloor nt\rfloor}\frac{\Delta_{i}^{n}Y^{(k)}}{\tau^{(k)}_{n}}\frac{\Delta_{i}^{n}Y^{(l)}}{\tau^{(l)}_{n}}-\sum_{r,m,q=1}^{p}\mathbb{E}\left[\frac{\Delta_{1}^{n}G^{(k,r,m)}}{\bar{\tau}^{(k)}_{n}}\frac{\Delta_{1}^{n}G^{(l,q,m)}}{\bar{\tau}^{(l)}_{n}}\right]\int_{0}^{t}\sigma^{(r,m)}_{s}\sigma^{(q,m)}_{s}ds\Bigg]\stackrel{P}{\rightarrow}0.
\end{equation*}
Then, by triangular inequality it follows that
\begin{equation*}
\frac{1}{n}\sum_{i=1}^{\lfloor nt\rfloor }\frac{\Delta_{i}^{n}Y^{(k)}}{\tau^{(k)}_{n}}\frac{\Delta_{i}^{n}Y^{(l)}}{\tau^{(l)}_{n}}\stackrel{P}{\rightarrow}\sum_{r,m,q=1}^{p}\bar{r}_{k,r,m;l,q,m}(0)\int_{0}^{t}\sigma^{(r,m)}_{s}\sigma^{(q,m)}_{s}ds.
\end{equation*}
The uniform convergence in probability comes from Remark 4.25 of \cite{PhD}, while the joint uniform convergence follows from Lemma \ref{lemma}. 
\end{proof}
\begin{rem}
Similar weak laws of large numbers corresponding to all the others central limit theorems presented in this work, including the ones for the multivariate stationary Gaussian processes, can be derived using the same arguments.
\end{rem}
\section{Feasible results}\label{FEASIBLE}
In all the limit theorems presented above, we considered scaling of the increments of the corresponding processes $X,Y$ or $G$ by $\tau_{n}$. However, in empirical applications $\tau_{n}$ would not be known, which makes the limit theorems infeasible in the sense that they are not computable from empirical data. We now want to move on to derive related feasible results which can be implemented in empirical applications. To this end we somehow need to get rid off the $\tau_{n}$s. A natural way of doing this is to consider suitable ratios of the statistics considered above so that the scaling parameters cancel out.

In this section we will focus on two kind of feasible results, which differ from each other for the type of ratio considered. They are the \textit{correlation ratio} and the \textit{relative covolatility}. For the latter see \cite{Pakkanen}. We will show feasible results for both scenarios of Case II. Moreover, we will present the results using the $vech$ formulation, however similar results hold true for the general formulation. The reason why we focus only on Case II is because for Case I it is not possible to get rid off the scaling factor $\tau$ (unless in trivial cases) and, hence, to get feasible limit theorems.
\begin{rem}
	From the feasible results developed in this section it is possible to obtain estimates for the mean of our process. Thus, we have ``first order" feasible results. It is an open question whether it is possible to obtain ``second order" feasible results, namely estimates for the asymptotic covariance. For the univariate BSS process, this question has been solved for the power variation case in \cite{Pakkanen}, but it still remains open for the multipower variation case.
\end{rem}
In this chapter we will make considerable use of certain random variables and in order to simplify the exposition we decided to use the following formulation. For any $l,k=1,...,p$, we define
\begin{equation*}
\bar{R}^{(k,l)}_{t,n}:=\sum_{r,m,q=1}^{p}\mathbb{E}\left[\frac{\Delta_{1}^{n}G^{(k,r,m)}}{\bar{\tau}^{(k)}_{n}}\frac{\Delta_{1}^{n}G^{(l,q,m)}}{\bar{\tau}^{(l)}_{n}}\right]\int_{0}^{t}\sigma^{(r,m)}_{s}\sigma^{(q,m)}_{s}ds,
\end{equation*}
\begin{equation*}
\bar{R}^{(k,l)}_{t}:=\sum_{r,m,q=1}^{p}\bar{r}_{k,r,m;l,q,m}(0)\int_{0}^{t}\sigma^{(r,m)}_{s}\sigma^{(q,m)}_{s}ds,
\end{equation*}
\begin{equation*}
\tilde{R}^{(k,l)}_{t,n}:=\sum_{m=1}^{p}\mathbb{E}\left[\frac{\Delta_{1}^{n}G^{(k,m)}}{\tilde{\tau}_{n}^{(k)}}\frac{\Delta_{1}^{n}G^{(l,m)}}{\tilde{\tau}_{n}^{(l)}}\right]\int_{0}^{t}\sigma^{(k,m)}_{s}\sigma^{(l,m)}_{s}ds,
\end{equation*}
\begin{equation*}
\text{and}\quad\tilde{R}^{(k,l)}_{t}:=\sum_{m=1}^{p}\tilde{r}_{k,m;l,m}(0)\int_{0}^{t}\sigma^{(k,m)}_{s}\sigma^{(l,m)}_{s}ds.
\end{equation*}
\subsection{Correlation ratio}
Recall that $\bar{\tau}_{n}^{(k)}=\sqrt{\mathbb{E}\left[\left(\Delta^{n}_{1}Y^{(k)}\right)^{2}\right]}$, $\bar{r}_{k,m;l,w}^{(n)}(h)=\mathbb{E}\left[\dfrac{\Delta^{n}_{1}G^{(k,m)}}{\bar{\tau}_{n}^{(k)}}\dfrac{\Delta^{n}_{1+h}G^{(l,w)}}{\bar{\tau}_{n}^{(l)}} \right]$ and $\bar{r}_{k,m;l,w}(h)=\lim\limits_{n\rightarrow\infty}\bar{r}^{(n)}_{k,m;l,w}(h)$. Moreover, observe that, for $k,m=1,...,p$,
\begin{equation*}
\sum_{r,q=1}^{p}\mathbb{E}\left[\frac{\Delta_{1}^{n}G^{(k,r,m)}}{\bar{\tau}^{(k)}_{n}}\frac{\Delta_{1}^{n}G^{(k,q,m)}}{\bar{\tau}^{(k)}_{n}}\right]\int_{0}^{t}\sigma^{(r,m)}_{s}\sigma^{(q,m)}_{s}ds\geq 0,
\end{equation*}
since it is a quadratic form. Therefore, for $k=1,...,p$ we have that $\bar{R}^{(k,k)}_{t,n}\geq 0$ and $\bar{R}^{(k,k)}_{t}\geq 0$ (and $=0$ only in the trivial case). The same applies to $\tilde{R}^{(k,k)}_{t,n}$ and $\tilde{R}^{(k,k)}_{t}$.\\
Before presenting the main results of this section, we introduce the following lemma, which is a generalisation of the functional delta method, see Chapter 3.9 of \cite{VW}, and of Proposition 2 of \cite{Survey}.
\begin{lem}\label{survey}
Let $\mathbb{D}$ and $\mathbb{E}$ be metrizable topological vector spaces and $r_{n}$ constants such that $r_{n}\rightarrow\infty$. Let $\phi:\mathbb{D}_{\phi}\subset \mathbb{D}\rightarrow\mathbb{E}$ be continuous and satisfy
\begin{equation}\label{Hada}
r_{n}(\phi(\theta_{n}+r_{n}^{-1}h_{n})-\phi(\theta_{n}))\rightarrow \phi'(\theta,h),
\end{equation}
for every converging sequences $\theta_{n}$, $h_{n}$ with $h_{n}\rightarrow h\in\mathbb{D}_{0}\subset\mathbb{D}$, $\theta_{n}\rightarrow \theta\in\mathbb{D}_{0}$ and with $\theta_{n} + r_{n}^{-1}h_{n}\in\mathbb{D}_{\phi}$ for all $n$, and some arbitrary map $\phi'(\theta,h)$ from $\mathbb{D}_{0}\times\mathbb{D}_{0}$ to $\mathbb{E}$. If $Y_{n},\bar{Y_{n}}: \Omega_{n}\rightarrow\mathbb{D}_{\phi}$ are maps
with $Y_{n}\stackrel{P}{\rightarrow}Y$ and $r_{n}(Y_{n}-\bar{Y}_{n})\stackrel{st}{\rightarrow}X$, where $X$ is separable and takes its values in $\mathbb{D}_{0}$,
then $r_{n}\left(\phi(Y_{n})-\phi(\bar{Y}_{n})\right)\stackrel{st}{\rightarrow}\phi'(Y,X)$.
\end{lem}
\begin{proof}
Since $\sqrt{n}(Y_{n}-\bar{Y}_{n})\stackrel{st}{\rightarrow}X$ then $|Y_{n}-\bar{Y}_{n}|\stackrel{P}{\rightarrow}0$ and, given that $Y_{n}\stackrel{P}{\rightarrow}Y$, applying triangular inequality we can deduce that $\bar{Y}_{n}\stackrel{P}{\rightarrow}Y$. Hence, we have $(Y_{n},\bar{Y}_{n})\stackrel{P}{\rightarrow}(Y,Y)$, and using the properties of the stable convergence we have $(\bar{Y}_{n},r_{n}(Y_{n}-\bar{Y}_{n}))\stackrel{st}{\rightarrow}(Y,X)$.\\
Now, define for each $n$ a map $g_{n}(\theta_{n},h_{n}):=r_{n}(\phi(\theta_{n}+r_{n}^{-1}h_{n})-\phi(\theta_{n}))$ on $\mathbb{D}_{n}\times\mathbb{D}_{n}:=\{h_{n},\theta_{n}:\theta_{n}+r_{n}^{-1}h_{n}\in\mathbb{D}_{\phi}\}$. These maps are continuous in $\mathbb{E}$ since $\phi$ is continuous and by $(\ref{Hada})$ they converge to $\phi'(\theta,h)$. Applying the continuous mapping theorem, we obtain $g_{n}(\bar{Y}_{n},r_{n}(Y_{n}-\bar{Y}_{n}))\stackrel{st}{\rightarrow}\phi'(Y,X)$, which is our result.
\end{proof}
\noindent We can now present the main results of this section: the feasible WLLN and CLT.
\begin{pro}\label{NEW-Multi-LLN-ratio-corr}
Let $\epsilon>0$. Under the assumptions of Theorem \ref{NEW-I-Vech} and for any interval $[\epsilon,T]$,
\begin{equation}\label{NEW-ucp-corr-eq}
\left(\frac{\sum_{i=1}^{\lfloor nt\rfloor}\Delta^{n}_{i}Y^{(k)}\Delta^{n}_{i}Y^{(l)}}{\sqrt{\sum_{i=1}^{\lfloor nt\rfloor}\left(\Delta^{n}_{i}Y^{(k)}\right)^{2}}\sqrt{\sum_{i=1}^{\lfloor nt\rfloor}\left(\Delta^{n}_{i}Y^{(l)}\right)^{2}}}\right)_{k=1,...,p;l\leq k}\stackrel{u.c.p.}{\rightarrow}\left(\frac{\bar{R}^{(k,l)}_{t}}{\sqrt{\bar{R}^{(k,k)}_{t}\bar{R}^{(l,l)}_{t}}}\right)_{k=1,...,p;l\leq k}
\end{equation}
\end{pro}
\begin{proof}
Fix $l,k\in\{1,...,p\}$ such that $l\leq k$, we have for $n\geq1/t$ (the case $n<1/t$ is a trivial one and moreover we are concerned with the behaviour as $n\rightarrow\infty$)
\begin{equation*}
\left(\sum_{i=1}^{\lfloor nt\rfloor}\Delta^{n}_{i}Y^{(k)}\Delta^{n}_{i}Y^{(l)}\right)\left(\sum_{i=1}^{\lfloor nt\rfloor}\left(\Delta^{n}_{i}Y^{(k)}\right)^{2}\sum_{i=1}^{\lfloor nt\rfloor}\left(\Delta^{n}_{i}Y^{(l)}\right)^{2}\right)^{-1/2}-\frac{\bar{R}_{t}^{(k,l)}}{\sqrt{\bar{R}_{t}^{(k,k)}\bar{R}_{t}^{(l,l)}}}
\end{equation*}
\begin{equation*}
=\left[\left(\frac{1}{n}\sum_{i=1}^{\lfloor nt\rfloor}\dfrac{\Delta^{n}_{i}Y^{(k)}}{\bar{\tau}_{n}^{(k)}}\dfrac{\Delta^{n}_{i}Y^{(l)}}{\bar{\tau}_{n}^{(l)}}\right)-\bar{R}_{t}^{(k,l)}\right]\left(\frac{1}{n}\sum_{i=1}^{\lfloor nt\rfloor}\left(\frac{\Delta^{n}_{i}Y^{(k)}}{\bar{\tau}_{n}^{(k)}}\right)^{2}\frac{1}{n}\sum_{i=1}^{\lfloor nt\rfloor}\left(\frac{\Delta^{n}_{i}Y^{(l)}}{\bar{\tau}_{n}^{(l)}}\right)^{2}\right)^{-1/2}
\end{equation*}
\begin{equation*}
+\frac{\bar{R}_{t}^{(k,l)}\left(\frac{1}{n}\sum_{i=1}^{\lfloor nt\rfloor}\left(\frac{\Delta^{n}_{i}Y^{(k)}}{\bar{\tau}_{n}^{(k)}}\right)^{2}\frac{1}{n}\sum_{i=1}^{\lfloor nt\rfloor}\left(\frac{\Delta^{n}_{i}Y^{(l)}}{\bar{\tau}_{n}^{(l)}}\right)^{2}\right)^{-1/2}}{\sqrt{\bar{R}_{t}^{(k,k)}\bar{R}_{t}^{(l,l)}}}
\end{equation*}
\begin{equation*}
\left[\sqrt{\bar{R}_{t}^{(k,k)}\bar{R}_{t}^{(l,l)}}-\left(\frac{1}{n}\sum_{i=1}^{\lfloor nt\rfloor}\left(\frac{\Delta^{n}_{i}Y^{(k)}}{\bar{\tau}_{n}^{(k)}}\right)^{2}\frac{1}{n}\sum_{i=1}^{\lfloor nt\rfloor}\left(\frac{\Delta^{n}_{i}Y^{(l)}}{\bar{\tau}_{n}^{(l)}}\right)^{2}\right)^{1/2}\right]\stackrel{u.c.p.}{\rightarrow}0,
\end{equation*}
where for the term in the first square bracket the u.c.p. convergence to zero comes from the fact that we have LLN results (see Theorem \ref{LLN}). For the second square bracket we have the following. First, we use the continuous mapping theorem knowing the joint convergence in probability and using the continuous function $g(x,y)=\sqrt{xy}$ (notice that in our case $x,y$ are positive). Then, we pass from the convergence in probability to the uniform convergence using the fact that the paths are non-decreasing in time and the paths of the limiting process are continuous almost surely. Concerning the elements outside the square brackets, they do not interfere with the uniform convergence since their suprema are bounded for any $t\in[\epsilon,T]$ (and that is why we have considered $\epsilon>0$).
\\Finally, the joint convergence follows from Lemma \ref{lemma}.
\end{proof}
\begin{pro}\label{NEW-Feasible-corr}
Under the assumptions of Theorem \ref{NEW-I-Vech}, we have for any $\epsilon>0$
\begin{equation*}
\vast\{\sqrt{n}\left(\frac{\sum_{i=1}^{\lfloor nt\rfloor}\Delta^{n}_{i}Y^{(k)}\Delta^{n}_{i}Y^{(l)}}{\sqrt{\sum_{i=1}^{\lfloor nt\rfloor}\left(\Delta^{n}_{i}Y^{(k)}\right)^{2}}\sqrt{\sum_{i=1}^{\lfloor nt\rfloor}\left(\Delta^{n}_{i}Y^{(l)}\right)^{2}}}-\frac{\bar{R}_{t,n}^{(k,l)}}{\sqrt{\bar{R}_{t,n}^{(k,k)}\bar{R}_{t,n}^{(l,l)}}}\right)_{k=1,...,p;l\leq k}\vast\}_{t\in[\epsilon,T]}
\end{equation*}
\begin{equation*}
\stackrel{st}{\rightarrow}\vast\{\vast(\frac{1}{\sqrt{\bar{R}^{(k,k)}_{t}\bar{R}^{(l,l)}_{t}}}\Bigg(\int_{0}^{t}(V_{s}\textit{\textbf{D}}^{1/2})_{(k,l)}dB^{(k,l)}_{s}-\frac{1}{2}\frac{\bar{R}^{(k,l)}_{t}}{\bar{R}^{(k,k)}_{t}}\int_{0}^{t}(V_{s}\textit{\textbf{D}}^{1/2})_{(k,k)}dB^{(k,k)}_{s} 
\end{equation*}
\begin{equation*}
- \frac{1}{2}\frac{\bar{R}^{(k,l)}_{t}}{\bar{R}^{(l,l)}_{t}}\int_{0}^{t}(V_{s}\textit{\textbf{D}}^{1/2})_{(l,l)}dB^{(l,l)}_{s}\Bigg)\vast)_{k=1,...,p;l\leq k}\vast\}_{t\in[\epsilon,T]},
\end{equation*}
where $(V_{s}\textit{\textbf{D}}^{1/2})_{(k,l)}$ indicates that we are considering only the $(k,l)$ row of the matrix $(V_{s}\textit{\textbf{D}}^{1/2})$ and $B_{s}^{(k,l)}$ are one-dimensional Brownian motions independent from each other.
\end{pro}
\begin{proof}First we prove the statement for fixed $k$ and $l$. As in the previous proof we concentrate on the case $n>1/t$. We have
\begin{equation*}
\sqrt{n}\left[\left(\sum_{i=1}^{\lfloor nt\rfloor}\Delta^{n}_{i}Y^{(k)}\Delta^{n}_{i}Y^{(l)}\right)\left(\sum_{i=1}^{\lfloor nt\rfloor}\left(\Delta^{n}_{i}Y^{(k)}\right)^{2}\sum_{i=1}^{\lfloor nt\rfloor}\left(\Delta^{n}_{i}Y^{(l)}\right)^{2}\right)^{-1/2}-\frac{\bar{R}_{t,n}^{(k,l)}}{\sqrt{\bar{R}_{t,n}^{(k,k)}\bar{R}_{t,n}^{(l,l)}}}\right]
\end{equation*}
\begin{equation*}
=\frac{\sqrt{n}\left(\frac{1}{n}\sum_{i=1}^{\lfloor nt\rfloor}\frac{\Delta^{n}_{i}Y^{(k)}}{\bar{\tau}_{n}^{(k)}}\frac{\Delta^{n}_{i}Y^{(l)}}{\bar{\tau}_{n}^{(l)}}-\bar{R}_{t,n}^{(k,l)}\right)}{\sqrt{\bar{R}_{t,n}^{(k,k)}\bar{R}_{t,n}^{(l,l)}}}
-\frac{\sqrt{n}\left(\frac{1}{n}\sum_{i=1}^{\lfloor nt\rfloor}\frac{\Delta^{n}_{i}Y^{(k)}}{\bar{\tau}_{n}^{(k)}}\frac{\Delta^{n}_{i}Y^{(l)}}{\bar{\tau}_{n}^{(l)}}\right)}{\sqrt{\bar{R}_{t,n}^{(k,k)}\bar{R}_{t,n}^{(l,l)}\left(\frac{1}{n}\sum_{i=1}^{\lfloor nt\rfloor}\left(\frac{\Delta^{n}_{i}Y^{(k)}}{\bar{\tau}_{n}^{(k)}}\right)^{2}\frac{1}{n}\sum_{i=1}^{\lfloor nt\rfloor}\left(\frac{\Delta^{n}_{i}Y^{(l)}}{\bar{\tau}_{n}^{(l)}}\right)^{2}\right)}}
\end{equation*}
\begin{equation*}
\left(\left(\frac{1}{n}\sum_{i=1}^{\lfloor nt\rfloor}\left(\dfrac{\Delta^{n}_{i}Y^{(k)}}{\bar{\tau}_{n}^{(k)}}\right)^{2}\frac{1}{n}\sum_{i=1}^{\lfloor nt\rfloor}\left(\dfrac{\Delta^{n}_{i}Y^{(l)}}{\bar{\tau}_{n}^{(l)}}\right)^{2}\right)^{1/2}-\sqrt{\bar{R}_{t,n}^{(k,k)}\bar{R}_{t,n}^{(l,l)}}\right),
\end{equation*}
\begin{equation*}
=\frac{\sqrt{n}}{\sqrt{\bar{R}_{t,n}^{(k,k)}\bar{R}_{t,n}^{(l,l)}}}\left(1,-\left(\frac{1}{n}\sum_{i=1}^{\lfloor nt\rfloor}\dfrac{\Delta^{n}_{i}Y^{(k)}}{\bar{\tau}_{n}^{(k)}}\dfrac{\Delta^{n}_{i}Y^{(l)}}{\bar{\tau}_{n}^{(l)}}\right)\left(\frac{1}{n}\sum_{i=1}^{\lfloor nt\rfloor}\left(\dfrac{\Delta^{n}_{i}Y^{(k)}}{\bar{\tau}_{n}^{(k)}}\right)^{2}\frac{1}{n}\sum_{i=1}^{\lfloor nt\rfloor}\left(\dfrac{\Delta^{n}_{i}Y^{(l)}}{\bar{\tau}_{n}^{(l)}}\right)^{2}\right)^{-1/2}\right)
\end{equation*}
\begin{equation*}
\left(\left(\frac{1}{n}\sum_{i=1}^{\lfloor nt\rfloor}\dfrac{\Delta^{n}_{i}Y^{(k)}}{\bar{\tau}_{n}^{(k)}}\dfrac{\Delta^{n}_{i}Y^{(l)}}{\bar{\tau}_{n}^{(l)}}\right)-\bar{R}_{t,n}^{(k,l)},
\left(\frac{1}{n}\sum_{i=1}^{\lfloor nt\rfloor}\left(\dfrac{\Delta^{n}_{i}Y^{(k)}}{\bar{\tau}_{n}^{(k)}}\right)^{2}\frac{1}{n}\sum_{i=1}^{\lfloor nt\rfloor}\left(\dfrac{\Delta^{n}_{i}Y^{(l)}}{\bar{\tau}_{n}^{(l)}}\right)^{2}\right)^{1/2}-\sqrt{\bar{R}_{t,n}^{(k,k)}\bar{R}_{t,n}^{(l,l)}}\right)^{\top}.
\end{equation*}
Notice that
\begin{equation*}
Z^{(k,l)}_{1,n}:=\frac{1}{\sqrt{\bar{R}_{t,n}^{(k,k)}\bar{R}_{t,n}^{(l,l)}}}\left(1,-\left(\frac{1}{n}\sum_{i=1}^{\lfloor nt\rfloor}\dfrac{\Delta^{n}_{i}Y^{(k)}}{\bar{\tau}_{n}^{(k)}}\dfrac{\Delta^{n}_{i}Y^{(l)}}{\bar{\tau}_{n}^{(l)}}\right)\left(\frac{1}{n}\sum_{i=1}^{\lfloor nt\rfloor}\left(\dfrac{\Delta^{n}_{i}Y^{(k)}}{\bar{\tau}_{n}^{(k)}}\right)^{2}\frac{1}{n}\sum_{i=1}^{\lfloor nt\rfloor}\left(\dfrac{\Delta^{n}_{i}Y^{(l)}}{\bar{\tau}_{n}^{(l)}}\right)^{2}\right)^{-1/2}\right)
\end{equation*}
\begin{equation*}
\stackrel{u.c.p.}{\rightarrow}\frac{1}{\sqrt{\bar{R}^{(k,k)}_{t}\bar{R}^{(l,l)}_{t}}}\left(1,-\frac{\bar{R}^{(k,l)}_{t}}{\sqrt{\bar{R}^{(k,k)}_{t}\bar{R}^{(l,l)}_{t}}}\right)=:Z^{(k,l)}_{1},
\end{equation*}
by Proposition \ref{Multi-LLN-ratio-corr} and by noticing that for any $\delta>0$
\begin{equation*}
\mathbb{P}\left(\sup\limits_{t\in[\epsilon,T]}\left(\sqrt{\bar{R}^{(k,k)}_{t}\bar{R}^{(l,l)}_{t}}\right)^{-1}>\delta\right)=\mathbb{P}\left(\left(\sqrt{\bar{R}^{(k,k)}_{\epsilon}\bar{R}^{(l,l)}_{\epsilon}}\right)^{-1}>\delta\right).
\end{equation*}For the other term, by Theorem \ref{I-Vech-2} and Lemma \ref{survey} applied to $\textbf{g}:\mathcal{D}([\epsilon,T],\mathbb{R}\times(0,\infty)^{2})\rightarrow\mathcal{D}([\epsilon,T],\mathbb{R}\times(0,\infty))$ (where both Skorokhod spaces as well as the Euclidean spaces are equipped with the uniform metric) defined as $\mathbf{g}(\{\textbf{x}_{t}\}_{t\in[\epsilon,T]})=\{g(x_{1,t},x_{2,t},x_{3,t})\}_{t\in[\epsilon,T]}=\{(x_{1,t},\sqrt{x_{2,t}x_{3,t}})\}_{t\in[\epsilon,T]}$ we have that
\begin{equation*}
Z^{(k,l)}_{2,n}:=\vast\{\sqrt{n}
\vast(\left(\frac{1}{n}\sum_{i=1}^{\lfloor nt\rfloor}\dfrac{\Delta^{n}_{i}Y^{(k)}}{\bar{\tau}_{n}^{(k)}}\dfrac{\Delta^{n}_{i}Y^{(l)}}{\bar{\tau}_{n}^{(l)}}\right)-\bar{R}_{t,n}^{(k,l)},
\end{equation*}
\begin{equation*}
\left(\frac{1}{n}\sum_{i=1}^{\lfloor nt\rfloor}\left(\dfrac{\Delta^{n}_{i}Y^{(k)}}{\bar{\tau}_{n}^{(k)}}\right)^{2}\frac{1}{n}\sum_{i=1}^{\lfloor nt\rfloor}\left(\dfrac{\Delta^{n}_{i}Y^{(l)}}{\bar{\tau}_{n}^{(l)}}\right)^{2}\right)^{1/2}-\sqrt{\bar{R}_{t,n}^{(k,k)}\bar{R}_{t,n}^{(l,l)}}\vast)^{\top}\vast\}_{t\in[\epsilon,T]}
\end{equation*}
\begin{equation*}
\stackrel{st}{\rightarrow}\vast\{\begin{pmatrix}
   1 & 0 & 0  \\
   0 & \frac{1}{2}\sqrt{\frac{\bar{R}_{t}^{(l,l)}}{\bar{R}_{t}^{(k,k)}}} & \frac{1}{2}\sqrt{\frac{\bar{R}_{t}^{(k,k)}}{\bar{R}_{t}^{(l,l)}}}\\
   
  \end{pmatrix}
  \end{equation*}
  \begin{equation*}
  \left(\int_{0}^{t}(V_{s}\textit{\textbf{D}}^{1/2})_{(k,l)}dB^{(k,l)}_{s},\int_{0}^{t}(V_{s}\textit{\textbf{D}}^{1/2})_{(k,k)}dB^{(k,k)}_{s},\int_{0}^{t}(V_{s}\textit{\textbf{D}}^{1/2})_{(l,l)}dB^{(l,l)}_{s}\right)^{\top}\vast\}_{t\in[\epsilon,T]}
\end{equation*}
\begin{equation*}
=\vast\{\begin{pmatrix}
    \int_{0}^{t}(V_{s}\textit{\textbf{D}}^{1/2})_{(k,l)}dB^{(k,l)}_{s} \\
    \frac{1}{2}\sqrt{\frac{\bar{R}^{(l,l)}_{t}}{\bar{R}^{(k,k)}_{t}}}\int_{0}^{t}(V_{s}\textit{\textbf{D}}^{1/2})_{(k,k)}dB^{(k,k)}_{s} + \frac{1}{2}\sqrt{\frac{\bar{R}^{(k,k)}_{t}}{\bar{R}^{(l,l)}_{t}}}\int_{0}^{t}(V_{s}\textit{\textbf{D}}^{1/2})_{(l,l)}dB^{(l,l)}_{s}\\
   
  \end{pmatrix}\vast\}_{t\in[\epsilon,T]}=:Z^{(k,l)}_{2},
\end{equation*}
where $B_{s}^{(k,l)},B_{s}^{(k,k)}$ and $B_{s}^{(l,l)}$ are three independent Brownian motions. Notice that we have not investigated whether such functional $\textbf{g}$ satisfies the conditions of Lemma \ref{survey}. We do it now. First, we check that $\textbf{g}$ is continuous and then, using the notations of Lemma \ref{survey}, that $\textbf{g}'(\{\theta\}_{t\in[\epsilon,T]},\{h\}_{t\in[\epsilon,T]})=\{\nabla g(\theta_{t})h_{t}\}_{t\in[\epsilon,T]}$, where $\nabla g$ is the Jacobian matrix of $g$. In particular, for the continuity we need to show that for every $(\{(x_{n,t}^{(1)},x_{n,t}^{(2)},x_{n,t}^{(3)})\}_{t\in[\epsilon,T]})_{n\in\mathbb{N}}\rightarrow\{(x_{t}^{(1)},x_{t}^{(2)},x_{t}^{(3)})\}_{t\in[\epsilon,T]}$ in $\mathcal{D}([\epsilon,T],\mathbb{R}\times(0,\infty)^{2})$ we have
\begin{equation*}
\lim\limits_{n\rightarrow\infty}\sup\limits_{t\in[\epsilon,T]}\left\|\left(x_{n,t}^{(1)}-x_{t}^{(1)},\sqrt{x_{n,t}^{(2)}x_{n,t}^{(3)}}-\sqrt{x_{t}^{(2)}x_{t}^{(3)}}\right)\right\|_{\infty}=0.
\end{equation*}
For the first component it is straightforward, while for the second we have
\begin{equation*}
\sup\limits_{t\in[\epsilon,T]}\left|\sqrt{x_{n,t}^{(2)}x_{n,t}^{(3)}}-\sqrt{x_{t}^{(2)}x_{t}^{(3)}}\right|\leq\sup\limits_{t\in[\epsilon,T]}\left|\sqrt{x_{n,t}^{(2)}}\left(\sqrt{x_{n,t}^{(3)}}-\sqrt{x_{t}^{(3)}}\right)\right|+\left|\sqrt{x_{t}^{(3)}}\left(\sqrt{x_{n,t}^{(2)}}-\sqrt{x_{t}^{(2)}}\right)\right|\rightarrow 0.
\end{equation*}
Regarding the map $\textbf{g}'$, we need to show that 
\begin{equation*}
\lim\limits_{n\rightarrow\infty}\sup\limits_{t\in[\epsilon,T]}\Bigg\|\Bigg(r_{n}\left(\theta_{n,t}^{(1)}+r_{n}^{-1}h_{n,t}^{(1)}-\theta_{n,t}^{(1)}\right)-h^{(1)}_{t},
\end{equation*}
\begin{equation}\label{bobo}
r_{n}\sqrt{(\theta_{n,t}^{(2)}+r_{n}^{-1}h_{n,t}^{(2)})(\theta_{n,t}^{(3)}+r_{n}^{-1}h_{n,t}^{(3)})}-r_{n}\sqrt{\theta^{(2)}_{n,t}\theta^{(3)}_{n,t}}-\frac{h_{t}^{(2)}}{2}\sqrt{\frac{\theta_{t}^{(3)}}{\theta_{t}^{(2)}}}-\frac{h_{t}^{(3)}}{2}\sqrt{\frac{\theta_{t}^{(2)}}{\theta_{t}^{(3)}}}\Bigg)\Bigg\|_{\infty}=0.
\end{equation}
For the first component it is straightforward since $h^{(n)}\rightarrow h$, while for the second using the Taylor series we have
\begin{equation*}
\sup\limits_{t\in[\epsilon,T]}\Bigg|r_{n}\sqrt{(\theta_{n,t}^{(2)}+r_{n}^{-1}h_{n,t}^{(2)})(\theta_{n,t}^{(3)}+r_{n}^{-1}h_{n,t}^{(3)})}-r_{n}\sqrt{\theta^{(2)}_{n,t}\theta^{(3)}_{n,t}}-\frac{h_{t}^{(2)}}{2}\sqrt{\frac{\theta_{t}^{(3)}}{\theta_{t}^{(2)}}}-\frac{h_{t}^{(3)}}{2}\sqrt{\frac{\theta_{t}^{(2)}}{\theta_{t}^{(3)}}}\Bigg|
\end{equation*}
\begin{equation}\label{bo}
=\sup\limits_{t\in[\epsilon,T]}\Bigg|\frac{h_{n,t}^{(2)}}{2}\sqrt{\frac{\theta_{n,t}^{(3)}}{\theta_{n,t}^{(2)}}}+\frac{h_{n,t}^{(3)}}{2}\sqrt{\frac{\theta_{n,t}^{(2)}}{\theta_{n,t}^{(3)}}}+o(r_{n}^{-1})-\frac{h_{t}^{(2)}}{2}\sqrt{\frac{\theta_{t}^{(3)}}{\theta_{t}^{(2)}}}-\frac{h_{t}^{(3)}}{2}\sqrt{\frac{\theta_{t}^{(2)}}{\theta_{t}^{(3)}}}\Bigg|.
\end{equation}
Now, observe that
\begin{equation*}
\Bigg|\frac{h_{n,t}^{(2)}}{2}\sqrt{\frac{\theta_{n,t}^{(3)}}{\theta_{n,t}^{(2)}}}-\frac{h_{t}^{(2)}}{2}\sqrt{\frac{\theta_{t}^{(3)}}{\theta_{t}^{(2)}}}\Bigg|\leq \Bigg|\frac{h_{n,t}^{(2)}}{2}\left(\sqrt{\frac{\theta_{n,t}^{(3)}}{\theta_{n,t}^{(2)}}}-\sqrt{\frac{\theta_{t}^{(3)}}{\theta_{t}^{(2)}}}\right)\Bigg|+\Bigg|\sqrt{\frac{\theta_{t}^{(3)}}{\theta_{t}^{(2)}}}\left(\frac{h_{n,t}^{(2)}}{2}-\frac{h_{t}^{(2)}}{2}\right)\Bigg|,
\end{equation*}
\begin{equation*}
\sqrt{\frac{\theta_{n,t}^{(3)}}{\theta_{n,t}^{(2)}}}-\sqrt{\frac{\theta_{t}^{(3)}}{\theta_{t}^{(2)}}}=\frac{1}{\sqrt{\theta_{t}^{(2)}}}\left(\sqrt{\theta_{n,t}^{(3)}}-\sqrt{\theta_{t}^{(3)}}\right)+\frac{1}{\sqrt{\theta_{t}^{(2)}}}\sqrt{\frac{\theta_{n,t}^{(3)}}{\theta_{n,t}^{(2)}}}\left(\sqrt{\theta_{t}^{(2)}}-\sqrt{\theta_{n,t}^{(2)}}\right),
\end{equation*}
and that $\sup\limits_{t\in[\epsilon,T]}\bigg|\frac{1}{\sqrt{\theta_{t}^{(2)}}}\bigg|<\infty$ because $\theta_{t}^{(2)}$ takes values in $(0,\infty)$ and because it is a c\`{a}dl\`{a}g function (and the same hold for the other $\theta$s). Then, taking the limit as $n\rightarrow\infty$ in $(\ref{bo})$ we obtain the desired result $(\ref{bobo})$.

Furthermore, since $Z^{(k,l)}_{1,n}\stackrel{P}{\rightarrow}Z^{(k,l)}_{1}$ and $Z^{(k,l)}_{2,n}\stackrel{st}{\rightarrow}Z^{(k,l)}_{2}$ we deduce that $(Z^{(k,l)}_{1,n},Z^{(k,l)}_{2,n})\stackrel{st}{\rightarrow}(Z^{(k,l)}_{1},Z^{(k,l)}_{2})$. Finally, by applying the continuous mapping theorem for the stable convergence using the continuous function $f(Z^{(k,l)}_{1,n},Z^{(k,l)}_{2,n})=\{Z^{(k,l)}_{1,n}(t)Z^{(k,l)}_{2,n}(t)\}_{t\in[\epsilon,T]}$ we obtain our result for fixed $k$ and $l$.

For the joint stable convergence we proceed similarly thanks to the uniform metric. Let
\begin{equation*}
\Theta_{n}:=\Vast\{\begin{pmatrix}
Z^{(1,1)}_{1,n}(t) & 0 &\cdots  & 0 \\ 0 &
\ddots &  & \vdots \\ \vdots &  & \ddots & 0\\ 0& \cdots& 0 & Z^{(p,p)}_{1,n} (t) \end{pmatrix}\Vast\}_{t\in[\epsilon,T]}.
\end{equation*}
Then,
\begin{equation*}
\vast\{\sqrt{n}\left(\frac{\sum_{i=1}^{\lfloor nt\rfloor}\Delta^{n}_{i}Y^{(k)}\Delta^{n}_{i}Y^{(l)}}{\sqrt{\sum_{i=1}^{\lfloor nt\rfloor}\left(\Delta^{n}_{i}Y^{(k)}\right)^{2}}\sqrt{\sum_{i=1}^{\lfloor nt\rfloor}\left(\Delta^{n}_{i}Y^{(l)}\right)^{2}}}-\frac{\bar{R}_{t,n}^{(k,l)}}{\sqrt{\bar{R}_{t,n}^{(k,k)}\bar{R}_{t,n}^{(l,l)}}}\right)_{k=1,...,p;l\leq k}\vast\}_{t\in[\epsilon,T]}
\end{equation*}
\begin{equation*}
=\left\{\Theta_{n}(t)\left(Z^{(1,1)}_{2,n}(t),...,Z^{(p,p)}_{2,n}(t)\right)^{\top}\right\}_{t\in[\epsilon,T]}.
\end{equation*}
Using Lemma \ref{lemma} and the arguments used before for fixed $k$ and $l$ we obtain the u.c.p. convergence of $\Theta_{n}$. Now, we would like to prove the stable convergence for $\Big\{\left(Z^{(1,1)}_{2,n}(t),...,Z^{(p,p)}_{2,n}(t)\right)^{\top}\Big\}_{t\in[\epsilon,T]}$. Define the function $\mathbf{\tilde{g}}:\mathcal{D}([\epsilon,T],(0,\infty)^{p}\times\mathbb{R}^{\frac{p}{2}(p-1)})\rightarrow\mathcal{D}([\epsilon,T],(0,\infty)^{\frac{p}{2}(p+1)}\times\mathbb{R}^{\frac{p}{2}(p+1)})$ as (using our variables)
\begin{equation*}
\mathbf{\tilde{g}}\left(\left\{\left( \frac{1}{n}\sum_{i=1}^{\lfloor nt\rfloor}\dfrac{\Delta^{n}_{i}Y^{(k)}}{\bar{\tau}_{n}^{(k)}}\dfrac{\Delta^{n}_{i}Y^{(l)}}{\bar{\tau}_{n}^{(l)}}\right)_{k=1,...,p;l\leq k}\right\}_{t\in[\epsilon,T]}\right)
\end{equation*}
\begin{equation*}
=\vast\{\vast(
\frac{1}{n}\sum_{i=1}^{\lfloor nt\rfloor}\left(\frac{\Delta^{n}_{i}Y^{(1)}}{\bar{\tau}_{n}^{(1)}}\right)^{2},\sqrt{\left(\frac{1}{n}\sum_{i=1}^{\lfloor nt\rfloor}\left(\frac{\Delta^{n}_{i}Y^{(1)}}{\bar{\tau}_{n}^{(1)}}\right)^{2}\right)^{2}},\frac{1}{n}\sum_{i=1}^{\lfloor nt\rfloor}\dfrac{\Delta^{n}_{i}Y^{(2)}}{\bar{\tau}_{n}^{(2)}}\dfrac{\Delta^{n}_{i}Y^{(1)}}{\bar{\tau}_{n}^{(1)}},
\end{equation*}
\begin{equation*}
\sqrt{\frac{1}{n}\sum_{i=1}^{\lfloor nt\rfloor}\left(\frac{\Delta^{n}_{i}Y^{(2)}}{\bar{\tau}_{n}^{(2)}}\right)^{2}\frac{1}{n}\sum_{i=1}^{\lfloor nt\rfloor}\left(\frac{\Delta^{n}_{i}Y^{(1)}}{\bar{\tau}_{n}^{(1)}}\right)^{2}},\frac{1}{n}\sum_{i=1}^{\lfloor nt\rfloor}\left(\frac{\Delta^{n}_{i}Y^{(2)}}{\bar{\tau}_{n}^{(2)}}\right)^{2},\sqrt{\left(\frac{1}{n}\sum_{i=1}^{\lfloor nt\rfloor}\left(\frac{\Delta^{n}_{i}Y^{(2)}}{\bar{\tau}_{n}^{(2)}}\right)^{2}\right)^{2}},
\end{equation*}
\begin{equation*}
\frac{1}{n}\sum_{i=1}^{\lfloor nt\rfloor}\dfrac{\Delta^{n}_{i}Y^{(3)}}{\bar{\tau}_{n}^{(3)}}\dfrac{\Delta^{n}_{i}Y^{(1)}}{\bar{\tau}_{n}^{(1)}},\sqrt{\frac{1}{n}\sum_{i=1}^{\lfloor nt\rfloor}\left(\frac{\Delta^{n}_{i}Y^{(3)}}{\bar{\tau}_{n}^{(3)}}\right)^{2}\frac{1}{n}\sum_{i=1}^{\lfloor nt\rfloor}\left(\frac{\Delta^{n}_{i}Y^{(1)}}{\bar{\tau}_{n}^{(1)}}\right)^{2}},\frac{1}{n}\sum_{i=1}^{\lfloor nt\rfloor}\dfrac{\Delta^{n}_{i}Y^{(3)}}{\bar{\tau}_{n}^{(3)}}\dfrac{\Delta^{n}_{i}Y^{(2)}}{\bar{\tau}_{n}^{(2)}},
\end{equation*}
\begin{equation*}
\sqrt{\frac{1}{n}\sum_{i=1}^{\lfloor nt\rfloor}\left(\frac{\Delta^{n}_{i}Y^{(3)}}{\bar{\tau}_{n}^{(3)}}\right)^{2}\frac{1}{n}\sum_{i=1}^{\lfloor nt\rfloor}\left(\frac{\Delta^{n}_{i}Y^{(2)}}{\bar{\tau}_{n}^{(2)}}\right)^{2}},...,\frac{1}{n}\sum_{i=1}^{\lfloor nt\rfloor}\left(\frac{\Delta^{n}_{i}Y^{(p)}}{\bar{\tau}_{n}^{(p)}}\right)^{2},\sqrt{\left(\frac{1}{n}\sum_{i=1}^{\lfloor nt\rfloor}\left(\frac{\Delta^{n}_{i}Y^{(p)}}{\bar{\tau}_{n}^{(p)}}\right)^{2}\right)^{2}}\vast)^{\top}\vast\}_{t\in[\epsilon,T]}.
\end{equation*}
Then, using the same arguments used for fixed $k$ and $l$ we obtain the stated result.
\end{proof}
Similar results can be obtained for the second scenario of Case II.
\begin{pro}\label{Multi-LLN-ratio-corr}
Let $\epsilon>0$. Under the assumptions of Theorem \ref{I-Vech-2} and for any interval $[\epsilon,T]$,
\begin{equation*}
\left(\frac{\sum_{i=1}^{\lfloor nt\rfloor}\Delta^{n}_{i}X^{(k)}\Delta^{n}_{i}X^{(l)}}{\sqrt{\sum_{i=1}^{\lfloor nt\rfloor}\left(\Delta^{n}_{i}X^{(k)}\right)^{2}}\sqrt{\sum_{i=1}^{\lfloor nt\rfloor}\left(\Delta^{n}_{i}X^{(l)}\right)^{2}}}\right)_{k=1,...,p;l\leq k}\stackrel{u.c.p.}{\rightarrow}\left(\frac{\tilde{R}^{(k,l)}_{t}}{\sqrt{\tilde{R}^{(k,k)}_{t}\tilde{R}^{(l,l)}_{t}}}\right)_{k=1,...,p;l\leq k}.
\end{equation*}
\end{pro}
\begin{proof}
It follows from the same arguments as the one used in the proof of Proposition \ref{NEW-Multi-LLN-ratio-corr}.
\end{proof}
\begin{pro}Under the assumptions of Theorem \ref{I-Vech-2}, we have for any $\epsilon>0$
\begin{equation*}
\vast\{\sqrt{n}\left(\frac{\sum_{i=1}^{\lfloor nt\rfloor}\Delta^{n}_{i}X^{(k)}\Delta^{n}_{i}X^{(l)}}{\sqrt{\sum_{i=1}^{\lfloor nt\rfloor}\left(\Delta^{n}_{i}X^{(k)}\right)^{2}}\sqrt{\sum_{i=1}^{\lfloor nt\rfloor}\left(\Delta^{n}_{i}X^{(l)}\right)^{2}}}-\frac{\tilde{R}_{t,n}^{(k,l)}}{\sqrt{\tilde{R}_{t,n}^{(k,k)}\tilde{R}_{t,n}^{(l,l)}}}\right)_{k=1,...,p;l\leq k}\vast\}_{t\in[\epsilon,T]}
\end{equation*}
\begin{equation*}
\stackrel{st}{\rightarrow}\vast\{\vast(\frac{1}{\sqrt{\tilde{R}_{t}^{(k,k)}\tilde{R}_{t}^{(l,l)}}}\Bigg(\int_{0}^{t}(V_{s}\textit{\textbf{D}}^{1/2})_{(k,l)}dB^{(k,l)}_{s}-\frac{1}{2}\frac{\tilde{R}_{t}^{(k,l)}}{\tilde{R}_{t}^{(k,k)}}\int_{0}^{t}(V_{s}\textit{\textbf{D}}^{1/2})_{(k,k)}dB^{(k,k)}_{s} 
\end{equation*}
\begin{equation*}
- \frac{1}{2}\frac{\tilde{R}_{t}^{(k,l)}}{\tilde{R}_{t}^{(l,l)}}\int_{0}^{t}(V_{s}\textit{\textbf{D}}^{1/2})_{(l,l)}dB^{(l,l)}_{s}\Bigg)\vast)_{k=1,...,p;l\leq k}\vast\}_{t\in[\epsilon,T]},
\end{equation*}
where $(V_{s}\textit{\textbf{D}}^{1/2})_{(k,l)}$ indicates that we are considering only the $(k,l)$ row of the matrix $(V_{s}\textit{\textbf{D}}^{1/2})$ and $B_{s}^{(k,l)}$ are one-dimensional Brownian motions independent from each other.
\end{pro}
\begin{proof}
It follows from the same arguments of Proposition \ref{NEW-Feasible-corr}.
\end{proof}
\subsection{Relative covolatility}
In this section we are looking at the relative volatility case (see \cite{Pakkanen}). Similarly to the previous section, we present first the results for the first scenario and then for the second one of Case II.
\begin{pro}\label{New-Feasible-relative-ucp}
Assume that $\forall n\in\mathbb{N},\sum_{i=1}^{\lfloor nT\rfloor}\Delta^{n}_{i}Y^{(k)}\Delta^{n}_{i}Y^{(l)}\neq0$ a.s.~for $k=1...,p$ and $l\leq k$, then under the assumptions of Theorem \ref{NEW-I-Vech} we have
\begin{equation*}
\left(\frac{\sum_{i=1}^{\lfloor nt\rfloor}\Delta^{n}_{i}Y^{(k)}\Delta^{n}_{i}Y^{(l)}}{\sum_{i=1}^{\lfloor nT\rfloor}\Delta^{n}_{i}Y^{(k)}\Delta^{n}_{i}Y^{(l)}}\right)_{k,l=1,...,p;l\leq k}\stackrel{u.c.p.}{\rightarrow}\left(\frac{\bar{R}^{(k,l)}_{t}}{\bar{R}^{(k,l)}_{T}}\right)_{k=1,...,p;l\leq k}.
\end{equation*}
\end{pro}
\begin{proof}
Fix $k,l$. We have
\begin{equation*}
\left(\sum_{i=1}^{\lfloor nt\rfloor}\Delta^{n}_{i}Y^{(k)}\Delta^{n}_{i}Y^{(l)}\right)\left(\sum_{i=1}^{\lfloor nT\rfloor}\Delta^{n}_{i}Y^{(k)}\Delta^{n}_{i}Y^{(l)}\right)^{-1}-\frac{\bar{R}^{(k,l)}_{t}}{\bar{R}^{(k,l)}_{T}}
\end{equation*}
\begin{equation*}
=\left(\frac{1}{n}\sum_{i=1}^{\lfloor nt\rfloor}\dfrac{\Delta^{n}_{i}Y^{(k)}}{\bar{\tau}_{n}^{(k)}}\dfrac{\Delta^{n}_{i}Y^{(l)}}{\bar{\tau}_{n}^{(l)}}-\bar{R}^{(k,l)}_{t}\right)\left(\frac{1}{n}\sum_{i=1}^{\lfloor nT\rfloor}\dfrac{\Delta^{n}_{i}Y^{(k)}}{\bar{\tau}_{n}^{(k)}}\dfrac{\Delta^{n}_{i}Y^{(l)}}{\bar{\tau}_{n}^{(l)}}\right)^{-1}
\end{equation*}
\begin{equation*}
+\frac{\bar{R}^{(k,l)}_{t}}{\bar{R}^{(k,l)}_{T}}\left(\frac{1}{n}\sum_{i=1}^{\lfloor nT\rfloor}\dfrac{\Delta^{n}_{i}Y^{(k)}}{\bar{\tau}_{n}^{(k)}}\dfrac{\Delta^{n}_{i}Y^{(l)}}{\bar{\tau}_{n}^{(l)}}\right)^{-1}\left(\bar{R}^{(k,l)}_{T}-\frac{1}{n}\sum_{i=1}^{\lfloor nT\rfloor}\dfrac{\Delta^{n}_{i}Y^{(k)}}{\bar{\tau}_{n}^{(k)}}\dfrac{\Delta^{n}_{i}Y^{(l)}}{\bar{\tau}_{n}^{(l)}}\right)\stackrel{u.c.p.}{\rightarrow}0,
\end{equation*}
by Theorem \ref{LLN}. Notice that the supremum of $\bar{R}^{(k,l)}_{t}$ over $t\in[0,T]$ is bounded since the $\sigma$s are compact on bounded intervals. Finally, the joint convergence follows from Lemma \ref{lemma}.
\end{proof}
\begin{pro}\label{New-Feasible-relative-CLT}
Assume that $\forall n\in\mathbb{N},\sum_{i=1}^{\lfloor nT\rfloor}\Delta^{n}_{i}Y^{(k)}\Delta^{n}_{i}Y^{(l)}\neq0$ a.s.~for $k=1...,p$ and $l\leq k$, then under the assumptions of Theorem \ref{NEW-I-Vech} we have
\begin{equation*}
\sqrt{n}\left(\frac{\sum_{i=1}^{\lfloor nt\rfloor}\Delta^{n}_{i}Y^{(k)}\Delta^{n}_{i}Y^{(l)}}{\sum_{i=1}^{\lfloor nT\rfloor}\Delta^{n}_{i}Y^{(k)}\Delta^{n}_{i}Y^{(l)}}-\frac{\bar{R}^{(k,l)}_{t,n}}{\bar{R}^{(k,l)}_{T,n}}\right)_{k=1,...,p;l\leq k}
\end{equation*}
\begin{equation*}
\stackrel{st}{\rightarrow}\left(\frac{1}{\bar{R}^{(k,l)}_{T}}\int_{0}^{t}(V_{s}\textit{\textbf{D}}^{1/2})_{(k,l)}dB^{(k,l)}_{s}-\frac{\bar{R}^{(k,l)}_{t}}{\big(\bar{R}^{(k,l)}_{T}\big)^{2}}\int_{0}^{T}(V_{s}\textit{\textbf{D}}^{1/2})_{(k,l)}dB^{(k,l)}_{s}\right)_{k=1,...,p;l\leq k}
\end{equation*}
in $\mathcal{D}([0,T],\mathbb{R}^{\frac{p}{2}(p+1)})$, where $(V_{s}\textit{\textbf{D}}^{1/2})_{(k,l)}$ indicates that we are considering only the $(k,l)$ row of the matrix $(V_{s}\textit{\textbf{D}}^{1/2})$ and $B_{s}^{(k,l)}$ are one-dimensional Brownian motions independent from each other.
\end{pro}
\begin{proof} Fix $k,l$. We have
\begin{equation*}
\sqrt{n}\left[\frac{\sum_{i=1}^{\lfloor nt\rfloor}\Delta^{n}_{i}Y^{(k)}\Delta^{n}_{i}Y^{(l)}}{\sum_{i=1}^{\lfloor nT\rfloor}\Delta^{n}_{i}Y^{(k)}\Delta^{n}_{i}Y^{(l)}}-\frac{\bar{R}^{(k,l)}_{t,n}}{\bar{R}^{(k,l)}_{T,n}}\right]=\frac{\sqrt{n}}{\bar{R}^{(k,l)}_{T,n}}\left(\frac{1}{n}\sum_{i=1}^{\lfloor nt\rfloor}\dfrac{\Delta^{n}_{i}Y^{(k)}}{\bar{\tau}_{n}^{(k)}}\dfrac{\Delta^{n}_{i}Y^{(l)}}{\bar{\tau}_{n}^{(l)}}-\bar{R}^{(k,l)}_{t,n}\right)
\end{equation*}
\begin{equation*}
-\frac{\sqrt{n}}{\bar{R}^{(k,l)}_{T,n}}\left(\frac{1}{n}\sum_{i=1}^{\lfloor nt\rfloor}\frac{\Delta^{n}_{i}Y^{(k)}}{\bar{\tau}_{n}^{(k)}}\frac{\Delta^{n}_{i}Y^{(l)}}{\bar{\tau}_{n}^{(l)}}\right)\left(\frac{1}{n}\sum_{i=1}^{\lfloor nT\rfloor}\frac{\Delta^{n}_{i}Y^{(k)}}{\bar{\tau}_{n}^{(k)}}\frac{\Delta^{n}_{i}Y^{(l)}}{\bar{\tau}_{n}^{(l)}}\right)^{-1}\left(\frac{1}{n}\sum_{i=1}^{\lfloor nT\rfloor}\dfrac{\Delta^{n}_{i}Y^{(k)}}{\bar{\tau}_{n}^{(k)}}\dfrac{\Delta^{n}_{i}Y^{(l)}}{\bar{\tau}_{n}^{(l)}}-\bar{R}^{(k,l)}_{T,n}\right),
\end{equation*}
which can be rewritten in vector notation as
\begin{equation*}
=\frac{\sqrt{n}}{\bar{R}^{(k,l)}_{T,n}}\left(1,-\left(\frac{1}{n}\sum_{i=1}^{\lfloor nt\rfloor}\dfrac{\Delta^{n}_{i}Y^{(k)}}{\bar{\tau}_{n}^{(k)}}\dfrac{\Delta^{n}_{i}Y^{(l)}}{\bar{\tau}_{n}^{(l)}}\right)\left(\frac{1}{n}\sum_{i=1}^{\lfloor nT\rfloor}\dfrac{\Delta^{n}_{i}Y^{(k)}}{\bar{\tau}_{n}^{(k)}}\dfrac{\Delta^{n}_{i}Y^{(l)}}{\bar{\tau}_{n}^{(l)}}\right)^{-1}\right)
\end{equation*}
\begin{equation*}
\left(\frac{1}{n}\sum_{i=1}^{\lfloor nt\rfloor}\dfrac{\Delta^{n}_{i}Y^{(k)}}{\bar{\tau}_{n}^{(k)}}\dfrac{\Delta^{n}_{i}Y^{(l)}}{\bar{\tau}_{n}^{(l)}}-\bar{R}^{(k,l)}_{t,n},\frac{1}{n}\sum_{i=1}^{\lfloor nT\rfloor}\dfrac{\Delta^{n}_{i}Y^{(k)}}{\bar{\tau}_{n}^{(k)}}\dfrac{\Delta^{n}_{i}Y^{(l)}}{\bar{\tau}_{n}^{(l)}}-\bar{R}^{(k,l)}_{T,n}\right)^{\top}.
\end{equation*}
Notice that 
\begin{equation*}
\frac{1}{\bar{R}^{(k,l)}_{T,n}}\left(1,-\left(\frac{1}{n}\sum_{i=1}^{\lfloor nt\rfloor}\dfrac{\Delta^{n}_{i}Y^{(k)}}{\bar{\tau}_{n}^{(k)}}\dfrac{\Delta^{n}_{i}Y^{(l)}}{\bar{\tau}_{n}^{(l)}}\right)\left(\frac{1}{n}\sum_{i=1}^{\lfloor nT\rfloor}\dfrac{\Delta^{n}_{i}Y^{(k)}}{\bar{\tau}_{n}^{(k)}}\dfrac{\Delta^{n}_{i}Y^{(l)}}{\bar{\tau}_{n}^{(l)}}\right)^{-1}\right)\stackrel{u.c.p}{\rightarrow}\frac{1}{\bar{R}^{(k,l)}_{T}}\left(1,-\frac{\bar{R}^{(k,l)}_{T}}{\bar{R}^{(k,l)}_{T}}\right)
\end{equation*}
using Proposition \ref{New-Feasible-relative-ucp}, and that
\begin{equation*}
\sqrt{n}\left(\frac{1}{n}\sum_{i=1}^{\lfloor nt\rfloor}\dfrac{\Delta^{n}_{i}Y^{(k)}}{\bar{\tau}_{n}^{(k)}}\dfrac{\Delta^{n}_{i}Y^{(l)}}{\bar{\tau}_{n}^{(l)}}-\bar{R}^{(k,l)}_{t,n},\frac{1}{n}\sum_{i=1}^{\lfloor nT\rfloor}\dfrac{\Delta^{n}_{i}Y^{(k)}}{\bar{\tau}_{n}^{(k)}}\dfrac{\Delta^{n}_{i}Y^{(l)}}{\bar{\tau}_{n}^{(l)}}-\bar{R}^{(k,l)}_{T,n}\right)^{\top}
\end{equation*}
\begin{equation*}
\stackrel{st}{\rightarrow}\left(\int_{0}^{t}(V_{s}\textit{\textbf{D}}^{1/2})_{(k,l)}dB^{(k,l)}_{s},\int_{0}^{T}(V_{s}\textit{\textbf{D}}^{1/2})_{(k,l)}dB^{(k,l)}_{s}\right)^{\top}\quad\text{in}\quad\mathcal{D}([0,T],\mathbb{R}^{2})
\end{equation*}
by Theorem \ref{NEW-I-Vech}, where $(V_{s}\textit{\textbf{D}}^{1/2})_{(k,l)}$ indicates that we are considering only the $(k,l)$ row of the matrix $(V_{s}\textit{\textbf{D}}^{1/2})$. Then using the properties of the stable convergence and the continuous mapping theorem we conclude the proof for fixed $k,l$. For the joint case we proceed as we have done in the proof of Proposition \ref{NEW-Feasible-corr}. In particular, we have an equivalent of $\Theta_{n}$ which converges in $u.c.p.$ since its elements do. Moreover, we have an equivalent of $\Big\{\left(Z^{(1,1)}_{2,n}(t),...,Z^{(p,p)}_{2,n}(t)\right)^{\top}\Big\}_{t\in[\epsilon,T]}$ whose stable convergence in the Skorokhod space is guaranteed by Theorem 4.11 and the continuous mapping theorem using the function $\textbf{g}(\{x_{1}(t),...,x_{\frac{p}{2}(p+1)}(t)\}_{t\in[0,T]})=\{(x_{1}(t),x_{1}(T),...,x_{\frac{p}{2}(p+1)}(t),x_{\frac{p}{2}(p+1)}(T))\}_{t\in[0,T]}$. Finally, using the properties of the stable convergence we obtain the stated result.
\end{proof}
Similar results can be obtained for the second scenario of Case II.
\begin{pro} Assume that $\forall n\in\mathbb{N},\sum_{i=1}^{\lfloor nT\rfloor}\Delta^{n}_{i}X^{(k)}\Delta^{n}_{i}X^{(l)}\neq0$ a.s.~for any $k=1,...,p$ and $l\leq k$, then under the assumptions of Theorem \ref{I-Vech-2} we have
\begin{equation*}
\left(\frac{\sum_{i=1}^{\lfloor nt\rfloor}\Delta^{n}_{i}X^{(k)}\Delta^{n}_{i}X^{(l)}}{\sum_{i=1}^{\lfloor nT\rfloor}\Delta^{n}_{i}X^{(k)}\Delta^{n}_{i}X^{(l)}}\right)_{k=1,...,p;l\leq k}\stackrel{u.c.p.}{\rightarrow}\left(\frac{\tilde{R}^{(k,l)}_{t}}{\tilde{R}^{(k,l)}_{T}}\right)_{k=1,...,p;l\leq k}.
\end{equation*}
\end{pro}
\begin{proof}
It follows from the same arguments as the ones used in the proof of Proposition \ref{New-Feasible-relative-ucp}.
\end{proof}
\begin{pro} Assume that $\forall n\in\mathbb{N},\sum_{i=1}^{\lfloor nT\rfloor}\Delta^{n}_{i}X^{(k)}\Delta^{n}_{i}X^{(l)}\neq0$ a.s.~for $k=1...,p$ and $l\leq k$, then under the assumptions of Theorem \ref{I-Vech-2} we have
\begin{equation*}
\sqrt{n}\left(\frac{\sum_{i=1}^{\lfloor nt\rfloor}\Delta^{n}_{i}X^{(k)}\Delta^{n}_{i}X^{(l)}}{\sum_{i=1}^{\lfloor nT\rfloor}\Delta^{n}_{i}X^{(k)}\Delta^{n}_{i}X^{(l)}}-\frac{\tilde{R}^{(k,l)}_{t,n}}{\tilde{R}^{(k,l)}_{T,n}}\right)_{k=1,...,p;l\leq k}
\end{equation*}
\begin{equation*}
\stackrel{st}{\rightarrow}\left(\frac{1}{\tilde{R}^{(k,l)}_{T}}\int_{0}^{t}(V_{s}\textit{\textbf{D}}^{1/2})_{(k,l)}dB^{(k,l)}_{s}-\frac{\tilde{R}^{(k,l)}_{t}}{\big(\tilde{R}^{(k,l)}_{T}\big)^{2}}\int_{0}^{T}(V_{s}\textit{\textbf{D}}^{1/2})_{(k,l)}dB^{(k,l)}_{s}\right)_{k=1,...,p;l\leq k},
\end{equation*}
in $\mathcal{D}([0,T],\mathbb{R}^{\frac{p}{2}(p+1)})$, where $(V_{s}\textit{\textbf{D}}^{1/2})_{(k,l)}$ indicates that we are considering only the $(k,l)$ row of the matrix $(V_{s}\textit{\textbf{D}}^{1/2})$ and $B_{s}^{(k,l)}$ are one-dimensional Brownian motions independent from each other.
\end{pro}
\begin{proof}
It follows from the same arguments of Proposition \ref{New-Feasible-relative-CLT}.
\end{proof}
\begin{rem}
	Similar feasible results for general stationary multivariate Gaussian processes can be derived by just setting all the $\sigma$s to be equal to 1.
\end{rem}
\section{Example: the gamma kernel}\label{EXAMPLE}
I this section we are going to explore an example of the multivariate process $\{\textbf{G}_{t}\}_{t\in[0,T]}$. We will focus on the gamma kernel (see \cite{BCD}) because it plays a central role in the modelling of (atmospheric) turbulences, which is one of the main objective of the development of BSS processes. We will show that the process obtained from using the gamma kernel in the kernel matrix satisfies our assumptions. Consider the stochastic process $\{\textbf{G}_{t}\}_{t\in[0,T]}$ defined as
\begin{equation*}
\textbf{G}_{t}:=\begin{pmatrix}
     G^{(1)}_{t}   \\
     \vdots\\
     G^{(p)}_{t} 
    \end{pmatrix}=\int_{-\infty}^{t}\begin{pmatrix}
   g^{(1,1)}(t-s) & \dots & g^{(1,p)}(t-s)  \\
   \vdots & \ddots & \vdots\\
   g^{(p,1)}(t-s) & \dots & g^{(p,p)}(t-s)
  \end{pmatrix} \begin{pmatrix}
     dW^{(1)}_{s}   \\
     \vdots\\
     dW^{(p)}_{s} 
    \end{pmatrix},
\end{equation*}
where $g^{(i,j)}(t)=t^{\delta^{(i,j)}}e^{-\lambda^{(i,j)}t}\mathbf{1}_{[0,\infty)}(t)$ and $W^{(i)}$ are independent Gaussian $\mathcal{F}_{t}$-Brownian measures on $\mathbb{R}$, for $i,j=1,...,p$. Hence, we have $G^{(i)}_{t}=\sum_{j=1}^{p}\int_{-\infty}^{t}g^{(i,j)}(t-s)dW^{(j)}_{s}$ and
\begin{equation*}
\mathbb{E}\left[G^{(i)}_{t+h}G^{(j)}_{t} \right]=\sum_{l=1}^{p}\int_{-\infty}^{t}g^{(i,l)}(t+h-s)g^{(j,l)}(t-s)ds
\end{equation*}
\begin{equation*}
=\sum_{l=1}^{p}\int_{0}^{\infty}g^{(i,l)}(x+h)g^{(j,l)}(x)dx=\sum_{l=1}^{p}\int_{0}^{\infty}(x+h)^{\delta^{(i,l)}}e^{-\lambda^{(i,l)}(x+h)}x^{\delta^{(j,l)}}e^{-\lambda^{(j,l)}x}dx
\end{equation*}
\begin{equation*}
=\sum_{l=1}^{p}e^{-\lambda^{(i,l)}h}\int_{0}^{\infty}(x+h)^{\delta^{(i,l)}}x^{\delta^{(j,l)}}e^{-(\lambda^{(i,l)}+\lambda^{(j,l)})x}dx.
\end{equation*}
It is important to notice that if $\delta^{(i,j)}\in(-\frac{1}{2},0)\cup(0,\frac{1}{2})$ then $\int_{-\infty}^{t}g^{(i,j)}(t-s)dW^{(j)}_{s}$ it is not a semimartingale (see \cite{BCD} and \cite{Andrea1}). Furthermore, observe that 
\begin{equation*}
r_{i,j}^{(n)}(k):=\mathbb{E}\left[\dfrac{\Delta^{n}_{1}G^{(i)}}{\tau_{n}^{(i)}}\dfrac{\Delta^{n}_{1+k}G^{(j)}}{\tau_{n}^{(j)}} \right]=\frac{\mathbb{E}\left[G^{(i)}_{\frac{1}{n}}G^{(j)}_{\frac{1+k}{n}} \right]-\mathbb{E}\left[G^{(i)}_{\frac{1}{n}}G^{(j)}_{\frac{k}{n}} \right]-\mathbb{E}\left[G^{(i)}_{0}G^{(j)}_{\frac{1+k}{n}} \right]+\mathbb{E}\left[G^{(i)}_{0}G^{(j)}_{\frac{k}{n}} \right]}{\tau_{n}^{(i)}\tau_{n}^{(j)}}
\end{equation*}
\begin{equation*}
=\frac{\mathbb{E}\left[G^{(i)}_{\frac{1}{n}}G^{(j)}_{\frac{1+k}{n}} \right]-\mathbb{E}\left[G^{(i)}_{\frac{1}{n}}G^{(j)}_{\frac{k}{n}} \right]-\mathbb{E}\left[G^{(i)}_{0}G^{(j)}_{\frac{1+k}{n}} \right]+\mathbb{E}\left[G^{(i)}_{0}G^{(j)}_{\frac{k}{n}} \right]}{\left(\mathbb{E}\left[(G^{(i)}_{0})^{2} \right]+\mathbb{E}\left[(G^{(i)}_{\frac{1}{n}})^{2}\right]-2\mathbb{E}\left[G^{(i)}_{\frac{1}{n}}G^{(i)}_{0} \right]\right)^{1/2}\left(\mathbb{E}\left[(G^{(j)}_{0})^{2} \right]+\mathbb{E}\left[(G^{(j)}_{\frac{1}{n}})^{2}\right]-2\mathbb{E}\left[G^{(j)}_{\frac{1}{n}}G^{(j)}_{0} \right]\right)^{1/2}},
\end{equation*}
and by stationarity we have
\begin{equation*}
=\frac{2\mathbb{E}\left[G^{(i)}_{\frac{1}{n}}G^{(j)}_{\frac{1+k}{n}} \right]-\mathbb{E}\left[G^{(i)}_{\frac{1}{n}}G^{(j)}_{\frac{k}{n}} \right]-\mathbb{E}\left[G^{(i)}_{0}G^{(j)}_{\frac{1+k}{n}} \right]}{\left(2\mathbb{E}\left[(G^{(i)}_{0})^{2} \right]-2\mathbb{E}\left[G^{(i)}_{\frac{1}{n}}G^{(i)}_{0} \right]\right)^{\frac{1}{2}}\left(2\mathbb{E}\left[(G^{(j)}_{0})^{2} \right]-2\mathbb{E}\left[G^{(j)}_{\frac{1}{n}}G^{(j)}_{0} \right]\right)^{\frac{1}{2}}}
\end{equation*}
\begin{equation}\label{example}
=\frac{\sum_{l=1}^{p}\int_{0}^{\infty}2g^{(i,l)}(x+\frac{k}{n})g^{(j,l)}(x)-g^{(i,l)}(x+\frac{k-1}{n})g^{(j,l)}(x)-g^{(i,l)}(x+\frac{k+1}{n})g^{(j,l)}(x)dx}{2\left(\sum_{l=1}^{p}\int_{0}^{\infty}\left(g^{(i,l)}(x)\right)^{2}-g^{(i,l)}(x+\frac{1}{n})g^{(i,l)}(x)dx\right)^{\frac{1}{2}}\left(\sum_{l=1}^{p}\int_{0}^{\infty}\left(g^{(j,l)}(x)\right)^{2}-g^{(j,l)}(x+\frac{1}{n})g^{(j,l)}(x)dx\right)^{\frac{1}{2}}}
\end{equation}
using the results of Section 7.2 of \cite{Andrea1} and Appendix of \cite{BCD-Creates} we have that the numerator is given by
\begin{equation*}
\sum_{l=1}^{p}K_{1}^{(i,j)}e^{-\lambda^{(j,l)}\frac{k}{n}}\sum_{r=0}^{\infty}\frac{(1+\delta^{(j,l)})_{r}}{(\delta^{(i,l)}+\delta^{(j,l)}+2)_{r}}\frac{1}{r!}\left(\lambda^{(i,l)}+\lambda^{(j,l)} \right)^{r}\Bigg(2\left(\frac{k}{n}\right)^{r+\delta^{(i,l)}+\delta^{(i,l)}+1}
\end{equation*}
\begin{equation*}
-\left(\frac{k-1}{n}\right)^{r+\delta^{(i,l)}+\delta^{(i,l)}+1}e^{\lambda^{(i,l)}\frac{1}{n}}-\left(\frac{k+1}{n}\right)^{r+\delta^{(i,l)}+\delta^{(i,l)}+1}e^{-\lambda^{(i,l)}\frac{1}{n}}\Bigg)
\end{equation*}
\begin{equation*}
+\sum_{l=1}^{p}K_{2}^{(i,j),(l)}e^{-\lambda^{(i,l)}\frac{k}{n}}\sum_{r=0}^{\infty}\frac{(\delta^{(i,l)})_{r}}{(\delta^{(i,l)}+\delta^{(j,l)})_{r}}\frac{1}{r!}\left(\lambda^{(i,l)}+\lambda^{(j,l)}\right)^{r}\Bigg(2\left(\frac{k}{n}\right)^{r}
\end{equation*}
\begin{equation*}
-\left(\frac{k-1}{n}\right)^{r}e^{\lambda^{(i,l)}\frac{1}{n}}-\left(\frac{k+1}{n}\right)^{r}e^{-\lambda^{(i,l)}\frac{1}{n}}\Bigg),
\end{equation*}
\begin{equation*}
\text{where}\quad
K_{1}^{(i,j)}:=\frac{\Gamma\left(\delta^{(j,l)}+1 \right)\Gamma\left(-1-\delta^{(j,l)}-\delta^{(i,l)} \right)}{\Gamma\left(-\delta^{(j,l)}\right)}\quad
\text{and}\quad
K_{2}^{(i,j),(l)}:=\frac{\Gamma\left(\delta^{(j,l)}+\delta^{(i,l)}+1 \right)}{\left(\lambda^{(i,l)}+\lambda^{(j,l)} \right)^{\delta^{(j,l)}+\delta^{(i,l)}+1}}.
\end{equation*}
Assume that $\delta^{(j,l)}=\delta^{(i,k)}$ for any $i,j,l,k=1,...,p$. Then  $\delta:=\delta^{(i,j)}$ and $K_{1}:=K_{1}^{(i,j)}$ for any $i,j=1,...,p$ and the numerator can be rewritten as
\begin{equation*}
K_{1}\left(\frac{1}{n}\right)^{2\delta+1}\sum_{l=1}^{p}e^{-\lambda^{(j,l)}\frac{k}{n}}\sum_{r=0}^{\infty}\frac{(1+\delta)_{r}}{(2\delta+2)_{r}}\frac{1}{r!}\left(\lambda^{(i,l)}+\lambda^{(j,l)} \right)^{r}\left(\frac{1}{n}\right)^{r}\bigg(2k^{r+2\delta+1}
\end{equation*}
\begin{equation*}
-\left(k-1\right)^{r+2\delta+1}e^{\lambda^{(i,l)}\frac{1}{n}}-\left(k+1\right)^{r+2\delta+1}e^{-\lambda^{(i,l)}\frac{1}{n}}\bigg)
\end{equation*}
\begin{equation*}
+\sum_{l=1}^{p}K_{2}^{(i,j),(l)}e^{-\lambda^{(i,l)}\frac{k}{n}}\sum_{r=0}^{\infty}\frac{(\delta)_{r}}{(2\delta)_{r}}\frac{1}{r!}\left(\lambda^{(i,l)}+\lambda^{(j,l)}\right)^{r}\left(\frac{1}{n}\right)^{r}\left(2k^{r}
-\left(k-1\right)^{r}e^{\lambda^{(i,l)}\frac{1}{n}}-\left(k+1\right)^{r}e^{-\lambda^{(i,l)}\frac{1}{n}}\right),
\end{equation*}
and using the results of section 7.2 of \cite{Andrea1} we get
\begin{equation*}
=K_{1}\left(\frac{1}{n}\right)^{2\delta+1}\sum_{l=1}^{p}e^{-\lambda^{(j,l)}\frac{k}{n}}\bar{f}^{(1)}_{i,l}\left(\frac{1}{n}\right)+\sum_{l=1}^{p}K_{2}^{(i,j),(l)}e^{-\lambda^{(i,l)}\frac{k}{n}}\bar{f}^{(2)}_{i,l}\left(\frac{1}{n}\right),
\end{equation*}
\begin{equation*}
\text{where}\quad\bar{f}^{(1)}_{i,l}(n):=\sum_{r=0}^{\infty}\frac{(1+\delta)_{r}}{(2\delta+2)_{r}}\frac{1}{r!}\left(\lambda^{(i,l)}+\lambda^{(j,l)} \right)^{r}\left(\frac{1}{n}\right)^{r}
\end{equation*}
\begin{equation*}
\left(2k^{r+2\delta+1}-\left(k-1\right)^{r+2\delta+1}e^{\lambda^{(i,l)}\frac{1}{n}}-\left(k+1\right)^{r+2\delta+1}e^{-\lambda^{(i,l)}\frac{1}{n}}\right)
\end{equation*}
\begin{equation*}
\text{and}\quad\bar{f}^{(2)}_{i,l}(n):=\sum_{r=0}^{\infty}\frac{(\delta)_{r}}{(2\delta)_{r}}\frac{1}{r!}\left(\lambda^{(i,l)}+\lambda^{(j,l)}\right)^{r}\left(\frac{1}{n}\right)^{r}\left(2k^{r}
-\left(k-1\right)^{r}e^{\lambda^{(i,l)}\frac{1}{n}}-\left(k+1\right)^{r}e^{-\lambda^{(i,l)}\frac{1}{n}}\right).
\end{equation*}
It is possible to notice that $\lim\limits_{n\rightarrow\infty}\bar{f}^{(1)}_{i,l}(n)=2k^{2\delta+1}-\left(k-1\right)^{2\delta+1}-\left(k+1\right)^{2\delta+1}$ and $\lim\limits_{n\rightarrow\infty}\bar{f}^{(1)}_{i,l}(n)=0$.

Moreover, regarding the denominator we observe that
\begin{equation*}
\sum_{l=1}^{p}\int_{0}^{\infty}\left(g^{(i,l)}(x)\right)^{2}-g^{(i,l)}\left(x+\frac{1}{n}\right)g^{(i,l)}(x)dx=\sum_{l=1}^{p}\frac{\Gamma(2\delta+1)}{\left(2\lambda^{(i,l)}\right)^{2\delta+1}}-
\end{equation*}
\begin{equation*}
e^{(\lambda^{(i,l)})\frac{k}{2n}}\pi^{-\frac{1}{2}}\Gamma\left(\delta+1\right)\left(2\lambda^{(i,l)}\right)^{-\delta-\frac{1}{2}}\left(\frac{k}{n}\right)^{\delta+\frac{1}{2}}K_{\delta+\frac{1}{2}}\left(\left(2\lambda^{(i,l)}\right)\frac{k}{n}\right),
\end{equation*}
\begin{equation*}
=\left(\frac{k}{n}\right)^{2\delta+1}\sum_{l=1}^{p}2K_{1}e^{(\lambda^{(i,l)})\frac{k}{n}}f^{(1)}_{i,l}\left(\frac{k}{n}\right)+f^{(5)}_{i,l}\left(\frac{k}{n}\right),
\end{equation*}
using the results of section 7.2 of \cite{Andrea1}, where $f^{(1)}_{i,l}$ is a power series such that $\lim\limits_{n\rightarrow\infty}f^{(1)}_{i,l}\left(\frac{k}{n}\right)=1$ and $f^{(5)}_{i,l}\left(\frac{k}{n}\right)=O\left( \left(\frac{k}{n}\right)^{1-2\delta}\right)$, hence $\lim\limits_{n\rightarrow\infty}f^{(5)}_{i,l}=0$ if and only if $\delta<\frac{1}{2}$. Therefore, we have that $(\ref{example})$ can be written as follows:
\begin{equation*}
\frac{1}{2}\frac{K_{1}\left(\frac{1}{n}\right)^{2\delta+1}\sum_{l=1}^{p}e^{-\lambda^{(i,l)}\frac{k}{n}}\bar{f}^{(1)}_{i,l}\left(\frac{1}{n}\right)+\sum_{l=1}^{p}K_{2}^{(i,j),(l)}e^{-\lambda^{(i,l)}\frac{k}{n}}\bar{f}^{(2)}_{i,l}\left(\frac{1}{n}\right)}{\left(\left(\frac{k}{n}\right)^{2\delta+1}\sum_{l=1}^{p}2K_{1}e^{(\lambda^{(i,l)})\frac{k}{n}}f^{(1)}_{i,l}\left(\frac{k}{n}\right)+f^{(5)}_{i,l}\left(\frac{k}{n}\right)\right)^{\frac{1}{2}}\left(\left(\frac{k}{n}\right)^{2\delta+1}\sum_{l=1}^{p}2K_{1}e^{(\lambda^{(j,l)})\frac{k}{n}}f^{(1)}_{j,l}\left(\frac{k}{n}\right)+f^{(5)}_{j,l}\left(\frac{k}{n}\right)\right)^{\frac{1}{2}}}
\end{equation*}
\begin{equation*}
=\frac{1}{2}\frac{K_{1}\sum_{l=1}^{p}e^{-\lambda^{(i,l)}\frac{k}{n}}\bar{f}^{(1)}_{i,l}\left(\frac{1}{n}\right)+\sum_{l=1}^{p}K_{2}^{(i,j),(l)}e^{-\lambda^{(i,l)}\frac{k}{n}}\bar{f}^{(2)}_{i,l}\left(\frac{1}{n}\right)}{\left(k^{2\delta+1}\sum_{l=1}^{p}2K_{1}e^{(\lambda^{(i,l)})\frac{k}{n}}f^{(1)}_{i,l}\left(\frac{k}{n}\right)+f^{(5)}_{i,l}\left(\frac{k}{n}\right)\right)^{\frac{1}{2}}\left(k^{2\delta+1}\sum_{l=1}^{p}2K_{1}e^{(\lambda^{(j,l)})\frac{k}{n}}f^{(1)}_{j,l}\left(\frac{k}{n}\right)+f^{(5)}_{j,l}\left(\frac{k}{n}\right)\right)^{\frac{1}{2}}}.
\end{equation*}
Now, assuming that $\delta<\frac{1}{2}$, by taking the limit as $n\rightarrow\infty$, we obtain that
\begin{equation}\label{limit}
\Rightarrow\frac{1}{2}\frac{pK_{1}\left(2k^{2\delta+1}-\left(k-1\right)^{2\delta+1}-\left(k+1\right)^{2\delta+1}\right)}{k^{2\delta+1}2pK_{1}}=\frac{2k^{2\delta+1}-\left(k-1\right)^{2\delta+1}-\left(k+1\right)^{2\delta+1}}{4k^{2\delta+1}}
\end{equation}
since $\left(r^{(n)}_{i,j}(k)\right)^{2}$ is a continuous function of $r^{(n)}_{i,j}(k)$ and the limit $(\ref{limit})$ is finite then
\begin{equation*}
\lim\limits_{n\rightarrow\infty}\left(r^{(n)}_{i,j}(k)\right)^{2}=\frac{\left(2k^{2\delta+1}-\left(k-1\right)^{2\delta+1}-\left(k+1\right)^{2\delta+1}\right)^{2}}{16k^{4\delta+2}}.
\end{equation*}
In order to see that Assumption \ref{2} is satisfied we need that $\sum_{k=1}^{\infty}\frac{\left(2k^{2\delta+1}-\left(k-1\right)^{2\delta+1}-\left(k+1\right)^{2\delta+1}\right)^{2}}{16k^{4\delta+2}}<\infty$. Let us see for which values of $\delta$ this holds. Define $x:=2\delta+1$ then we have that using the generalised Binomial theorem
\begin{equation}\label{series}
2k^{x}-\left(k-1\right)^{x}-\left(k+1\right)^{x}=\sum_{s=0}^{\infty}\binom{x}{s}\left[2(k-1)^{s}-k^{s}-(k-2)^{s}\right],
\end{equation}
where
\begin{equation*}
2(k-1)^{s}-k^{s}-(k-2)^{s}=\sum_{n=0}^{s}\binom{s}{n}\left[2k^{s-n}(-1)^{n}-k^{s-n}(-2)^{n} \right]-k^{s}
\end{equation*}
\begin{equation*}
=\sum_{n=1}^{s}\binom{s}{n}k^{s-n}\left[2(-1)^{n}-(-2)^{n} \right]=\sum_{n=2}^{s}\binom{s}{n}k^{s-n}\left[2(-1)^{n}-(-2)^{n} \right]
\end{equation*}
\begin{equation*}
\leq k^{s-2}\sum_{n=2}^{s}\binom{s}{n}\left[2+2^{n} \right]\leq k^{s-2}\sum_{n=0}^{s}\binom{s}{n}\left[2+2^{n} \right]=k^{s-2}\left(2^{s+1}+3^{s} \right)\leq 2 k^{s-2}4^{s}.
\end{equation*}
Hence, when we plug this into $(\ref{series})$ we get
\begin{equation*}
\sum_{s=0}^{\infty}\binom{x}{s}\left[2(k-1)^{s}-k^{s}-(k-2)^{s}\right]\leq \sum_{s=0}^{\infty}\binom{x}{s}2 k^{s-2}4^{s}=2k^{-2}\sum_{s=0}^{\infty}\binom{x}{s}(4k)^{s}=2k^{-2}(4k+1)^{x}
\end{equation*}
Therefore, we have that
\begin{equation*}
\sum_{k=1}^{\infty}\frac{\left(2k^{x}-\left(k-1\right)^{x}-\left(k+1\right)^{x}\right)^{2}}{16k^{2x}}=\frac{1}{4}\sum_{k=1}^{\infty}\frac{(4k+1)^{2x}}{k^{2x+4}}
\end{equation*}
which converges for any $x\in\mathbb{R}$ and hence for any $\delta\in\mathbb{R}$. However, recall that in our computations we have assumed that $\delta<\frac{1}{2}$.

Concerning Assumption \ref{1}, we have the following. Let $\epsilon_{n}=n^{-\kappa}$, $\kappa\in(0,1)$, then
\begin{equation*}
\pi_{n}^{(m,l)}((n^{-\kappa},\infty))=\frac{\int_{n^{-\kappa}}^{\infty}\left(g^{(m,l)}(s+\Delta_{n})-g^{(m,l)}(s)\right)^{2}ds}{\int_{0}^{\infty}\left(g^{(m,l)}(s+\Delta_{n})-g^{(m,l)}(s)\right)^{2}ds}.
\end{equation*}
Given the fact that our $\pi_{n}^{(m,l)}$ has the same structure as $\pi_{n}$ in \cite{BCD}, then the same arguments used in \cite{BCD} hold here and we can conclude that if $\delta\in\left(-\frac{1}{2},\frac{1}{2}\right)$ then Assumption \ref{1} is satisfied.

Combining the ranges obtained, we conclude that when $\delta\in\left(-\frac{1}{2},\frac{1}{2}\right)$ then all the results presented in this work apply to our example.
\section{Conclusion}\label{CONCLUSION}
In this paper we introduced the multivariate BSS process and studied the joint asymptotic behaviour of its realised covariation, presenting limit theorems, feasible results and an explicit example. We also provided central limit theorems and weak laws of large numbers for general stationary multivariate Gaussian processes. There are at least two directions which will be worth exploring in more detail in the future:

First, is it possible to find feasible estimates for the asymptotic variance of the multivariate BSS processes? \textit{I.e}.~can ``second order" feasible results be obtained in addition to the ``first order" results we already presented?

Second, we considered the asymptotic theory for BSS processes outside the semimartingale setting. In doing so, we concentrate on a particular scenario (as described by the assumptions on the deterministic function $g$ in Assumption \ref{1}). However, one can imagine other scenarios which lead to BSS processes (or other volatility modulated Gaussian processes) beyond the semimartingale framework. Can similar asymptotic results for the (scaled) realised covariation be obtained in such settings?
\section*{Appendix: the matrices D and V for the BSS process}\label{Remark1}
In this appendix we are going to specify the explicit structure and value of the matrices $\textbf{D}$ and $\textbf{V}_{s}$. The reason why we put them into the appendix is that in order to present them we need some combinatorial arguments which are easy but tedious, and they are similar for the different cases presented in Chapter \ref{CLT-BSS-chapter}.
\subsection*{Case I}
Let $\textit{\textbf{D}}^{1/2}\in\mathcal{M}^{p^{6}\times p^{6}}(\mathbb{R})$ be defined as 
\begin{equation*}
(\textit{\textbf{D}})_{z,y}:=\lim\limits_{n\rightarrow\infty}\frac{2}{n}\sum_{h=1}^{n-1}(n-h)
\end{equation*}
\begin{equation*}
\left(r_{k_{z},r_{z},m_{z};k_{y},r_{y},m_{y}}^{(n)}(h)r_{l_{z},q_{z},w_{z};l_{y},q_{y},w_{y}}^{(n)}(h)+r_{l_{z},q_{z},w_{z};k_{y},r_{y},m_{y}}^{(n)}(h)r_{k_{z},r_{z},m_{z};l_{y},q_{y},w_{y}}^{(n)}(h)\right)
\end{equation*}
\begin{equation*}
+\left(r_{k_{z},r_{z},m_{z};k_{y},r_{y},m_{y}}^{(n)}(0)r_{l_{z},q_{z},w_{z};l_{y},q_{y},w_{y}}^{(n)}(0)+r_{l_{z},q_{z},w_{z};k_{y},r_{y},m_{y}}^{(n)}(0)r_{k_{z},r_{z},m_{z};l_{y},q_{y},w_{y}}^{(n)}(0)\right),
\end{equation*}
where for each of the $p^{6}\times p^{6}$ combinations of $(z,y)$ there is a unique combination of $((r_{z},m_{z},q_{z},w_{z},k_{z},l_{z})$, $(r_{y},m_{y},q_{y},w_{y},k_{y},l_{y}))$ where each of these elements takes value in $\{1,...,p\}$. Let $\nu(r,m,q,w)$ be the set of all the possible combinations of $r,m,q,w\in\{1,...,p\}$ and let $\nu_{s}(r,m,q,w)$ determines the $s$ element of $\nu(r,m,q,w)$. Notice that $\nu(r,m,q,w)$ has certain order for its element, which is not relevant for us since we only care about the consistent use of the order adopted. Then the association is given by
\begin{equation*}
(z,y)\leftrightarrow\Bigg(\Big(\nu_{z-\lfloor\frac{z-1}{p^{4}}\rfloor p^{4}}(r,m,q,w),\lfloor\frac{\lfloor\frac{z-1}{p^{4}}\rfloor}{p}\rfloor+1,\lfloor\frac{z-1}{p^{4}}\rfloor+1-p\lfloor\frac{\lfloor\frac{z-1}{p^{4}}\rfloor}{p}\rfloor\Big),
\end{equation*}
\begin{equation*}
\Big(\nu_{y-\lfloor\frac{y-1}{p^{4}}\rfloor p^{4}}(r,m,q,w),\lfloor\frac{\lfloor\frac{y-1}{p^{4}}\rfloor}{p}\rfloor+1,\lfloor\frac{y-1}{p^{4}}\rfloor+1-p\lfloor\frac{\lfloor\frac{y-1}{p^{4}}\rfloor}{p}\rfloor\Big)\Bigg).
\end{equation*}
For a proof of this statement for the case $p=2$ see the proof of Theorem \ref{I}; the extension to the case $p>2$ is trivial. Moreover, define for $s\in[0,T]$ the $p^{2}\times p^{6}$ matrix
\begin{equation}\label{V}
V_{s}:=\begin{pmatrix}
   \sigma_{s} & \textbf{0} &\cdots  & \textbf{0} \\ \textbf{0} &
   \ddots &  & \vdots \\ \vdots &  & \ddots & \textbf{0}\\ \textbf{0} & \cdots& \textbf{0} & \sigma_{s}    
  \end{pmatrix},
\end{equation}
where $\sigma_{s}:=\left(\sigma^{(r,m)}_{s}\sigma^{(q,w)}_{s}\right)_{(r,m,q,w)\in\nu(r,m,q,w)}^{\top}$, so it is a row vector of $p^{4}$ elements (here the consistency of the order of the elements of $\nu(r,m,q,w)$ is fundamental), and $\textbf{0}$ is a row vector of $p^{4}$ elements containing only zeros. Hence, $\sigma_{s}$ and $\mathbf{0}$ contain the same number of elements.

In case of the \textit{vech} notation, the association for $\textit{\textbf{D}}$ is the following. First, let us define for $i\in\mathbb{N}$
\begin{equation*}
\chi(i):=\lfloor\sqrt{2i}+\frac{1}{2} \rfloor ,\quad\xi(i):=i-\frac{1}{2}\lfloor\frac{\sqrt{8i-7}-1}{2} \rfloor\left(\lfloor\frac{\sqrt{8i-7}-1}{2} \rfloor+1\right),
\end{equation*}
then we have
\begin{equation*}
(z,y)\leftrightarrow\Bigg(\Big(\nu_{z-\lfloor\frac{z-1}{p^{4}}\rfloor p^{4}}(r,m,q,w),\chi\left(\lfloor\frac{z-1}{p^{4}}\rfloor\right),\xi\left(\lfloor\frac{z-1}{p^{4}}\rfloor\right)\Big),
\end{equation*}
\begin{equation}\label{association}
\Big(\nu_{y-\lfloor\frac{y-1}{p^{4}}\rfloor p^{4}}(r,m,q,w),\chi\left(\lfloor\frac{y-1}{p^{4}}\rfloor\right),\xi\left(\lfloor\frac{y-1}{p^{4}}\rfloor\right)\Big)\Bigg)
\end{equation}
The association $(\ref{association})$ comes from the fact that we have $k=1,...,p$ with $k\leq l$. The couple $(k,l)$ has the following sequence $(1,1),(2,1),(2,2),(3,1),(3,2),(3,3),(4,1),(4,2),(4,3),(4,4),(5,1),...$ which is a fractal sequence of form given by $(\xi(i),\chi(i))$ where $i$ is the $i$-th term of the sequence. Concerning the matrix $V_{s}$, we have the same structure as (\ref{V}). However, the dimension is now $\frac{p}{2}(p+1)\times \frac{p^{5}}{2}(p+1)$. The value of the $\sigma_{s}$ and of the $\mathbf{0}$ is the same as before and, hence, we are considering fewer of them (\textit{e.g}.~there are $\frac{p}{2}(p+1)$ number of $\sigma_{s}$ while in (\ref{V}) we had $p^{2}$ number of them).
\subsection*{Case II first scenario}
We have the same as in Case I. The only difference is the use of $\bar{r}^{(n)}$ instead of $r^{(n)}$.
\subsection*{Case II second scenario}
For this case we have something similar to the previous sections but simpler since there are not any more the variables $r$ and $q$. Let $\textit{\textbf{D}}^{1/2}\in\mathcal{M}^{p^{4}\times p^{4}}(\mathbb{R})$ be defined as
\begin{equation*}
(\textit{\textbf{D}})_{z,y}:=\lim\limits_{n\rightarrow\infty}\frac{2}{n}\sum_{h=1}^{n-1}(n-h)\left(\tilde{r}_{k_{z},m_{z};k_{y},m_{y}}^{(n)}(h)\tilde{r}_{l_{z},w_{z};l_{y},w_{y}}^{(n)}(h)+\tilde{r}_{l_{z},w_{z};k_{y},m_{y}}^{(n)}(h)\tilde{r}_{k_{z},m_{z};l_{y},w_{y}}^{(n)}(h)\right)
\end{equation*}
\begin{equation*}
+\left(\tilde{r}_{k_{z},m_{z};k_{y},m_{y}}^{(n)}(0)\tilde{r}_{l_{z},w_{z};l_{y},w_{y}}^{(n)}(0)+\tilde{r}_{l_{z},w_{z};k_{y},m_{y}}^{(n)}(0)\tilde{r}_{k_{z},m_{z};l_{y},w_{y}}^{(n)}(0)\right),
\end{equation*}
where for each of the $p^{4}\times p^{4}$ combinations of $(z,y)$ there is a unique combination of\\ $((m_{z},w_{z},k_{z},l_{z}),(m_{y},w_{y},k_{y},l_{y}))$. The association is given by
\begin{equation*}
(z,y)\leftrightarrow\Bigg(\Big(\mu_{z-\lfloor\frac{z-1}{p^{2}}\rfloor p^{2}}(m,w),\lfloor\frac{\lfloor\frac{z-1}{p^{2}}\rfloor}{p}\rfloor+1,\lfloor\frac{z-1}{p^{2}}\rfloor+1-p\lfloor\frac{\lfloor\frac{z-1}{p^{2}}\rfloor}{p}\rfloor\Big),
\end{equation*}
\begin{equation*}
\Big(\mu_{y-\lfloor\frac{y-1}{p^{2}}\rfloor p^{2}}(m,w),\lfloor\frac{\lfloor\frac{y-1}{p^{2}}\rfloor}{p}\rfloor+1,\lfloor\frac{y-1}{p^{2}}\rfloor+1-p\lfloor\frac{\lfloor\frac{y-1}{p^{2}}\rfloor}{p}\rfloor\Big)\Bigg).
\end{equation*}
where $\mu(m,w)$ is the set of all the possible combinations of $m,w\in\{1,...,p\}$ and $\mu_{s}(m,w)$ is the $s$ element of $\mu(m,w)$. Moreover, define for $s\in[0,T]$ the $p^{2}\times p^{4}$ matrix
\begin{equation}\label{lastV}
V_{s}:=\begin{pmatrix}
   \sigma_{(1,1),s} & \textbf{0} &\cdots  & \textbf{0} \\ \textbf{0} &
   \ddots &  & \vdots \\ \vdots &  & \ddots & \textbf{0}\\ \textbf{0} & \cdots& \textbf{0} & \sigma_{(p,p),s}

  \end{pmatrix},
\end{equation}
where $\sigma_{(k,l),s}:=\left(\sigma^{(k,m)}_{s}\sigma^{(l,w)}_{s}\right)_{m,w\in\mu(m,w)}^{\top}$ and $\textbf{0}$ is a row vector of $p^{2}$ elements containing only zeros.

For the \textit{vech} notation we have that the structure of the matrices $\textit{\textbf{D}}$ and $V$ remain the same but their dimension reduce. Similarly to the previous cases, the association is given by
\begin{equation*}
(z,y)\leftrightarrow\Bigg(\Big(\mu_{z-\lfloor\frac{z-1}{p^{2}}\rfloor p^{2}}(m,w),\chi\left(\lfloor\frac{z-1}{p^{4}}\rfloor\right),\xi\left(\lfloor\frac{z-1}{p^{4}}\rfloor\right)\Big),
\end{equation*}
\begin{equation*}
\Big(\mu_{y-\lfloor\frac{y-1}{p^{2}}\rfloor p^{2}}(m,w),\chi\left(\lfloor\frac{y-1}{p^{4}}\rfloor\right),\xi\left(\lfloor\frac{y-1}{p^{4}}\rfloor\right)\Big)\Bigg).
\end{equation*}
Concerning the matrix $V_{s}$, we have the same structure as (\ref{lastV}). However, the dimension is now $\frac{p}{2}(p+1)\times \frac{p^{3}}{2}(p+1)$. The value of the $\sigma_{s}$ and of the $\mathbf{0}$ is the same as before and, hence, we are considering fewer of them (\textit{e.g}.~there are $\frac{p}{2}(p+1)$ number of $\sigma_{s}$ while in (\ref{lastV}) we had $p^{2}$ number of them).
\section*{Acknowledgement}
The authors would like to thank Andrea Granelli, Mikko Pakkanen and Mark Podolskij for useful discussions and comments. RP acknowledges the financial support of the CDT in MPE and the Grantham Institute.


\begin{thebibliography}{0}\small
\bibitem{Aldo}Aldous D. J. and Eagleson G. K. (1978), On mixing and stability of limit theorems, \textit{The Annals of Probability} 6, 325-331.
\bibitem{BBV}Barndorff-Nielsen O. E., Benth F. E. and Veraart A. E. D.  (2013). Modelling energy spot prices by volatility modulated L\'{e}vy-driven Volterra processes. \textit{Bernoulli} 19 (3), 803-845.
\bibitem{BCD} Barndorff-Nielsen O. E., Corcuera J.M. and Podolskij M. (2011), Multipower variation for Brownian semistationary processes. \textit{Bernoulli} 17(4), 1159-1194.
\bibitem{BCD-Creates}Barndorff-Nielsen, O.E., Corcuera, J.M. and Podolskij, M. (2009). Multipower variation for Brownian semistationary processes (full version). CREATES research paper 2009-21, Aarhus Univ.
\bibitem{BCD-new}Barndorff-Nielsen, O. E., Corcuera, J. M. and Podolskij, M. (2013), Limit theorems for functionals of higher order differences of brownian semi-stationary processes, \textit{Prokhorov and contemporary probability theory, Springer}, pp. 69-96.
\bibitem{Pakkanen}Barndorff-Nielsen, O. E.; Pakkanen, M. S.; Schmiegel, J. Assessing relative volatility/ intermittency/energy dissipation. \textit{Electron. J. Statist.} 8 (2014), no. 2, 1996-2021.
\bibitem{BSS}Barndorff-Nielsen O.E. , Schmiegel J., Brownian semistationary processes and volatility/intermittency, in: \textit{H. Albrecher, W. Runggaldier, W. Schachermayer (Eds.), Advanced Financial Modelling, Walter de Gruyter,} 2009, pp. 1–26.
\bibitem{Econometrica} Barndorff-Nielsen, O. E. and Shephard, N. (2004), Econometric Analysis of Realized Covariation: High Frequency Based Covariance, Regression, and Correlation in Financial Economics. \textit{Econometrica}, 72: 885-925. doi:10.1111/j.1468-0262.2004.00515.
\bibitem{last}Bennedsen M, (2017). A rough multi-factor model of electricity spot prices. \textit{Energ. Econ.} 63, 301-313.
\bibitem{scheme} Bennedsen M., Lunde A. and Pakkanen M. (2017), Hybrid scheme for Brownian semistationary processes
\textit{Finance and Stochastics} 21(4), 931–965, 2017. article: doi:10.1007/s00780-017-0335-5.
\bibitem{logarithmicvolatilityPakkanen}Bennedsen M., Lunde A. and Pakkanen M. (2017), Decoupling the short- and long-term behavior of stochastic volatility. July 2017, 46 pages (revised version), e-print: arXiv:1610.00332.
\bibitem{Cor} Corcuera, J. M. (2012). New Central Limit Theorems for Functionals of Gaussian Processes
and their Applications. In: \textit{Methodology and Computing in Applied Probability} 14.3, pp. 477-500
\bibitem{CHPP}Corcuera J. M., Hedevang E., Podolskij M., Pakkanen M. (2013). Asymptotic theory for Brownian semi-stationary processes with application to turbulence. \textit{Stochastic Processes and their Applications} 123(7), 2552-2574
\bibitem{PhD} Granelli A., Limit theorems and stochastic models for
dependence and contagion in financial
markets. \textit{PhD thesis} (2017), Imperial College London.
\bibitem{Andrea1} Granelli A. and Veraart A.E.D., Central limit theorem for the realised covariation of a Bivariate semistationary process. \textit{ArXiv}:1707.08507.
\bibitem{Jac} Jacod J. (1997). On continuous conditional Gaussian martingales and stable convergence in law. \textit{S\'{e}minaire de Probabilit\'{e}s} XXXI, 232-246.
\bibitem{Jacod}Jacod J. (2008), Asymptotic properties of realized power variations and related functionals of semimartingales. \textit{Stochastic Process and  Applications}, 118, pp. 517-559
\bibitem{NP}Nourdin, I. and Peccati G., Normal approximations with Malliavin calculus: from Stein's method to universality. Vol. 192. \textit{Cambridge University Press}, 2012.
\bibitem{Survey}Podolskij M. and Vetter M., Understanding limit theorems for
semimartingales: a short survey. \textit{Statistica Neerlandica} (2010) Vol. 64, nr. 3, pp. 329–351
\bibitem{ReedBook}Reed, M. and Simon, B. (1975), Methods of modern mathematical physics, Vol. 2, \textit{Academic press}.
\bibitem{V}Van der Vaart, A. W. (1998). Asymptotic statistics. \textit{New York: Cambridge University Press}.
\bibitem{VW}Van der Vaart, A. W. and Wellner J. A. (1996). Weak Convergence and empirical processes: with applications to statistics. \textit{Springer Series in Statistics}.
\end{thebibliography}
\end{document}